\documentclass[a4paper,fleqn]{cas-sc}

\usepackage[authoryear,longnamesfirst]{natbib}
\usepackage{algorithm}
\usepackage{algpseudocode}
\usepackage{placeins}

\usepackage[utf8]{inputenc}
\usepackage{tikz}
\usetikzlibrary{positioning}
\def\tsc#1{\csdef{#1}{\textsc{\lowercase{#1}}\xspace}}
\tsc{WGM}
\tsc{QE}

\newtheorem{expl}{Example}[section]

\newtheorem{theorem}{Theorem}[section]
\newtheorem{lemma}[theorem]{Lemma}
\newtheorem{corollary}[theorem]{Corollary}

\newtheorem{rem}[theorem]{Remark}

 \numberwithin{figure}{section}

\DeclareMathOperator{\Tr}{Tr}

\begin{document}
\let\WriteBookmarks\relax
\def\floatpagepagefraction{1}
\def\textpagefraction{.001}

\shorttitle{Structure-preserving full- and low-rank exponential 
methods for the optimal control of Lindblad equations}    

\shortauthors{H. Chen et~al.}  

\title [mode = title]{Structure-preserving full- and low-rank exponential 
methods for the optimal control of Lindblad equations}



%

\author[1]{Hao Chen}[orcid=0000-0001-9976-6732]

\cormark[1]


\ead{hch@cqnu.edu.cn}


\credit{Conceptualization, Investigation, Methodology, Project administration, Software, Validation, Writing-original draft, Writing-review \& editing}

\affiliation[1]{organization={College of Mathematics Science, Chongqing Normal University},
            addressline={University Town}, 
            city={Chongqing},
            postcode={401331}, 
            country={China}}

\author[2]{Alfio Borz\`{i}}[orcid=0000-0002-8050-1336]


\ead{alfio.borzi@uni-wuerzburg.de}


\credit{Conceptualization, Investigation, Methodology, Project administration, Software, Validation, Writing-original draft, Writing-review \& editing}

\affiliation[2]{organization={Institut f\"{u}r Mathematik, Universit\"{a}t W\"{u}rzburg},
            addressline={Campus Hubland Nord}, 
            city={W\"{u}rzburg},
            postcode={97074}, 
            country={Germany}}

\cortext[1]{Corresponding author}



\begin{abstract}
Optimal control of open quantum systems governed by the Lindblad master
equation requires the repeated solution of forward and adjoint evolution
equations. For large Hilbert spaces, these computations demand numerical
methods that are structure-preserving, accurate, and computationally
efficient, while remaining compatible with nonsmooth optimization
techniques. In this work, a unified framework for optimal control of
Lindblad equations is presented. The continuous optimality system is
derived from the Pontryagin maximum principle and solved by a sequential
quadratic Hamiltonian (SQH) method. Its numerical realization is based
on second-order exponential midpoint propagators for the forward and
adjoint Lindblad equations with time-dependent Hamiltonians. Full-rank
schemes preserving the Hermitian and positive-semidefinite structure of
the density matrix are developed together with low-rank formulations
that substantially reduce storage requirements and computational cost.
Rigorous error estimates are established for the full- and low-rank
forward and adjoint propagators. Numerical experiments confirm the
predicted convergence rates, demonstrate the effectiveness of the
low-rank approximations, and illustrate the performance of the proposed
FREM-SQH and LREM-SQH algorithms for optimal control problems with
smooth and nonsmooth control costs.
\end{abstract}


\maketitle

\section{Introduction}
\label{sec-introduction}

The Lindblad master equation, also referred to as the
Gorini-Kossakowski-Sudarshan-Lindblad equation
\citep{Gorini,Lindblad}, provides the standard mathematical model for
the Markovian evolution of open quantum systems
\citep{Breuer,Davies1976}. It plays a fundamental role in quantum
optics, quantum information processing, quantum technologies, and the
description of decoherence and noise in realistic quantum devices. The
exact Lindblad evolution maps density matrices to density matrices,
preserving Hermiticity, positive semidefiniteness, and unit trace.
Maintaining these structural properties at the discrete level is
essential for obtaining physically meaningful numerical simulations
\citep{Riesch}.

In many applications, the objective is to steer an open quantum system
towards a prescribed target state or quantum operation despite the
presence of dissipative effects. This leads to optimal control problems
constrained by the Lindblad equation, which have attracted increasing
attention in recent years; see, for example,
\citep{Goerz,Machnes,Wenin}. Important applications arise in NMR
spectroscopy \citep{Fouquieres2011,Herbruggen2005,Khaneja2001,TOSNER2009},
quantum chemistry
\citep{Khaneja2005,Maday2003,Reich2012,Zhu1998},
quantum information processing
\citep{Doria2011,Egger2014},
and molecular physics \citep{Palao2002}. For general accounts of
quantum optimal control and its computational aspects, we refer to
\citep{Alessandro2008,Wiseman2009,Borzi2017}.

Regardless of the particular optimization strategy, the numerical
solution of an open quantum optimal control problem requires the
repeated integration of a forward Lindblad equation together with a
backward adjoint equation associated with the Pontryagin optimality
system. Consequently, the overall efficiency of the optimization
process is largely determined by the quality of the underlying
propagators. These propagators should preserve the physical structure
of the density matrix, achieve high accuracy for time-dependent
Hamiltonians, and remain computationally efficient for large Hilbert
spaces, where the repeated propagation and storage of full density
matrices become the dominant computational cost. This challenge
becomes particularly significant in modern quantum technologies,
including quantum computing, quantum communication, quantum sensing,
and quantum simulation.

These applications require the repeated solution of optimal control problems
for systems composed of many interacting quantum subsystems. For a
register of $n_q$ identical $d$-level qudits, the Hilbert space has
dimension $m=d^{n_q}$, whereas the density matrix and the adjoint
variable each contain $m^2=d^{2n_q}$ real degrees of freedom.
Consequently, the dominant matrix operations scale as
$\mathcal{O}(m^3)=\mathcal{O}(d^{3n_q})$. For example, a register of
only ten qutrits already requires dense matrices of dimension
$59\,049\times59\,049$, occupying approximately $52\,\mathrm{GiB}$ each
in double-precision complex arithmetic. This exponential growth
severely limits the applicability of conventional full-rank optimal
control algorithms and provides a strong motivation for the
structure-preserving low-rank formulations developed in the present
work.

Structure-preserving numerical methods play a central role in the
simulation and optimal control of open quantum systems. A variety of
integrators has been developed for the Lindblad equation, including
matrix exponential methods \citep{Riesch,Riesch1}, Strang splitting
schemes \citep{Songolo}, Kraus representation methods \citep{Cao},
integrating-factor Runge-Kutta methods
\citep{Appelo1,DelMastro}, exponential Euler integrators
\citep{Chen}, quantum channel methods \citep{Robin}, and nested Picard
iterative schemes \citep{Hu}. Most of these methods are designed for
time-independent Hamiltonians or provide only first-order accuracy in
the time-dependent setting. 
Moreover, rigorous numerical analysis of structure-preserving low-rank
integrators remains very limited. To the best of our knowledge,
existing convergence theory is restricted to the first-order
exponential Euler method developed in \citep{Chen}. Corresponding
second-order error estimates and a rigorous convergence analysis for
low-rank adjoint Lindblad propagators appear to be unavailable.

The numerical solution of open quantum optimal control problems is
traditionally based on adjoint-state techniques such as GRAPE
\citep{Boutin,Schulte}, Krotov methods \citep{Krotov,Pereira}, and
second-order approaches including semismooth Newton methods
\citep{Ciaramella2015,CiaramellaSKRYN} and sequential quadratic
Hamiltonian (SQH) methods
\citep{BorziSQH,Breitenbach2020}. Most existing approaches follow 
a discretize-then-optimize strategy,
in which the continuous dynamics are first discretized and the
resulting finite-dimensional optimization problem is solved
iteratively. By contrast, the present work adopts an
optimize-then-discretize formulation based on the continuous
Pontryagin maximum principle, allowing the numerical discretization to
preserve the intrinsic matrix structure of both the state and adjoint
equations.

A natural strategy for reducing computational costs is to exploit low-rank
representations of the density matrix, in which only factors of size
$m\times r$, with $r\ll m$, are propagated and stored. 
From the viewpoint of numerical analysis, the present framework also
belongs to the comparatively small class of optimal control methods
developed directly for matrix differential equations, rather than for
their equivalent vectorized formulations. Working in matrix form
preserves the underlying algebraic structure throughout the analysis
and naturally supports structure-preserving full- and low-rank
discretizations. Although
low-rank methods for Lindblad equations have been proposed
\citep{Appelo1,Chen,DelMastro,Hu,LeBris1,LeBris2}, their integration
into adjoint-based optimal control algorithms has received little
attention. In particular, rigorous error analysis for low-rank adjoint
propagators and their systematic incorporation into a continuous
optimize-then-discretize framework appear to be unavailable.

In the present work, a unified computational framework for the optimal
control of open quantum systems governed by the Lindblad equation is
developed. First, the continuous optimality system is derived from the
Pontryagin maximum principle and solved by a sequential quadratic
Hamiltonian (SQH) method. Second, second-order
structure-preserving exponential midpoint propagators are developed
for both the forward and adjoint Lindblad equations in full-rank and
low-rank form. Furthermore, rigorous error estimates are established for
the resulting propagators, including the low-rank adjoint
approximation. The optimize-then-discretize formulation naturally
combines structure-preserving time integration, low-rank matrix
approximation, and continuous optimality conditions into a unified
framework for large-scale open quantum optimal control. 

The proposed full-rank propagators preserve the Hermitian and
positive-semidefinite structure of the density matrix, while
corresponding low-rank formulations substantially reduce storage
requirements and computational cost. Rigorous error estimates are
established for the full- and low-rank forward and adjoint propagators.
For the optimal control problem, the SQH framework is extended to
Lindblad dynamics, allowing smooth, nonsmooth and also nonconvex 
control costs to be treated within a unified setting.
The resulting algorithms combine structure-preserving time integration,
low-rank approximation, and continuous optimality conditions into a
single computational framework for open quantum optimal control.

The remainder of the paper is organized as follows.
Section~\ref{sec:control-problem} formulates the controlled Lindblad
equation and the corresponding optimal control problem, derives the
Pontryagin optimality system, and introduces the SQH method.
Section~\ref{sec-ExpInt} develops the full- and low-rank exponential
midpoint propagators for the forward and adjoint Lindblad equations.
Sections~\ref{sec-err1} and \ref{sec-err2} establish the corresponding
error analysis. Numerical experiments are presented in
Section~\ref{sec-experiments}, where the theoretical convergence
results are validated and the performance of the proposed
FREM-SQH and LREM-SQH algorithms is demonstrated. Finally,
concluding remarks are given in the last section.

\section{Controlled Lindblad dynamics and optimal control}
\label{sec:control-problem}

The Lindblad master equation, also known as the
Gorini-Kossakowski-Sudarshan-Lindblad equation
\citep{Gorini,Lindblad}, provides the standard Markovian description of
the dynamics of open quantum systems; see, for example,
\citep{Breuer,Davies1976}. Its exact evolution maps density matrices to
density matrices and preserves Hermiticity, positive semidefiniteness,
and unit trace. In the present work, the Lindblad dynamics is used as
the state equation of an optimal control problem for an open quantum
system.

Let
\[
\mathbb{H}_m
:=
\left\{
X\in\mathbb{C}^{m\times m}:X=X^\dagger
\right\}
\]
denote the real vector space of Hermitian matrices. We equip
\(\mathbb{H}_m\) with the real Hilbert-Schmidt inner product
\begin{equation}
\label{eq:HS-pairing}
\langle X,Y\rangle_{\mathbb{H}_m}
:=
\operatorname{Re}\operatorname{Tr}(X^\dagger Y)
=
\operatorname{Tr}(XY),
\qquad
X,Y\in\mathbb{H}_m,
\end{equation}
and the associated Frobenius norm 
$\|X\|_F=\langle X,X\rangle_{\mathbb{H}_m}^{1/2}$. 

The state of the open quantum system is described by a density matrix
\[
\rho(t)\in\mathcal{D}_m
:=
\left\{
\rho\in\mathbb{H}_m:
\rho\succeq0,\ 
\operatorname{Tr}(\rho)=1
\right\}.
\]

Optimal control problems for open quantum systems and their numerical
solution have been studied using, among other approaches, open-system
versions of GRAPE, Krotov-type methods, and monotonic optimization
schemes; see, e.g.,
\citep{Goerz,Krotov,Machnes,Pereira,Schulte,Wenin}.
For a general account of quantum optimal control and its computational
formulation, we refer to \citep{Borzi2017}.

We consider coherent control through the Hermitian Hamiltonian
\begin{equation}
\label{eq:controlled-Hamiltonian}
H(u(t))
=
H_0+u(t)H_1,
\end{equation}
where \(H_0,H_1\in\mathbb{H}_m\), and \(u\) is a real-valued control.

We remark that, although the presentation is restricted to a single real-valued
control, the formulation extends directly to several coherent controls,
\[
H(u(t))
=
H_0+\sum_{j=1}^{J}u_j(t)H_j,
\]
with \(u(t)\in\mathbb{R}^J\). In this case, the scalar pointwise
minimization appearing below is replaced by a finite-dimensional
minimization over the admissible control set. We retain the
single-control formulation in order to simplify the notation.

A second control mechanism consists in modifying the interaction with
the environment, for example through time-dependent or
control-dependent decay rates \(\gamma_k\); see
\citep{Pechen2006}. Such incoherent controls can also be accommodated by
the exponential midpoint discretization developed below. The
Pontryagin and SQH analysis in the present work is restricted to
coherent Hamiltonian control in order to keep the control dependence
affine and the presentation focused.

The controlled Lindblad initial-value problem is given by 
\begin{equation}
\label{LindbladEq}
\begin{aligned}
\dot{\rho}(t) = \mathcal{L}_{u(t)}(\rho(t)), \qquad 
\rho(0)=\rho_0\in\mathcal{D}_m,
\end{aligned}
\end{equation}
where $t\in(0,T)$, and 
\begin{equation}
\label{eq:controlled-Lindblad-generator}
\mathcal{L}_{v}(\rho)
=
-\mathrm{i}[H_0+vH_1,\rho]
+
\sum_{k=1}^{K}
\gamma_k
\left(
L_k\rho L_k^\dagger
-
\frac12
\left\{
L_k^\dagger L_k,\rho
\right\}
\right) ,
\end{equation}
where \(L_k\in\mathbb{C}^{m\times m}\) are the Lindblad or jump operators 
characterizing the dissipation channels, and \(\gamma_k\geq0\) are the corresponding decay rates 
for each of the $K$ channels. We use the commutators, $[A,B]:=AB-BA$ and $\{A,B\}=AB+BA$, 
and \(\mathrm{i}\) denotes the imaginary unit.

Let \(U\subset\mathbb{R}\) be a nonempty compact set of admissible
control values. The admissible control set is defined by
\begin{equation}
\label{eq:admissible-controls}
\mathcal{U}_{\mathrm{ad}}
=
\left\{
u\in L^\infty(0,T):
u(t)\in U
\text{ for almost every }t\in(0,T)
\right\}.
\end{equation}
For \(u\in\mathcal{U}_{\mathrm{ad}}\), let \(\rho_u\) denote the
solution of \eqref{LindbladEq}. We consider the optimal
control problem
\begin{equation}
\label{eq:optimal-control-problem}
\min_{u\in\mathcal{U}_{\mathrm{ad}}}
J(u),
\end{equation}
with reduced objective functional given by
\begin{equation}
\label{eq:objective-functional}
J(u) =\Theta(\rho_u(T)) + \int_0^T \ell(u(t))\,dt ,
\end{equation}
where the terminal functional $\Theta:\mathbb{H}_m\to\mathbb{R}$ 
is assumed to be continuously differentiable. The control cost $\ell:U\to\mathbb{R}$ 
is assumed to be bounded from below and lower semicontinuous, but it
needs not be differentiable or convex. This setting includes smooth
quadratic regularization as well as sparsity-promoting, and also nonconvex control costs.

A useful decomposition is as follows: 
\begin{equation}
\label{eq:control-cost-decomposition}
\ell(v)
=
\frac{\alpha}{2}|v|^2+\beta\,\varphi(v),
\end{equation}
where \(\alpha,\beta\geq0\) and \(\varphi\) may be nonsmooth or
nonconvex. Typical examples include
\[
\varphi(v)=|v|,
\qquad
\varphi(v)=|v|^p,\quad 0<p<1,
\qquad
\varphi(v)=
\begin{cases}
0,&v=0,\\
1,&v\neq0,
\end{cases}
\]
corresponding respectively to \(L^1\)-, nonconvex \(L^p\)-, and
\(L^0\)-type control costs. Sequential quadratic Hamiltonian methods
for nonsmooth and sparsity-promoting quantum control were developed in
\citep{BorziSQH,Breitenbach2020}; the present work extends this approach
to open quantum systems governed by Lindblad dynamics.

Two terminal objectives of particular interest are the observable
functional
\begin{equation}
\label{eq:observable-terminal-cost}
\Theta(\rho)
=
-\operatorname{Tr}(Q\rho),
\qquad
Q\in\mathbb{H}_m,
\end{equation}
and the quadratic tracking functional
\begin{equation}
\label{eq:tracking-terminal-cost}
\Theta(\rho)
=
\frac12\|\rho-\rho_T\|_F^2,
\qquad
\rho_T\in\mathcal{D}_m.
\end{equation}
Their derivatives with respect to the pairing \eqref{eq:HS-pairing} are given by 
$\nabla\Theta(\rho)=-Q$ and $\nabla\Theta(\rho)=\rho-\rho_T$, 
respectively. In particular, the terminal derivative associated with
quadratic tracking is Hermitian but is generally indefinite. This fact
must be taken into account in the low-rank approximation of the
adjoint variable.

The optimality conditions are derived at the continuous level before
introducing a time discretization. This optimize-then-discretize
ordering is essential for the present approach: it yields a forward
Lindblad equation, a backward adjoint equation, and a pointwise
Hamiltonian minimization that can subsequently be realized by either
full-rank or low-rank exponential midpoint propagators.

\section{Pontryagin optimality system}
\label{sec:PMP}

We now derive the first-order optimality system associated with
\eqref{eq:optimal-control-problem}. The derivation is carried out on the
real Hilbert space \(\mathbb{H}_m\), equipped with the pairing
\eqref{eq:HS-pairing}, and follows the Pontryagin framework commonly
used in quantum optimal control; see, for example,
\citep{BorziSQH,Borzi2017}. The continuous optimality system is derived
before the state and adjoint equations are discretized.

We first determine the adjoint of the controlled Lindblad generator
with respect to the real Hilbert-Schmidt pairing
\eqref{eq:HS-pairing}. 

For \(v\in U\), the controlled Lindblad generator is
\[
\mathcal{L}_v(\rho)
=
-\mathrm{i}[H_0+vH_1,\rho]
+
\sum_{k=1}^{K}
\gamma_k
\left(
L_k\rho L_k^\dagger
-
\frac12
\left\{
L_k^\dagger L_k,\rho
\right\}
\right).
\]
Its adjoint with respect to the real Hilbert-Schmidt pairing is the
operator $\mathcal{L}_v^\ast:\mathbb{H}_m\to\mathbb{H}_m$ 
defined by
\[
\langle q,\mathcal{L}_v(\rho)\rangle_{\mathbb{H}_m}
=
\langle \mathcal{L}_v^\ast(q),\rho\rangle_{\mathbb{H}_m},
\qquad
\rho,q\in\mathbb{H}_m.
\]

For the Hamiltonian contribution, using the Hermiticity of
\(H_0+vH_1\), \(\rho\), and \(q\), together with the cyclicity of the
trace, we obtain
\[
\begin{aligned}
\left\langle
q,-\mathrm{i}[H_0+vH_1,\rho]
\right\rangle_{\mathbb{H}_m}
&=
\operatorname{Tr}
\left(
q\left[-\mathrm{i}(H_0+vH_1)\rho
+\mathrm{i}\rho(H_0+vH_1)\right]
\right)
\\
&=
\operatorname{Tr}
\left(
\mathrm{i}[H_0+vH_1,q]\rho
\right).
\end{aligned}
\]
For the dissipative contribution, we have 
\[\operatorname{Tr}
\left(
qL_k\rho L_k^\dagger
\right)
=
\operatorname{Tr}
\left(
L_k^\dagger qL_k\rho
\right),
\]
and
\[
\operatorname{Tr}
\left(
q\{L_k^\dagger L_k,\rho\}
\right)
=
\operatorname{Tr}
\left(
\{L_k^\dagger L_k,q\}\rho
\right).
\]
Consequently, we obtain
\begin{equation}
\label{eq:adjoint-Lindblad-generator}
\mathcal{L}_v^\ast(q)
=
\mathrm{i}[H_0+vH_1,q]
+
\sum_{k=1}^{K}
\gamma_k
\left(
L_k^\dagger qL_k
-
\frac12
\left\{
L_k^\dagger L_k,q
\right\}
\right).
\end{equation}

The adjoint equation associated with a given control \(u\) and state
\(\rho_u\) is therefore
\begin{equation}
\label{eq:adjoint-equation-PMP}
\begin{aligned}
-\dot q(t) = \mathcal{L}_{u(t)}^\ast(q(t)),  \qquad q(T)=\nabla\Theta(\rho_u(T)).
\end{aligned}
\end{equation}
Equivalently, when written as a differential equation evolving backward
from \(T\) to \(0\), we have
\begin{equation}
\label{AdjointLindbladEq}
\dot q(t)
=
-\mathrm{i}[H_0+u(t)H_1,q(t)]
-
\sum_{k=1}^{K}
\gamma_k
\left(
L_k^\dagger q(t)L_k
-
\frac12
\left\{
L_k^\dagger L_k,q(t)
\right\}
\right).
\end{equation}

The adjoint evolution preserves Hermiticity. In general, however, it
does not preserve positive semidefiniteness. 
If the terminal condition is positive semidefinite, the backward
adjoint propagation preserves positive semidefiniteness when viewed
through the corresponding time-reversed completely positive evolution.

Having determined the adjoint operator, we now introduce the
Hamilton-Pontryagin (HP) function, which combines the running cost with the
controlled dynamics. Owing to the bilinear dependence of the Lindblad
equation on the control, the HP function is affine with respect to the
control variable, a property that will be fundamental for the SQH
iteration developed in the next section.

For \(\rho,q\in\mathbb{H}_m\) and \(v\in U\), we define the 
Hamilton Pontryagin (HP) function by
\begin{equation}
\label{eq:Pontryagin-Hamiltonian}
\mathscr{H}(\rho,q,v)
=
\ell(v)
+
\left\langle
q,\mathcal{L}_v(\rho)
\right\rangle_{\mathbb{H}_m}.
\end{equation}
Since the control enters the Hamiltonian linearly through $H(v)=H_0+vH_1$, 
we may decompose
\[
\mathcal{L}_v(\rho)
=
\mathcal{L}_0(\rho)
+
v\mathcal{B}(\rho),
\]
where $ \mathcal{B}(\rho) = -\mathrm{i}[H_1,\rho]$. It follows that
\begin{equation}
\label{eq:Pontryagin-Hamiltonian-affine}
\mathscr{H}(\rho,q,v)
=
\ell(v)
+
\left\langle
q,\mathcal{L}_0(\rho)
\right\rangle_{\mathbb{H}_m}
+
v\left\langle
q,\mathcal{B}(\rho)
\right\rangle_{\mathbb{H}_m}.
\end{equation}
The part depending on the control is therefore 
$h(\rho,q,v) = \ell(v) + v\,\zeta(\rho,q)$, where
\begin{equation}
\label{eq:switching-function}
\zeta(\rho,q)
=
\left\langle
q,-\mathrm{i}[H_1,\rho]
\right\rangle_{\mathbb{H}_m}.
\end{equation}
Using the cyclicity of the trace, this quantity can also be written as
\begin{equation}
\label{eq:switching-function-alternative}
\zeta(\rho,q)
=
\operatorname{Tr}
\left(
-\mathrm{i}[q,H_1]\rho
\right).
\end{equation}
Since \(-\mathrm{i}[H_1,\rho]\) is Hermitian whenever \(H_1\) and
\(\rho\) are Hermitian, the switching function is real-valued.

The Pontryagin optimality system is obtained by combining the state
equation, the adjoint equation, and the stationarity condition
associated with the HP function. We therefore compute
the first variation of the objective functional and derive the
corresponding first-order necessary optimality conditions. 

Let \(u\in\mathcal{U}_{\mathrm{ad}}\), let \(\rho_u\) be the
corresponding state, and consider an admissible perturbation
\(u+\varepsilon\delta u\). The linearized state
\(z\in C([0,T];\mathbb{H}_m)\) satisfies
\begin{equation}
\label{eq:linearized-state}
\begin{aligned}
\dot z(t)
=
\mathcal{L}_{u(t)}(z(t))
+
\delta u(t)\mathcal{B}(\rho_u(t)),
\qquad
z(0)= 0.
\end{aligned}
\end{equation}
Assume temporarily that \(\ell\) is differentiable. Then
\[
J'(u)\delta u
=
\left\langle
\nabla\Theta(\rho_u(T)),z(T)
\right\rangle_{\mathbb{H}_m}
+
\int_0^T
\ell'(u(t))\delta u(t)\,dt.
\]
Let \(q\) solve \eqref{eq:adjoint-equation-PMP}. By differentiating
\(\langle q(t),z(t)\rangle_{\mathbb{H}_m}\), using the state,
linearized, and adjoint equations, we obtain
\[
\left\langle q(T),z(T)\right\rangle_{\mathbb{H}_m}
=
\int_0^T
\delta u(t)
\left\langle
q(t),\mathcal{B}(\rho_u(t))
\right\rangle_{\mathbb{H}_m}
\,dt.
\]
Hence, we have
\begin{equation}
\label{eq:reduced-gradient}
J'(u)\delta u
=
\int_0^T
\left[
\ell'(u(t))
+
\zeta(\rho_u(t),q(t))
\right]
\delta u(t)\,dt.
\end{equation}

We see that the appropriate PMP optimality condition is 
expressed through pointwise minimization of the Hamilton Pontryagin function.
This characterization is the basis of the SQH strategy developed for
nonsmooth quantum control in
\citep{BorziSQH,Breitenbach2020} and extended below to the Lindblad
setting.

\begin{theorem}[Pontryagin maximum principle]
\label{thm:PMP}
Let \(u^\ast\in\mathcal{U}_{\mathrm{ad}}\) be an optimal control and let
\(\rho^\ast\) be the corresponding solution of
\eqref{LindbladEq}. Then there exists an adjoint variable
\(q^\ast\in C([0,T];\mathbb{H}_m)\) satisfying
\begin{equation}
\label{eq:PMP-adjoint}
\begin{aligned}
-\dot q^\ast(t) = \mathcal{L}_{u^\ast(t)}^\ast(q^\ast(t)),
\qquad
q^\ast(T) = \nabla\Theta(\rho^\ast(T)),
\end{aligned}
\end{equation}
such that
\begin{equation}
\label{eq:PMP-minimum-condition}
u^\ast(t)
\in
\operatorname*{arg\,min}_{v\in U}
\mathscr{H}
\left(
\rho^\ast(t),q^\ast(t),v
\right)
\end{equation}
for almost every \(t\in(0,T)\).

Equivalently, since the control-independent part of
\(\mathscr{H}\) does not affect the minimization,
\begin{equation}
\label{eq:PMP-reduced-minimum-condition}
u^\ast(t)
\in
\operatorname*{arg\,min}_{v\in U}
\left\{
\ell(v)
+
v\,\zeta(\rho^\ast(t),q^\ast(t))
\right\}
\end{equation}
for almost every \(t\in(0,T)\).
\end{theorem}

\begin{proof}
The controlled Lindblad equation is a finite-dimensional control system
on the real vector space \(\mathbb{H}_m\), which may be identified with
\(\mathbb{R}^{m^2}\). Its vector field is linear in the state and affine
in the control. The adjoint equation follows from
\eqref{eq:adjoint-Lindblad-generator}, while the terminal condition is
determined by the derivative of the terminal functional.

Under the hypotheses of the Pontryagin principle appropriate to the
admissible control class and the running cost \(\ell\), an optimal
control satisfies the pointwise Hamiltonian condition
\eqref{eq:PMP-minimum-condition}. The compactness of \(U\) and the lower
semicontinuity of \(\ell\) guarantee that the pointwise minimization
problem admits a solution.
\end{proof}

\begin{rem}
\label{rem:maximum-minimum-convention}
The terminology ``maximum principle'' is retained in accordance with
the classical literature. Since the problem is formulated as a
minimization problem and the Hamiltonian is defined by
\eqref{eq:Pontryagin-Hamiltonian}, the corresponding pointwise
condition takes the minimum form
\eqref{eq:PMP-minimum-condition}. An alternative sign convention for
the Hamiltonian would lead to an equivalent maximum condition.
\end{rem}

\begin{rem}
\label{rem:smooth-PMP-condition}
If \(\ell\) is differentiable and \(u^\ast(t)\) belongs to the interior
of \(U\), then \eqref{eq:PMP-reduced-minimum-condition} implies
\begin{equation}
\label{eq:smooth-PMP-stationarity}
\ell'(u^\ast(t))
+
\zeta(\rho^\ast(t),q^\ast(t))
=
0
\end{equation}
for almost every \(t\in(0,T)\). For example, if $\ell(v)=\frac{\alpha}{2}v^2$, 
$\alpha>0$,  then
\[
u^\ast(t)
=
\operatorname{proj}_U
\left(
-\frac{1}{\alpha}
\zeta(\rho^\ast(t),q^\ast(t))
\right).
\]
For nonsmooth costs, the minimization condition
\eqref{eq:PMP-reduced-minimum-condition} remains meaningful even when
no derivative of \(\ell\) exists.
\end{rem}

\section{Sequential quadratic Hamiltonian method}
\label{sec:SQH}

The sequential quadratic Hamiltonian (SQH) method
\citep{BorziSQH,Breitenbach2020}
belongs to the class of successive approximation methods derived from
the Pontryagin maximum principle. These methods generate a sequence of
controls by repeated pointwise minimization of the Pontryagin
Hamiltonian while keeping the state and adjoint variables fixed at the
current iterate. Their computational simplicity makes them attractive
for large-scale optimal control problems, since each update reduces to
a family of independent low-dimensional minimization problems.

However, classical successive approximation methods may exhibit poor
robustness because unrestricted pointwise Hamiltonian minimization can
produce excessively large changes of the control and therefore may fail
to decrease the objective functional. The SQH method overcomes this
difficulty by introducing a quadratic stabilization term that penalizes
the deviation from the current control iterate
\citep{BorziSQH,Breitenbach2020}. Unlike conventional gradient-based
optimization methods, the resulting pointwise minimization remains
well-defined for smooth, nonsmooth, and even nonconvex 
control costs, making the method particularly suitable for quantum
optimal control problems with general admissible control sets. 

The SQH method is formulated entirely at the continuous level.
Consequently, the optimization algorithm is independent of the
particular numerical approximation used for the state and adjoint
equations. In the present work, these equations are discretized only
after the continuous control update has been defined, following an
optimize-then-discretize strategy. This separation allows the same SQH
framework to be combined naturally with both the full-rank and the
low-rank exponential midpoint propagators developed in the following
sections, without modifying the optimization procedure itself.

Starting from a control iterate $u^k\in\mathcal U_{\mathrm{ad}}$, 
the SQH iteration first computes the corresponding state and adjoint
variables by solving the continuous optimality system. These variables
define the local information required for the subsequent pointwise
Hamiltonian minimization. Specifically, the state \(\rho^k\) is
obtained from
\begin{equation}
\label{eq:SQH-state}
\begin{aligned}
\dot{\rho}^k(t)= \mathcal{L}_{u^k(t)}(\rho^k(t)),
\qquad
\rho^k(0) = \rho_0.
\end{aligned}
\end{equation}
The adjoint \(q^k\) is then computed backward in time from
\begin{equation}
\label{eq:SQH-adjoint}
\begin{aligned}
-\dot q^k(t) = \mathcal{L}_{u^k(t)}^\ast(q^k(t)),
\qquad
q^k(T) = \nabla\Theta(\rho^k(T)).
\end{aligned}
\end{equation} 

The classical HP function is then augmented by a quadratic
stabilization term that penalizes deviations from the current control
iterate. This regularization preserves the local minimization
principle while preventing excessively large control updates and
providing the mechanism underlying the sufficient-decrease analysis
developed below. 

For a stabilization parameter \(\varepsilon_k>0\), we define the
quadratically augmented Hamilton Pontryagin function
\begin{equation}
\label{eq:augmented-Pontryagin-Hamiltonian}
\mathscr{H}_{\varepsilon_k}
(\rho,q,v;u^k)
=
\mathscr{H}(\rho,q,v)
+
\frac{\varepsilon_k}{2}|v-u^k|^2.
\end{equation}
Since only the terms depending on the optimization variable
\(v\) influence the pointwise minimization, all remaining
contributions to the Hamiltonian may be omitted. Consequently, the
control update reduces to the family of independent scalar
minimization problems 
\begin{equation}
\label{eq:SQH-pointwise-update}
u^{k+1}(t)
\in
\operatorname*{arg\,min}_{v\in U}
\left\{
\ell(v)
+
v\,\zeta(\rho^k(t),q^k(t))
+
\frac{\varepsilon_k}{2}|v-u^k(t)|^2
\right\}
\end{equation}
for almost every \(t\in(0,T)\).

For almost every \(t\), the admissible set \(U\) is nonempty and compact and the
objective function in
\eqref{eq:SQH-pointwise-update}
is lower semicontinuous with respect to \(v\). Therefore, the set of
minimizers is nonempty. When the running cost
\(\ell\) is nonconvex, uniqueness of the minimizer
cannot generally be expected. In this case, any measurable selection
satisfying \eqref{eq:SQH-pointwise-update} may be employed in the SQH iteration. 
By comparison with the admissible choice \(v=u^k(t)\), the update
satisfies
\begin{equation}
\label{eq:SQH-Hamiltonian-decrease-pointwise}
\mathscr{H}
(\rho^k(t),q^k(t),u^{k+1}(t))
-
\mathscr{H}
(\rho^k(t),q^k(t),u^k(t))
\leq
-\frac{\varepsilon_k}{2}
|u^{k+1}(t)-u^k(t)|^2
\end{equation}
for almost every \(t\in(0,T)\).

Notice that the minimization problem
\eqref{eq:SQH-pointwise-update}
is completely decoupled with respect to time. Consequently, once the
state and adjoint variables have been computed, the control update
consists of independent finite-dimensional optimization problems,
which can be solved analytically for many common running costs or by
inexpensive local optimization procedures otherwise.

\medskip

The decrease of the HP function must be compared with the
change of the state generated by the new control. We first establish a
continuous-dependence estimate.

\begin{lemma}
\label{lem:state-control-dependence}
Let \(u,v\in\mathcal{U}_{\mathrm{ad}}\), and let \(\rho_u,\rho_v\)
denote the corresponding solutions of
\eqref{LindbladEq}. Then there exists a constant
\(C_\rho>0\), independent of \(u\) and \(v\), such that
\begin{equation}
\label{eq:state-control-dependence}
\max_{0\leq t\leq T}
\|\rho_u(t)-\rho_v(t)\|_1
\leq
C_\rho
\|u-v\|_{L^1(0,T)},
\end{equation}
where $\|\cdot\|_1$ denotes the trace norm of a matrix. In particular, it holds
\begin{equation}
\label{eq:state-control-dependence-L2}
\max_{0\leq t\leq T}
\|\rho_u(t)-\rho_v(t)\|_1
\leq
C_\rho\sqrt{T}
\|u-v\|_{L^2(0,T)}.
\end{equation}
\end{lemma}

\begin{proof}
Set $\delta\rho(t)=\rho_u(t)-\rho_v(t)$. Using
\[
\mathcal{L}_{u}(\rho)
=
\mathcal{L}_0(\rho)
+
u\mathcal{B}(\rho),
\qquad
\mathcal{B}(\rho)=-\mathrm{i}[H_1,\rho],
\]
we obtain
\[
\dot{\delta\rho}
=
\mathcal{L}_{u(t)}(\delta\rho)
+
(u(t)-v(t))\mathcal{B}(\rho_v).
\]
Since the exact Lindblad evolution preserves positivity and trace, $\|\rho_v(t)\|_1=1$. 
Moreover, we have
\[
\|\mathcal{B}(\rho_v(t))\|_1
=
\|-\mathrm{i}[H_1,\rho_v(t)]\|_1
\leq
2\|H_1\|\,\|\rho_v(t)\|_1
=
2\|H_1\|.
\]
A standard stability estimate for the Lindblad evolution and
Gronwall's inequality therefore give
\[
\|\delta\rho(t)\|_1
\leq
C_\rho
\int_0^t|u(s)-v(s)|\,ds.
\]
Taking the maximum over \(t\in[0,T]\) proves
\eqref{eq:state-control-dependence}. The estimate
\eqref{eq:state-control-dependence-L2} follows from the
Cauchy-Schwarz inequality.
\end{proof}

The sufficient-decrease analysis also requires a uniform estimate for
the adjoint variable. The following lemma shows that the adjoint remains
uniformly bounded in the trace norm, independently of the admissible
control.

\begin{lemma}
\label{lem:adjoint-bound}
Assume that \(\nabla\Theta\) is bounded on the set of density matrices,
that is,
\begin{equation}
\label{eq:terminal-gradient-bound}
\sup_{\rho\in\mathcal{D}_m}
\|\nabla\Theta(\rho)\|_1
\leq
C_\Theta.
\end{equation}
Then there exists a constant \(C_q>0\), independent of
\(u\in\mathcal{U}_{\mathrm{ad}}\), such that the solution of
\eqref{eq:adjoint-equation-PMP} satisfies
\begin{equation}
\label{eq:adjoint-uniform-bound}
\max_{0\leq t\leq T}\|q(t)\|_1
\leq
C_q.
\end{equation}
\end{lemma}

\begin{proof}
Introduce the time-reversed variable $p(s)=q(T-s)$. 
Then \(p\) satisfies a forward linear equation generated by the
time-reversed adjoint Lindblad operator, with initial value
\[
p(0)=\nabla\Theta(\rho_u(T)).
\]
A stability estimate for this evolution gives
\[
\|p(s)\|_1
\leq
C_{\mathcal{L}}
\|p(0)\|_1,
\]
where \(C_{\mathcal{L}}>0\) depends on \(T\), the decay rates, and the
Lindblad operators, but not on the control. Assumption
\eqref{eq:terminal-gradient-bound} therefore yields
\[
\|q(t)\|_1
\leq
C_{\mathcal{L}}C_\Theta.
\]
The assertion follows with $C_q=C_{\mathcal{L}}C_\Theta$. 
\end{proof}

The preceding stability estimates provide the ingredients required to
establish the sufficient-decrease property of the SQH iteration. The
only additional assumption concerns the regularity of the terminal
functional.

For a general terminal functional, the variation of the objective
contains a nonlinear terminal remainder. We therefore impose the
following assumption:
\begin{equation}
\label{eq:Lipschitz-terminal-gradient}
\|\nabla\Theta(\rho_1)-\nabla\Theta(\rho_2)\|_F
\leq
L_\Theta
\|\rho_1-\rho_2\|_F
\end{equation}
for all \(\rho_1,\rho_2\in\mathcal{D}_m\), with some
\(L_\Theta\geq0\).

This assumption holds with \(L_\Theta=0\) for the observable objective
\eqref{eq:observable-terminal-cost}, and with \(L_\Theta=1\) for the
quadratic tracking objective
\eqref{eq:tracking-terminal-cost}.

\begin{theorem}[Sufficient decrease]
\label{thm:SQH-sufficient-decrease}
Assume \eqref{eq:terminal-gradient-bound} and
\eqref{eq:Lipschitz-terminal-gradient}. Let \(u^k\) be the current
control, let \(\rho^k\) and \(q^k\) solve
\eqref{eq:SQH-state}-\eqref{eq:SQH-adjoint}, and let \(u^{k+1}\) be
defined by \eqref{eq:SQH-pointwise-update}. Then there exists a constant
\(C_{\mathrm{SQH}}>0\), independent of \(k\), such that
\begin{equation}
\label{eq:SQH-objective-estimate}
J(u^{k+1})-J(u^k)
\leq
-
\left(
\frac{\varepsilon_k}{2}
-
C_{\mathrm{SQH}}
\right)
\|u^{k+1}-u^k\|_{L^2(0,T)}^2.
\end{equation}
Consequently, if
\begin{equation}
\label{eq:SQH-stabilization-condition}
\varepsilon_k
\geq
2C_{\mathrm{SQH}}+2\eta
\end{equation}
for some \(\eta>0\), then
\begin{equation}
\label{eq:SQH-sufficient-decrease}
J(u^{k+1})
\leq
J(u^k)
-
\eta \, 
\|u^{k+1}-u^k\|_{L^2(0,T)}^2.
\end{equation}
\end{theorem}

\begin{proof}
Set
\[
\delta u=u^{k+1}-u^k,
\qquad
\delta\rho=\rho^{k+1}-\rho^k.
\]
By the descent lemma associated with
\eqref{eq:Lipschitz-terminal-gradient}, we have 
\begin{equation}
\label{eq:terminal-descent-lemma}
\Theta(\rho^{k+1}(T))
-
\Theta(\rho^k(T))
\leq
\left\langle
q^k(T),\delta\rho(T)
\right\rangle_{\mathbb{H}_m}
+
\frac{L_\Theta}{2}
\|\delta\rho(T)\|_F^2.
\end{equation}
Using the state equations for \(\rho^{k+1}\) and \(\rho^k\), the
adjoint equation for \(q^k\), and integration by parts, we obtain
\begin{align}
J(u^{k+1})-J(u^k)
&\leq
\int_0^T
\Bigl[
\mathscr{H}
(\rho^k,q^k,u^{k+1})
-
\mathscr{H}
(\rho^k,q^k,u^k)
\Bigr]\,dt
\nonumber\\
&\quad
+
\int_0^T
\delta u(t)
\left\langle
q^k(t),\mathcal{B}(\delta\rho(t))
\right\rangle_{\mathbb{H}_m}
\,dt
+
\frac{L_\Theta}{2}
\|\delta\rho(T)\|_F^2.
\label{eq:SQH-fundamental-estimate}
\end{align}
The pointwise estimate
\eqref{eq:SQH-Hamiltonian-decrease-pointwise} yields
\[
\int_0^T
\Bigl[
\mathscr{H}
(\rho^k,q^k,u^{k+1})
-
\mathscr{H}
(\rho^k,q^k,u^k)
\Bigr]\,dt
\leq
-\frac{\varepsilon_k}{2}
\|\delta u\|_{L^2(0,T)}^2.
\]
Furthermore, we have
\[
\left|
\left\langle
q^k,\mathcal{B}(\delta\rho)
\right\rangle_{\mathbb{H}_m}
\right|
\leq
2\|H_1\|\,
\|q^k\|_1
\|\delta\rho\|_1.
\]
By Lemmas~\ref{lem:state-control-dependence} and
\ref{lem:adjoint-bound}, we obtain 
\[
\begin{aligned}
\left|
\int_0^T
\delta u(t)
\left\langle
q^k(t),\mathcal{B}(\delta\rho(t))
\right\rangle_{\mathbb{H}_m}
\,dt
\right| \leq C
\|\delta u\|_{L^1(0,T)}^2 \leq
CT\|\delta u\|_{L^2(0,T)}^2.
\end{aligned}
\]
Similarly, we have
\[
\|\delta\rho(T)\|_F^2
\leq
C
\|\delta u\|_{L^1(0,T)}^2
\leq
CT\|\delta u\|_{L^2(0,T)}^2.
\]
Combining these estimates with
\eqref{eq:SQH-fundamental-estimate} gives
\eqref{eq:SQH-objective-estimate}. The sufficient-decrease estimate
\eqref{eq:SQH-sufficient-decrease} follows from
\eqref{eq:SQH-stabilization-condition}.
\end{proof}

By Theorem~\ref{thm:SQH-sufficient-decrease}, the acceptance condition
\eqref{eq:SQH-accepted-decrease} is guaranteed whenever the trial
stabilization parameter satisfies
\eqref{eq:SQH-stabilization-condition}. The adaptive inner loop
therefore produces an accepted iterate after finitely many increases of
the stabilization parameter.

\begin{corollary}[Monotonicity and asymptotic regularity]
\label{cor:SQH-monotonicity}

Assume that the objective functional \(J\) is bounded from below on
\(\mathcal{U}_{\mathrm{ad}}\), and let
\(\{u^k\}_{k\geq0}\) be the sequence of accepted control iterates
generated by Algorithm~\ref{alg:SQH}. Suppose that every accepted step
satisfies
\begin{equation}
\label{eq:SQH-accepted-decrease}
J(u^{k+1})
\leq
J(u^k)
-
\eta
\|u^{k+1}-u^k\|_{L^2(0,T)}^2
\end{equation}
with a fixed constant \(\eta>0\). Then the sequence
\(\{J(u^k)\}_{k\geq0}\) is monotonically decreasing and converges.
Moreover,
\begin{equation}
\label{eq:SQH-square-summability}
\sum_{k=0}^{\infty}
\|u^{k+1}-u^k\|_{L^2(0,T)}^2
<
\infty,
\end{equation}
and consequently
\begin{equation}
\label{eq:SQH-asymptotic-regularity}
\lim_{k\to\infty}
\|u^{k+1}-u^k\|_{L^2(0,T)}
=
0.
\end{equation}

\end{corollary}

\begin{proof}
From \eqref{eq:SQH-accepted-decrease}, it follows that $J(u^{k+1}) \leq J(u^k)$ 
for every \(k\geq0\). Hence, \(\{J(u^k)\}_{k\geq0}\) is monotonically
decreasing. Since \(J\) is bounded from below on
\(\mathcal{U}_{\mathrm{ad}}\), there exists
\(J_\infty\in\mathbb{R}\) such that $J(u^k)\longrightarrow J_\infty$ as 
$k\to\infty$. Summing \eqref{eq:SQH-accepted-decrease} from \(k=0\) to \(M\) gives
\[
\eta
\sum_{k=0}^{M}
\|u^{k+1}-u^k\|_{L^2(0,T)}^2
\leq
J(u^0)-J(u^{M+1}).
\]
Passing to the limit \(M\to\infty\), we obtain
\[
\eta
\sum_{k=0}^{\infty}
\|u^{k+1}-u^k\|_{L^2(0,T)}^2
\leq
J(u^0)-J_\infty
<
\infty.
\]
This proves \eqref{eq:SQH-square-summability}. Since the terms of a
convergent series of nonnegative numbers must converge to zero,
\eqref{eq:SQH-asymptotic-regularity} follows.
\end{proof}

\begin{rem}[Finite stopping for a positive tolerance]
\label{rem:SQH-finite-stopping}

Let \(\kappa>0\). By \eqref{eq:SQH-asymptotic-regularity}, there exists an index
\(k_0\) such that
\[
\|u^{k_0+1}-u^{k_0}\|_{L^2(0,T)}
\leq
\kappa.
\]
Therefore, in the absence of a smaller prescribed maximum iteration
count, Algorithm~\ref{alg:SQH} terminates after finitely many accepted
outer iterations.
\end{rem}

\subsection{Adaptive choice of the stabilization parameter}
\label{subsec:adaptive-SQH}

The constant \(C_{\mathrm{SQH}}\) in
Theorem~\ref{thm:SQH-sufficient-decrease} is generally not known
explicitly and, even if an estimate were available, it could lead to an
overly conservative stabilization parameter. We therefore determine the
stabilization parameter adaptively.

Let
\[
\eta>0,
\qquad
\sigma>1,
\qquad
\zeta\in(0,1),
\]
where \(\eta\) is the prescribed sufficient-decrease constant,
\(\sigma\) is the parameter-increase factor, and \(\zeta\) is the
parameter-decrease factor. Starting from a trial value
\(\varepsilon>0\), the pointwise SQH update is computed and the state
equation is solved with the resulting trial control
\(u^{\mathrm{trial}}\). The trial step is accepted if
\begin{equation}
\label{eq:SQH-acceptance-condition}
J(u^{\mathrm{trial}}) \leq J(u^k) - \eta \, \|u^{\mathrm{trial}}-u^k\|_{L^2(0,T)}^2.
\end{equation}
If \eqref{eq:SQH-acceptance-condition} fails, the stabilization
parameter is increased according to $\varepsilon \leftarrow \sigma\varepsilon$, 
and the pointwise minimization is repeated using the same state
\(\rho^k\) and adjoint \(q^k\). After an accepted step, the initial
trial value for the next outer iteration is chosen as $\varepsilon_{k+1}=\zeta\varepsilon$. 
Thus, rejected steps increase the stabilization, whereas accepted steps
allow a moderate reduction of the stabilization parameter.

\begin{corollary}[Finite termination of the inner SQH loop]
\label{cor:SQH-inner-loop}

Assume the hypotheses of
Theorem~\ref{thm:SQH-sufficient-decrease}. For every outer iterate
\(u^k\) and every initial trial value \(\varepsilon>0\), the inner loop
that repeatedly replaces $\varepsilon \leftarrow \sigma\varepsilon$, 
$\sigma>1$, terminates after finitely many updates.
\end{corollary}

\begin{proof}
By Theorem~\ref{thm:SQH-sufficient-decrease}, the acceptance condition
\eqref{eq:SQH-acceptance-condition} is guaranteed whenever 
$\varepsilon \geq 2C_{\mathrm{SQH}}+2\eta$. 
Since the sequence of trial parameters generated in the inner loop is 
$\varepsilon,\, \sigma\varepsilon,\, \sigma^2\varepsilon,\ldots$, 
and \(\sigma>1\), this sequence eventually exceeds
\(2C_{\mathrm{SQH}}+2\eta\). Hence the sufficient-decrease condition
must hold after finitely many inner iterations.
\end{proof}

\begin{corollary}[Boundedness of the accepted stabilization parameters]
\label{cor:SQH-epsilon-bounded}

Let \(\widehat\varepsilon_k\) denote the stabilization parameter at
which the \(k\)-th SQH step is accepted. Then
\[
\widehat\varepsilon_k
\leq
\max
\left\{
\varepsilon_0,\,
\sigma\bigl(2C_{\mathrm{SQH}}+2\eta\bigr)
\right\}
\]
for every \(k\geq0\).

\end{corollary}

\begin{proof}
If the initial trial value at an outer iteration already satisfies the
acceptance condition, no increase is required. Otherwise, let
\(\varepsilon\) be the last rejected value. Since the next value
\(\sigma\varepsilon\) is accepted and rejection is impossible once
\(\varepsilon\geq2C_{\mathrm{SQH}}+2\eta\), it follows that
\[
\widehat\varepsilon_k < \sigma\bigl(2C_{\mathrm{SQH}}+2\eta\bigr).
\]
After acceptance, the next initial trial value is reduced by the factor
\(\zeta\in(0,1)\), and hence cannot create a larger upper bound.
Including the initial value \(\varepsilon_0\) gives the assertion.
\end{proof}

The adaptive SQH iteration depends on the initial stabilization
parameter, the increase and decrease factors, and the prescribed
sufficient-decrease constant. Numerical experience with SQH methods in
other optimization settings indicates that the algorithm is relatively
insensitive to moderate variations of these parameters, while the
stopping tolerance affects only the final termination accuracy. In the
present work, the same parameter values are employed throughout all
numerical experiments.

Algorithm~\ref{alg:SQH} summarizes one complete SQH iteration.
\begin{algorithm}[h!]
\caption{Sequential quadratic Hamiltonian method}
\label{alg:SQH}
\begin{algorithmic}[1]
\State Choose
\[
u^0\in\mathcal{U}_{\mathrm{ad}},
\qquad
\varepsilon_0>0,
\qquad
\sigma>1,
\qquad
\zeta\in(0,1),
\qquad
\eta>0,
\qquad
\kappa>0,
\]
and a maximum number of outer iterations \(k_{\max}\).
\State Solve the state equation
\eqref{eq:SQH-state} for \(\rho^0\).
\For{\(k=0,1,\ldots,k_{\max}-1\)}
  \State Solve the adjoint equation
  \eqref{eq:SQH-adjoint} backward for \(q^k\).
  \State Set the trial stabilization parameter
  \[
  \varepsilon=\varepsilon_k.
  \]
  \Repeat
    \State Compute a measurable pointwise minimizer
    \[
    u^{\mathrm{trial}}(t)
    \in
    \operatorname*{arg\,min}_{v\in U}
    \left\{
    \ell(v)
    +
    v\,\zeta(\rho^k(t),q^k(t))
    +
    \frac{\varepsilon}{2}
    |v-u^k(t)|^2
    \right\}.
    \]
    \State Solve the state equation with
    \(u^{\mathrm{trial}}\) and denote the resulting state by
    \(\rho^{\mathrm{trial}}\).
    \State Evaluate \(J(u^{\mathrm{trial}})\).
    \If{
    \[
    J(u^{\mathrm{trial}})
    \leq
    J(u^k)
    -
    \eta
    \|u^{\mathrm{trial}}-u^k\|_{L^2(0,T)}^2
    \]
    }
      \State Accept
      \[
      u^{k+1}=u^{\mathrm{trial}},
      \qquad
      \rho^{k+1}=\rho^{\mathrm{trial}},
      \qquad
      \widehat\varepsilon_k=\varepsilon.
      \]
      \State Set
      \[
      \varepsilon_{k+1}
      =
      \zeta\widehat\varepsilon_k.
      \]
    \Else
      \State Increase
      \[
      \varepsilon
      \leftarrow
      \sigma\varepsilon.
      \]
    \EndIf
  \Until{the trial control is accepted}
  \State Set
  \[
  \omega_k
  =
  \|u^{k+1}-u^k\|_{L^2(0,T)}.
  \]
  \If{\(\omega_k\leq\kappa\)}
    \State Stop.
  \EndIf
\EndFor
\end{algorithmic}
\end{algorithm}

\FloatBarrier

The SQH algorithm is formulated at the continuous level and is
independent of the particular time discretization. Its implementation
requires accurate and structure-preserving propagators for the forward
and adjoint Lindblad equations. In the present work, these equations
are discretized by second-order full- and low-rank exponential midpoint
schemes, leading to the FREM-SQH and LREM-SQH algorithms.

\section{Exponential midpoint propagation of the optimality system}
\label{sec-ExpInt}

The forward and adjoint equations arising in the SQH iteration are now
discretized by exponential midpoint propagators. To treat coherent and
time-dependent incoherent coefficients in a common form, define

\begin{equation}\label{equ2.1}
  A(t):=- i \, H(t)-\frac{1}{2}\sum_{k=1}^K\gamma_k(t) \, L_k^{\dag} \, L_k.
\end{equation}
Then, we can rewrite \eqref{LindbladEq} as
\begin{equation}\label{equ2.2}
  \dot{\rho}(t)=A(t) \, \rho(t)+\rho(t) \,A^{\dag}(t)+\sum_{k=1}^K\gamma_k(t) \, L_k \, \rho(t) \,L_k^{\dag}, \qquad \rho(0)=\rho_0,
\end{equation}
and the adjoint Lindblad equation \eqref{AdjointLindbladEq} can be written as
\begin{equation}\label{equ2.3}
  \dot{q}(t)=-A^{\dag}(t) \, q(t) - q(t) \,A(t)-\sum_{k=1}^K\gamma_k(t) \, L_k^{\dag} \, q(t) \,L_k, \qquad q(T)=q_T.
\end{equation}

This representation includes coherent Hamiltonian control when the
decay rates are constant and also permits time-dependent decay
coefficients. The SQH analysis developed above is restricted to the
coherent-control setting.

\subsection{Full-rank midpoint propagators}

Let \(t_n=n\tau\), \(n=0,\ldots,N\), with \(\tau=T/N\), and let
\(\rho_n\) approximate \(\rho(t_n)\). To derive the forward propagator,
freeze the linear part of \eqref{equ2.2} at \(A_j=A(t_j)\) and write
\begin{equation}\label{equ2.4}
  \dot{\rho}(t)=A_j \, \rho(t)+\rho(t) \,A^{\dag}_j+F(t,\rho(t),A_j), \qquad \rho(0)=\rho_0,
\end{equation}
where
\[F(t,\rho(t),A_j)=\sum_{k=1}^K\gamma_k(t) \, L_k \, \rho(t) \,L_k^{\dag}+(A(t)-A_j)\, \rho(t)+\rho(t)(A^{\dag}(t) -A^{\dag}_j).\]
Integrating \eqref{equ2.4} from $t_n$ to $t$ and applying the variation-of-constants formula, we get
\begin{equation}\label{equ2.5}
  \rho(t)=e^{(t-t_n)A_j} \rho(t_n) e^{(t-t_n)A^{\dag}_j} +
  \int_0^{t-t_n}e^{(t-t_n-s)A_j} F(t_n+s, \rho(t_n+s), A_j) e^{(t-t_n-s)A^{\dag}_j}ds.
\end{equation}
Letting $t=t_n+\frac{\tau}{2}:=t_{n+1/2}$ (resp. $t=t_{n+1}$) and $j=n$ (resp. $j=n+1/2$) in \eqref{equ2.5}, we obtain
\begin{subequations}\label{equ2.6}
\begin{align}
  &\rho(t_{n+1/2})=e^{\frac{\tau}{2} A_n} \, \rho(t_n) \, e^{\frac{\tau}{2} A_n^{\dag}} +
  \int_0^{\frac{\tau}{2}}e^{(\frac{\tau}{2}-s)A_n} \, F(t_n+s, \rho(t_n+s), A_n) \, e^{(\frac{\tau}{2}-s)A_n^{\dag}}ds,\\
    &\rho(t_{n+1}) = e^{\tau A_{n+1/2}} \rho(t_n)  e^{\tau A_{n+1/2}^{\dag}} 
  + \int_0^{\tau}e^{(\tau-s)A_{n+1/2}}  F(t_n+s, \rho(t_n+s), A_{n+1/2}) e^{(\tau-s)A_{n+1/2}^{\dag}}ds.
\end{align}
\end{subequations}
Approximating the integrals in \eqref{equ2.6} by left-rectangle quadrature formula  and
midpoint quadrature formula, respectively, we get the full-rank exponential midpoint (FREM) scheme
\begin{subequations}\label{equ2.7}
\begin{align}
  &\rho_{n+1/2}=e^{\frac{\tau}{2} A_n} \left( \rho_n + \frac{\tau}{2}\sum_{k=1}^K\gamma_k(t_n) \, L_k \, \rho_n \,L_k^{\dag} \right) e^{\frac{\tau}{2} A_n^{\dag}} ,\label{equ2.7a}\\
  &\rho_{n+1}=e^{\tau A_{n+1/2}} \, \rho_n \, e^{\tau A_{n+1/2}^{\dag}} +
  \tau  \sum_{k=1}^K\gamma_k(t_{n+1/2})e^{\frac{\tau}{2} A_{n+1/2}} \, L_k \, \rho_{n+1/2} \,L_k^{\dag} \, e^{\frac{\tau}{2} A_{n+1/2}^{\dag}},\label{equ2.7b}
\end{align}
\end{subequations}
$n=0,\ldots,N-1$.

The adjoint propagator is derived analogously by writing
\begin{equation}\label{equ2.8}
  \dot{q}(t)=-A^{\dag}_j \, q(t)-q(t) \,A_j-\widetilde{F}(t,q(t),A_j), \qquad q(T)=q_T,
\end{equation}
where
\[\widetilde{F}(t,q(t),A_j)=\sum_{k=1}^K\gamma_k(t) \, L_k^{\dag} \, q(t) \,L_k+(A^{\dag}(t)-A^{\dag}_j)\, q(t)+q(t)(A(t) -A_j).\]
Integrating \eqref{equ2.8} from $t_{n+1}$ to $t$ and applying the variation-of-constants formula, we obtain
\begin{eqnarray}\label{equ2.9}
  q(t)=e^{(t_{n+1}-t)A^{\dag}_j} q(t_{n+1})  e^{(t_{n+1}-t)A_j} +
  \int_{t-t_n}^{\tau}e^{(t_n+s-t)A^{\dag}_j}  \widetilde{F}(t_n+s, q(t_n+s), A_j)  e^{(t_n+s-t)A_j}ds.
\end{eqnarray}
Taking $t=t_{n+1/2}$ (resp. $t=t_{n}$) and $j=n+1$ (resp. $j=n+1/2$) in \eqref{equ2.9}, we get
\begin{subequations}\label{equ2.10}
\begin{align}
    &q(t_{n+1/2})=e^{\frac{\tau}{2} A_{n+1}^{\dag}}  q(t_{n+1})  e^{\frac{\tau}{2} A_{n+1}} 
     + \int_{\frac{\tau}{2}}^{\tau}e^{(s-\frac{\tau}{2})A_{n+1}^{\dag}} \widetilde{F}(t_n+s, q(t_n+s), A_{n+1})  e^{(s-\frac{\tau}{2})A_{n+1}}ds,\\
  &q(t_{n})=e^{\tau A_{n+1/2}^{\dag}}  q(t_{n+1})  e^{\tau A_{n+1/2}} +
  \int_0^{\tau}e^{sA_{n+1/2}^{\dag}}  \widetilde{F}(t_n+s, q(t_n+s), A_{n+1/2}) e^{sA_{n+1/2}}ds.
\end{align}
\end{subequations}
Approximating the integrals in \eqref{equ2.10} by right-rectangle quadrature formula  and
midpoint quadrature formula, respectively, we obtain the FREM scheme for the adjoint Lindblad equation
\begin{subequations}\label{equ2.11}
\begin{align}
  &q_{n+1/2}=e^{\frac{\tau}{2} A_{n+1}^{\dag}} \left( q_{n+1} + \frac{\tau}{2}\sum_{k=1}^K\gamma_k(t_{n+1}) \, L_k^{\dag} \, q_{n+1} \,L_k \right) e^{\frac{\tau}{2} A_{n+1}} ,\label{equ2.11a}\\
  &q_{n}=e^{\tau A_{n+1/2}^{\dag}} \, q_{n+1} \, e^{\tau A_{n+1/2}} +
  \tau \sum_{k=1}^K\gamma_k(t_{n+1/2}) \,e^{\frac{\tau}{2} A_{n+1/2}^{\dag}} \, L_k^{\dag} \, q_{n+1/2} \,L_k \, e^{\frac{\tau}{2} A_{n+1/2}},\label{equ2.11b}
\end{align}
\end{subequations}
$n=N-1,\ldots,0$. Note that $q_n$ and $q_{n+1/2}$ are approximations to $q(t_n)$ and $q(t_{n+1/2})$, respectively.

\begin{rem}
\label{rem:FREM-positivity}
If \(\rho_0\succeq0\) and \(q_T\succeq0\), then the full-rank forward
and adjoint propagators preserve Hermiticity and positive
semidefiniteness for every \(\tau>0\). This follows directly from the
representation of each update as a sum of congruence transformations
of positive-semidefinite matrices.
\end{rem}

In order to be concise, we simply denote the FREM scheme \eqref{equ2.7} (resp. \eqref{equ2.11}) as 
the map $\rho_{n+1}=\Phi(t_n,\rho_n)$ (resp. $q_n=\Psi(t_{n+1},q_{n+1})$). 
Note that the FREM schemes \eqref{equ2.7} might not preserve the unit trace of the 
density matrices. In order to preserve unit trace of the density matrices, 
we propose the normalized FREM scheme
\begin{equation}\label{equ2.12}
  \widetilde{\rho}_{n+1}=\Phi(t_n,\rho_n),\qquad \rho_{n+1}=\frac{\widetilde{\rho}_{n+1}}{\Tr(\widetilde{\rho}_{n+1})},\qquad n=0,\ldots,N-1. 
\end{equation}
We remark that the solution of adjoint Lindblad equation is usually not trace-preserving, 
so the normalization procedure is not adapted to the FREM scheme \eqref{equ2.11}.

\subsection{Low-rank exponential integrators}
Now we consider the low-rank variants of the FREM schemes \eqref{equ2.7} and \eqref{equ2.11}. 
Our aim is to reduce the computational cost while at the same time retain 
the accuracy of the underlying FREM schemes. 
The idea is to seek and do computations on factors $X_n\in \mathbb{C}^{m\times r_n}$ 
(resp. $Y_n\in \mathbb{C}^{m\times \tilde{r}_n}$, $S_n\in \mathbb{C}^{\tilde{r}_n\times \tilde{r}_n}$) 
with $r_n\ll m$ (resp. $\tilde{r}_n\ll m$) instead of $\rho_n$ (resp. $q_n$) such that the solutions 
of the forward and adjoint 
Lindblad equations can be well approximated as
\[
\rho(t_n)\approx X_nX_n^{\dag}:=\varrho_n,\quad q(t_n)\approx Y_nS_nY_n^{\dag}:=p_n,
\]
respectively, where we denote with $\varrho_n$ (resp. $p_n$) the numerical low-rank solution to the forward (resp. adjoint) 
Lindblad equation in order to distinguish it from $\rho_n$ (resp. $q_n$), the full-rank numerical solution of the same equation. The reason that we employ different low-rank approximation factorizations for $\rho(t_n)$ and $q(t_n)$ is that the density matrices are positive semidefinite, while the adjoint matrices may be indefinite, depending on terminal objective.

Now, assume that $\rho_n=X_nX_n^{\dag}$ and $\rho_{n+1/2}=X_{n+1/2}X_{n+1/2}^{\dag}$ and inserting these 
factorizations into \eqref{equ2.7} yields
\begin{subequations}\label{equ2.13}
  \begin{align}
  &G_{n}=\left[\sqrt{\tau\gamma_1(t_n)}L_1X_{n},\ldots,\sqrt{\tau\gamma_K(t_n)}L_KX_{n}\right],\label{equ2.13a}\\
  &X_{n+1/2}=e^{\frac{\tau}{2} A_n}\left[X_n, \sqrt{0.5}G_{n}\right],\label{equ2.13b}\\
  &G_{n+1/2}=\left[\sqrt{\tau\gamma_1(t_{n+1/2})}L_1X_{n+1/2},\ldots,\sqrt{\tau\gamma_K(t_{n+1/2})}L_KX_{n+1/2}\right],\label{equ2.13c}\\
  &X_{n+1}=\left[e^{\tau A_{n+1/2}}X_n, \, e^{\frac{\tau}{2} A_{n+1/2}}G_{n+1/2}\right].\label{equ2.13d}
\end{align}
\end{subequations}
By the notation in \eqref{equ2.13a} we mean that the $K$ matrices $\sqrt{\tau\gamma_k(t_n)}L_kX_{n}$ are placed side by side.

We remark that for many problems, the exact matrix exponential $e^{\frac{\tau}{2} A_n}$ or the exact value of the product of 
matrix exponential times vectors $e^{\frac{\tau}{2} A_n}X_n$  may be costly to compute and approximations may be
required. In our low-rank algorithms we will denote by $\mathfrak{e}^{\frac{\tau}{2} A_n}$ (resp. $\mathfrak{e}^{\frac{\tau}{2} A_n}X_n$) 
an approximation of $e^{\frac{\tau}{2} A_n}$ (resp. $e^{\frac{\tau}{2} A_n}X_n$).

In addition, note that matrices $G_{n}$ and $X_{n+1}$ have much more columns than $X_n$. Better approximations can be obtained by 
applying column compression techniques to these factors. This can be computed by truncating the singular value 
decomposition (SVD) of the given matrix. We denote with $\mathcal{T}_{\varepsilon_1}(\cdot)$ the truncated SVD of a 
matrix with error tolerance $\varepsilon_1>0$ in the sense that $\mathcal{T}_{\varepsilon_1}(X)$ represents the 
best rank $r$ approximation of the matrix $X\in \mathbb{C}^{m\times s}$ in Frobenius norm, where $r$ is the minimal integer 
such that $\sum_{j=r+1}^s\sigma^2_{j}(X)\leq \varepsilon_1$. We then get
\begin{equation}\label{equ2.14}
  \left\|XX^{\dag}-\mathcal{T}_{\varepsilon_1}(X)\mathcal{T}_{\varepsilon_1}(X)^{\dag}\right\|_1=
  \sum_{j=r+1}^s\sigma^2_{j}(X)\leq \varepsilon_1.
\end{equation}

Now, given initial low-rank approximation $\rho_0\approx \varrho_0=X_0X_0^{\dag}$ 
with $X_0\in \mathbb{C}^{m\times r_0}$ and $\Tr(\varrho_0)=1$, 
we define one step of the low-rank exponential midpoint (LREM) scheme as follows
\begin{subequations}\label{equ2.15}
  \begin{align}
  &\widetilde{G}_{n}=\left[\sqrt{\tau\gamma_1(t_n)}L_1X_{n},\ldots,\sqrt{\tau\gamma_K(t_n)}L_KX_{n}\right], \qquad G_n = \mathcal{T}_{\varepsilon_1}\left(\widetilde{G}_n\right),\\
  &\widetilde{X}_{n+1/2}=\mathfrak{e}^{\frac{\tau}{2} A_n}\left[X_n, \sqrt{0.5}G_{n}\right],\qquad X_{n+1/2}=\mathcal{T}_{\varepsilon_1}\left(\widetilde{X}_{n+1/2}\right),\\
  &\widetilde{G}_{n+1/2}=\left[\sqrt{\tau\gamma_1(t_{n+1/2})}L_1X_{n+1/2},\ldots,\sqrt{\tau\gamma_K(t_{n+1/2})}L_KX_{n+1/2}\right],\\
  &G_{n+1/2}=\mathcal{T}_{\varepsilon_1}\left(\widetilde{G}_{n+1/2}\right),\qquad \widetilde{X}_{n+1}=\left[\mathfrak{e}^{\tau A_{n+1/2}}X_n, \, \mathfrak{e}^{\frac{\tau}{2} A_{n+1/2}}G_{n+1/2}\right],\\
  &\widehat{X}_{n+1}=\mathcal{T}_{\varepsilon_1}\left(\widetilde{X}_{n+1}\right),
  \qquad X_{n+1}=\frac{\widehat{X}_{n+1}}{\left\|\widehat{X}_{n+1}\right\|_F},
\end{align}
\end{subequations}
$n=0,\ldots,N-1$.

\begin{rem}
Note that $\varrho_{n+1}=X_{n+1}X_{n+1}^{\dag}$ and it follows that the LREM scheme \eqref{equ2.15} is positivity 
and trace preserving, that is, 
\end{rem}
\begin{equation}\label{equ2.16}
  \|\varrho_{n+1}\|_1=\Tr\left(X_{n+1}X_{n+1}^{\dag}\right)=
  \frac{\Tr\left(\widehat{X}_{n+1}\widehat{X}_{n+1}^{\dag}\right)}{\left\|\widehat{X}_{n+1}\right\|_F^2}=1,\qquad n=0,\ldots,N-1.
\end{equation}

We also remark that the LREM scheme \eqref{equ2.15} is equivalent to
\begin{subequations}\label{equ2.17}
  \begin{align}
  &\widetilde{\varrho}_{n+1}=\Phi(t_n,\varrho_n),\label{equ2.17a}\\
  &\widehat{\varrho}_{n+1}=\widetilde{\varrho}_{n+1}-\vartheta_{n+1},\label{equ2.17b}\\
  & \varrho_{n+1}=\frac{\widehat{\varrho}_{n+1}}{\Tr(\widehat{\varrho}_{n+1})},\label{equ2.17c}
\end{align}
\end{subequations}
$n=0,\ldots,N-1$, where $\widehat{\varrho}_{n+1}=\widehat{X}_{n+1}\widehat{X}_{n+1}^{\dag}$ and the matrix $\vartheta_{n+1}$ can be 
seen as the perturbation caused by the approximations to the matrix exponential times vectors and the column compression procedures.

Now we consider low-rank scheme for the adjoint Lindblad equation. Assume that $q_{n+1}=Y_{n+1}S_{n+1}Y_{n+1}^{\dag}$ and $q_{n+1/2}=Y_{n+1/2}S_{n+1/2}Y_{n+1/2}^{\dag}$ and inserting these factorizations into FREM scheme \eqref{equ2.11} yields
\begin{subequations}\label{equ:adjoint-factorization}
  \begin{align}
  &W_{n+1}=\left[\sqrt{\tau\gamma_1(t_{n+1})}L_1^{\dag}Y_{n+1},\ldots,\sqrt{\tau\gamma_K(t_{n+1})}L_K^{\dag}Y_{n+1}\right], \quad Z_{n+1} = I_K\otimes S_{n+1},\\
  &Y_{n+1/2}=e^{\frac{\tau}{2} A_{n+1}^{\dag}}\left[Y_{n+1}, \sqrt{0.5}W_{n+1}\right],\qquad S_{n+1/2}=\operatorname{diag}\left(S_{n+1},Z_{n+1}\right),\\
  &W_{n+1/2}=\left[\sqrt{\tau\gamma_1(t_{n+1/2})}L_1^{\dag}Y_{n+1/2},\ldots,\sqrt{\tau\gamma_K(t_{n+1/2})}L_K^{\dag}Y_{n+1/2}\right],
  \quad Z_{n+1/2} = I_K\otimes S_{n+1/2},\\
  & Y_{n}=\left[e^{\tau A_{n+1/2}^{\dag}}Y_{n+1}, \, e^{\frac{\tau}{2} A_{n+1/2}^{\dag}}W_{n+1/2}\right],\qquad S_{n}=\operatorname{diag}\left(S_{n+1},Z_{n+1/2}\right),
\end{align}
\end{subequations}
where we use the notation $\operatorname{diag}(\cdot)$ to describe a matrix in block-diagonal form and $I_K$ is the $K\times K$ identity matrix. Similarly, we need to apply column compression techniques to the factors $Y_n$ and $S_n$. Let $Y\in \mathbb{C}^{m\times p}$, and $S\in \mathbb{C}^{p\times p}$ be a Hermitian matrix, 
we denote with $[\widetilde{Y},\widetilde{S}]=\widetilde{\mathcal{T}}_{\varepsilon_1}(Y,S)$ the column compression procedure with error tolerance $\varepsilon_1$ such that
\begin{equation}\label{equ:column-compression}
  \left\|YSY^{\dag}-\widetilde{Y}\widetilde{S}\widetilde{Y}^{\dag}\right\|_1\leq\varepsilon_1,
\end{equation}
where $\widetilde{Y}\in \mathbb{C}^{m\times r}$, $\widetilde{S}\in \mathbb{C}^{r\times r}$ is a Hermitian matrix and $r\leq p$. To implement the column compression procedure \eqref{equ:column-compression}, we first let $Y=VR$ be the thin QR factorization, where $V^{\dag}V=I_p$ and $R\in \mathbb{C}^{p\times p}$. Let $RSR^{\dag}=U\Lambda U^{\dag}$
be the spectral decomposition, where $\Lambda=\operatorname{diag}(\lambda_1,\lambda_2,\ldots,\lambda_p)$ and $|\lambda_1|\geq|\lambda_2|\geq\ldots\geq|\lambda_p|$, then we define $\widetilde{Y}=VU_r$ and $\widetilde{S}=\Lambda_r$, where $U_r$ (resp. $\Lambda_r$) is the submatrix including the first $r$ columns of $U$ (resp. $\Lambda$) and the rank $r$ is chosen such that $\sum_{j=r+1}^p|\lambda_j|\leq\varepsilon_1$. In this setting, we have
\begin{eqnarray*}
  \left\|YSY^{\dag}-\widetilde{Y}\widetilde{S}\widetilde{Y}^{\dag}\right\|_1 &= &
  \left\|VRSR^{\dag}V^{\dag}-VU_r\Lambda_rU_r^{\dag}V^{\dag}\right\|_1\\
  &=&\left\|U\Lambda U^{\dag}-U_r\Lambda_rU_r^{\dag}\right\|_1=\sum_{j=r+1}^p|\lambda_j|\leq\varepsilon_1.
\end{eqnarray*}

Now, given terminal low-rank approximation $q_T\approx p_N=Y_NS_NY_N^{\dag}$ with $Y_N\in \mathbb{C}^{m\times \tilde{r}_N}$ and $S_N\in \mathbb{C}^{\tilde{r}_N\times \tilde{r}_N}$, and applying 
column compression and matrix exponential approximation to \eqref{equ:adjoint-factorization}, we can write the LREM scheme for the adjoint Lindblad equation as follows
\begin{subequations}\label{equ2.18}
  \begin{align}
  &\widetilde{W}_{n+1}=\left[\sqrt{\tau\gamma_1(t_{n+1})}L_1^{\dag}Y_{n+1},\ldots,\sqrt{\tau\gamma_K(t_{n+1})}L_K^{\dag}Y_{n+1}\right], \quad \widetilde{Z}_{n+1} = I_K\otimes S_{n+1},\\
  &\left[W_{n+1},Z_{n+1}\right]=\widetilde{\mathcal{T}}_{\varepsilon_1}\left(\widetilde{W}_{n+1},\widetilde{Z}_{n+1}\right),\\ &\widetilde{Y}_{n+1/2}=\mathfrak{e}^{\frac{\tau}{2} A_{n+1}^{\dag}}\left[Y_{n+1}, \sqrt{0.5}W_{n+1}\right],\qquad \widetilde{S}_{n+1/2}=\operatorname{diag}\left(S_{n+1},Z_{n+1}\right),\\
  & \left[Y_{n+1/2},S_{n+1/2}\right]=\widetilde{\mathcal{T}}_{\varepsilon_1}\left(\widetilde{Y}_{n+1/2},\widetilde{S}_{n+1/2}\right),\\ 
  &\widetilde{W}_{n+1/2}=\left[\sqrt{\tau\gamma_1(t_{n+1/2})}L_1^{\dag}Y_{n+1/2},\ldots,\sqrt{\tau\gamma_K(t_{n+1/2})}L_K^{\dag}Y_{n+1/2}\right],
  \quad \widetilde{Z}_{n+1/2} = I_K\otimes S_{n+1/2},\\
  & \left[W_{n+1/2},Z_{n+1/2}\right]=\widetilde{\mathcal{T}}_{\varepsilon_1}\left(\widetilde{W}_{n+1/2},\widetilde{Z}_{n+1/2}\right),\\ 
  & \widetilde{Y}_{n}=\left[\mathfrak{e}^{\tau A_{n+1/2}^{\dag}}Y_{n+1}, \, \mathfrak{e}^{\frac{\tau}{2} A_{n+1/2}^{\dag}}W_{n+1/2}\right],\qquad \widetilde{S}_{n}=\operatorname{diag}\left(S_{n+1},Z_{n+1/2}\right),\quad \left[Y_{n},S_{n}\right]=\widetilde{\mathcal{T}}_{\varepsilon_1}\left(\widetilde{Y}_{n},\widetilde{S}_{n}\right),
\end{align}
\end{subequations}
$n=N-1,\ldots,0$.

We remark that the LREM scheme \eqref{equ2.18} is equivalent to
\begin{subequations}\label{equ2.20}
  \begin{align}
  &\widetilde{p}_{n}=\Psi(t_{n+1},p_{n+1}),\label{equ2.20a}\\
  &p_{n}=\widetilde{p}_{n}-\theta_{n},\label{equ2.20b}
\end{align}
\end{subequations}
$n=N-1,\ldots,0$, where the matrix $\theta_{n}$ is the perturbation 
due to the approximations to the matrix exponential times vectors and the column compression procedures. 

\begin{rem}
The principal computational advantage of the low-rank exponential midpoint schemes lies in replacing operations on $m\times m$
density matrices by operations on factors of size $m\times r$, where typically $r\ll m$.
Consequently, the storage requirement is reduced from $\mathcal O(m^2)$ to $\mathcal O(mr)$, while the dominant matrix-vector
products involving the matrix exponential are also performed on $m\times r$ matrices instead of full density matrices.
The computational savings become increasingly significant as the Hilbert-space dimension grows.
\end{rem}

\section{Error analysis of the forward problem}\label{sec-err1}
In this section, we perform error analysis of the proposed FREM scheme \eqref{equ2.12} and LREM scheme \eqref{equ2.15} 
for the forward Lindblad equation. In the proofs, we will use the following result.
\begin{lemma}\label{lem3.1}
  (see \citep{Chen}) For any Hermitian matrix $\sigma\in \mathbb{C}^{m\times m}$, it holds that
  \begin{equation*}
    \left\|e^{tA(s)}\sigma e^{tA(s)^{\dag}}\right\|_1\leq \left\|\sigma\right\|_1,\qquad \forall t\geq0,~s\in[0,T].
  \end{equation*}
\end{lemma}

\subsection{Error estimate of the FREM scheme}
First we perform consistency analysis of the FREM scheme \eqref{equ2.7}. Considering \eqref{equ2.6} 
and using the consistency of the left-rectangle quadrature formula and midpoint quadrature formula, we have
\begin{subequations}\label{equ3.1}
\begin{align}
  &\rho(t_{n+1/2})=e^{\frac{\tau}{2} A_n} \left( \rho(t_n) + \frac{\tau}{2}\sum_{k=1}^K\gamma_k(t_n) \, L_k \, \rho(t_n) \,L_k^{\dag} \right) e^{\frac{\tau}{2} A_n^{\dag}} +\mathcal{O}(\tau^2),\label{equ3.1a}\\
    &\rho(t_{n+1})=e^{\tau A_{n+1/2}}  \rho(t_n)  e^{\tau A_{n+1/2}^{\dag}} +
  \tau \sum_{k=1}^K\gamma_k(t_{n+1/2}) e^{\frac{\tau}{2} A_{n+1/2}}  L_k  \rho(t_{n+1/2}) L_k^{\dag}  e^{\frac{\tau}{2} A_{n+1/2}^{\dag}}+\mathcal{O}(\tau^3).\label{equ3.1b}
\end{align}
\end{subequations}
Inserting \eqref{equ3.1a} into \eqref{equ3.1b} yields
\begin{equation}\label{equ3.2}
  \rho(t_{n+1})=\Phi(t_n,\rho(t_n)) + R_{n+1},
\end{equation}
and the truncation error $R_{n+1}$ satisfies
\begin{equation}\label{equ3.3}
  \max_{0\leq n\leq N-1}\left\|R_{n+1}\right\|_1\leq C_{\varrho} \, \tau^3,
\end{equation}
where the positive constant $C_{\varrho}$ depends on $L_k$, $A(t)$, $\gamma_k(t)$, $\rho(t)$ and their first and second order derivatives.

\begin{lemma}\label{lem3.2}
 For any Hermitian matrices $\rho_n,\varrho_n\in \mathbb{C}^{m\times m}$, it holds that
 \[ \left\|\Phi(t_n,\rho_n)-\Phi(t_n,\varrho_n)\right\|_1\leq \left(1+\tau C_L+\tau^2C_L^2/2\right) \, \left\|\rho_n-\varrho_n\right\|_1,\qquad n=0,\ldots,N-1,\]
 where $C_L:=\sum\limits_{k=1}^K\left(\max\limits_{0\leq t\leq T}\gamma_k(t)\right)\left\|L_k\right\|_1^2$.
\end{lemma}
\begin{proof}
We need to consider a single step of the FREM method \eqref{equ2.7}, applied at $t_n$ to the initial matrices $\rho_n$ and $\varrho_n$. 
We denote the intermediate values by $\rho_{n+1/2}$ and $\varrho_{n+1/2}$, respectively.

Considering the difference of the equations \eqref{equ2.7a} with respect to different initial values $\rho_n$ and $\varrho_n$ 
and using Lemma \ref{lem3.1}, we obtain
  \begin{eqnarray*}
    \|\rho_{n+1/2}-\varrho_{n+1/2}\|_1&\leq & \|\rho_n-\varrho_n\|_1+\frac{\tau }{2} \sum_{k=1}^{K}\gamma_k(t_n)\|L_k(\rho_{n}-\varrho_{n})L_k^{\dag}\|_1\\
    &\leq& (1+\frac{\tau}{2}C_L)\|\rho_n-\varrho_n\|_1.
  \end{eqnarray*}
  Similarly, using \eqref{equ2.7b}, Lemma \ref{lem3.1}, and noting that $\rho_{n+1}=\Phi(t_n,\rho_n)$ and $\varrho_{n+1}=\Phi(t_n,\varrho_n)$, we have
  \begin{eqnarray*}
   \|\Phi(t_n,\rho_n)-\Phi(t_n,\varrho_n)\|_1 &=& \|\rho_{n+1}-\varrho_{n+1}\|_1\\
   &\leq& \|\rho_n-\varrho_n\|_1+\tau \sum_{k=1}^{K}\gamma_k(t_{n+1/2})\|L_k(\rho_{n+1/2}-\varrho_{n+1/2})L_k^{\dag}\|_1\\
    &\leq&\|\rho_n-\varrho_n\|_1+\tau C_L \|\rho_{n+1/2}-\varrho_{n+1/2}\|_1\\
    &=&\left(1+\tau C_L+\tau^2 C_L^2/2\right)\|\rho_{n}-\varrho_{n}\|_1,
  \end{eqnarray*}
  which completes the proof.
\end{proof}

\begin{lemma}\label{lem3.3}
  Let $\sigma$ be Hermitian and positive semidefinite with unit trace, then it holds that
  \[1-C_{\varrho} \, \tau^3\leq\Tr(\Phi(t_n,\sigma))\leq 1+C_{\varrho} \, \tau^3,\qquad n=0,\ldots,N-1.\]
\end{lemma}
\begin{proof}
  Let $\varrho(t)$ be the solution of the Lindblad equation \eqref{equ2.2} with initial condition
  $\varrho(t_n)=\sigma$. It then follows that $\varrho(t)$ is Hermitian and positive semidefinite with 
  unit trace, i.e., $\|\varrho(t)\|_1=\Tr(\varrho(t))=1$ for all $t\geq t_n$. Note that
  $\Phi(t_n,\sigma)$ is the numerical approximation to $\varrho(t_{n+1})$ by using the FREM scheme \eqref{equ2.7} 
  for a single step with exact initial value $\varrho(t_n)$.
  By the consistency \eqref{equ3.2}-\eqref{equ3.3} of the FREM scheme \eqref{equ2.7}, we obtain that
  \[\|\Phi(t_n,\sigma)-\varrho(t_{n+1})\|_1\leq C_{\varrho} \, \tau^3.\]
  Using
 \[ \left|\|\Phi(t_n,\sigma)\|_1-\|\varrho(t_{n+1})\|_1\right|\leq\|\Phi(t_n,\sigma)-\varrho(t_{n+1})\|_1,\]
 we get
 \[\|\varrho(t_{n+1})\|_1-\|\Phi(t_n,\sigma)-\varrho(t_{n+1})\|_1\leq\|\Phi(t_n,\sigma)\|_1\leq
 \|\varrho(t_{n+1})\|_1+\|\Phi(t_n,\sigma)-\varrho(t_{n+1})\|_1.\]
  The desired result then follows from the positivity preserving property of the FREM scheme \eqref{equ2.7} 
  and $\|\varrho(t_{n+1})\|_1=\Tr(\varrho(t_{n+1}))=1$.
\end{proof}

Now we present the error estimate for the numerical solution derived from the unnormalized FREM 
scheme \eqref{equ2.7} for the forward Lindblad equation \eqref{equ2.2}.

\begin{theorem}\label{thm3.1}
  The numerical solution $\rho_n$ generated by the unnormalized FREM scheme \eqref{equ2.7} with $\rho_0=\rho(0)$ 
  satisfies the error estimate
  \[
  \|\rho(t_n)-\rho_n\|_1\leq C_1 \, \tau^2,\qquad 0\leq n\leq N,
  \]
  where $C_1=\frac{C_{\varrho}}{C_L}\left(e^{C_Lt_n}-1\right)$.
\end{theorem}
\begin{proof}
  Considering the difference between $\rho_{n+1}=\Phi(t_n,\rho_n)$ and \eqref{equ3.2}, 
  and using \eqref{equ3.3} and Lemma \ref{lem3.2}, we obtain
  \begin{eqnarray*}
    \|\rho(t_{n+1})-\rho_{n+1}\|_1&\leq& \|\Phi(t_n,\rho(t_n))-\Phi(t_n,\rho_n)\|_1+\|R_{n+1}\|_1 \\
    &\leq & \left(1+\tau C_L+\tau^2C_L^2/2\right)\|\rho(t_n)-\rho_n\|_1+C_{\varrho}\tau^3.
  \end{eqnarray*}
  By recursion, we obtain
  \begin{eqnarray*}
    \|\rho(t_{n})-\rho_{n}\|_1\leq \left(1+\tau C_L+\tau^2C_L^2/2\right)^n\|\rho(t_0)-\rho_0\|_1
    +C_{\varrho}\tau^3\sum_{j=0}^{n-1}\left(1+\tau C_L+\tau^2C_L^2/2\right)^j.
  \end{eqnarray*}
  Noting that $\rho(t_0)-\rho_0=0$, we have
  \begin{eqnarray*}
    \|\rho(t_{n})-\rho_{n}\|_1\leq \frac{C_{\varrho}}{C_L}\left(e^{C_Lt_n}-1\right)\tau^2,
  \end{eqnarray*}
  and the desired result follows.
\end{proof}

Now we are in the position to prove the convergence of the normalized FREM scheme \eqref{equ2.12}.

\begin{theorem}\label{thm3.2}
    The numerical solution $\rho_n$ generated by the normalized FREM scheme \eqref{equ2.12} with $\rho_0=\rho(0)$ 
    satisfies the error estimate
  \[
  \|\rho(t_n)-\rho_n\|_1\leq 2C_1 \, \tau^2,\qquad 0\leq n\leq N,
  \]
  where the constant $C_1$ is as defined in Theorem \ref{thm3.1}.
  \end{theorem}
  \begin{proof}
    We denote $\widehat{\rho}_{n+1}=\Phi(t_n,\widehat{\rho}_n)$ with $\widehat{\rho}_0=\rho(0)$, that means that $\widehat{\rho}_n$ is 
    the numerical solution generated by the unnormalized FREM scheme \eqref{equ2.7}.
    By Theorem \ref{thm3.1}, we have
   \begin{eqnarray*}
     \|\rho(t_{n+1})-\rho_{n+1}\|_1\leq\|\rho(t_{n+1})-\widehat{\rho}_{n+1}\|_1+\|\widehat{\rho}_{n+1}-\rho_{n+1}\|_1
     \leq C_1\tau^2+\|\widehat{\rho}_{n+1}-\rho_{n+1}\|_1.
   \end{eqnarray*}
   Using Lemmas \ref{lem3.2} and \ref{lem3.3} and noting that $\rho_n\geq0$ and $\|\rho_n\|_1=1$, we obtain
   \begin{eqnarray*}
     \|\widehat{\rho}_{n+1}-\rho_{n+1}\|_1&\leq&\|\widehat{\rho}_{n+1}-\widetilde{\rho}_{n+1}\|_1+\|\widetilde{\rho}_{n+1}-\rho_{n+1}\|_1\\
     &=& \|\Phi(t_n,\widehat{\rho}_n)-\Phi(t_n,\rho_n)\|_1+\|\rho_{n+1}(\Tr(\widetilde{\rho}_{n+1})-1)\|_1\\
     &\leq& \left(1+\tau C_L+\tau^2C_L^2/2\right)\|\widehat{\rho}_n-\rho_n\|_1+|\Tr(\Phi(t_n,\rho_n))-1|\cdot\|\rho_{n+1}\|_1\\
     &\leq&\left(1+\tau C_L+\tau^2C_L^2/2\right)\|\widehat{\rho}_n-\rho_n\|_1+C_{\varrho}\tau^3.
   \end{eqnarray*}
   By recursion and noting that $\widehat{\rho}_0=\rho_0=\rho(0)$, we obtain
  \begin{eqnarray*}
    \|\widehat{\rho}_{n+1}-\rho_{n+1}\|_1&\leq&\left(1+\tau C_L+\tau^2C_L^2/2\right)^{n+1}\|\hat{\rho}_0-\rho_0\|_1
    + C_{\varrho}\tau^3\sum_{j=0}^{n}\left(1+\tau C_L+\tau^2C_L^2/2\right)^{j}\\
    &\leq& C_1\tau^2,
  \end{eqnarray*}
  which completes the proof.
  \end{proof}

\subsection{Error estimate of the LREM scheme}
Now we consider error estimate of the proposed LREM scheme \eqref{equ2.15} for the forward Lindblad equation. 
First, we analyze the bound of perturbation $\vartheta_{n+1}$ (defined in \eqref{equ2.17b}) in the following lemma.

\begin{lemma}\label{lem3.4}
 Let $\varepsilon_1>0$ be the error tolerance of the column compression algorithm used in the LREM scheme \eqref{equ2.15}. 
 Assume that the matrix-exponential approximation used in the LREM scheme \eqref{equ2.15} satisfies 
 \[\left\|e^{\mu\tau A_{\nu}}\sigma e^{\mu\tau A_{\nu}^{\dag}}-\mathfrak{e}^{\mu\tau A_{\nu}}\sigma \mathfrak{e}^{\mu\tau A_{\nu}^{\dag}}\right\|_1\leq C_e\varepsilon_2,\] 
 for
\[
\mu\in\left\{\frac12,1\right\},
\qquad
\nu\in
\left\{
n, n+\frac12
\right\},
\qquad
n=0,\ldots,N-1,
\]
and for every Hermitian
\(\sigma\in\mathbb{C}^{m\times m}\),
 where $\varepsilon_2>0$ is the 
 corresponding error tolerance and $C_e>0$ is the error constant. Then it holds that
  \begin{equation*}
  \|\vartheta_{n+1}\|_1\leq \widetilde{c}_1 \, \varepsilon_1 + \widetilde{c}_2  \, \varepsilon_2,\qquad n=0,1,\ldots,N-1,
\end{equation*}
where $\widetilde{c}_1=2+3C_L\tau/2$ and $\widetilde{c}_2=3C_e+C_eC_L\tau(3+\tau)/2$.
\end{lemma}
\begin{proof}
  Let us first define
  \begin{equation*}
    \varphi_{n+1}=\mathfrak{e}^{\tau A_{n+1/2}}\varrho_n\mathfrak{e}^{\tau A_{n+1/2}^{\dag}}+
    \mathfrak{e}^{\frac{\tau}{2} A_{n+1/2}}G_{n+1/2}G_{n+1/2}^{\dag}\mathfrak{e}^{\frac{\tau}{2} A_{n+1/2}^{\dag}}.
  \end{equation*}
  By the definition of $\vartheta_{n+1}$, we have
  \begin{eqnarray}\label{equ3.4}
    \|\vartheta_{n+1}\|_1=\|\widetilde{\varrho}_{n+1}-\varphi_{n+1}+
    \varphi_{n+1}-\widehat{\varrho}_{n+1}\|_1
    \leq\|\widetilde{\varrho}_{n+1}-\varphi_{n+1}\|_1
    +\|\varphi_{n+1}-\widehat{\varrho}_{n+1}\|_1.
  \end{eqnarray}
  Note from \eqref{equ2.15} and \eqref{equ2.17} that $\varphi_{n+1}=\widetilde{X}_{n+1}\widetilde{X}_{n+1}^{\dag}$, $\widehat{\varrho}_{n+1}=\widehat{X}_{n+1}\widehat{X}_{n+1}^{\dag}$ and
  $\widehat{X}_{n+1}= \mathcal{T}_{\varepsilon_1}(\widetilde{X}_{n+1})$. It then follows that
  \begin{equation}\label{equ3.5}
    \|\varphi_{n+1}-\widehat{\varrho}_{n+1}\|_1=
    \left\|\widetilde{X}_{n+1}\widetilde{X}_{n+1}^{\dag}-\mathcal{T}_{\varepsilon_1}(\widetilde{X}_{n+1})
    \mathcal{T}_{\varepsilon_1}(\widetilde{X}_{n+1})^{\dag}\right\|_1\leq \varepsilon_1.
  \end{equation}

  Note from \eqref{equ2.17a} and \eqref{equ2.7} that
  \begin{equation*}
    \widetilde{\varrho}_{n+1}=e^{\tau A_{n+1/2}}\varrho_ne^{\tau A_{n+1/2}^{\dag}}
    +\tau e^{\frac{\tau}{2} A_{n+1/2}}\widehat{F}(t_{n+1/2},\widetilde{\varrho}_{n+1/2})e^{\frac{\tau}{2} A_{n+1/2}^{\dag}},
  \end{equation*}
  where
  \[\widetilde{\varrho}_{n+1/2}=e^{\frac{\tau}{2} A_n}\left(\varrho_n+\frac{\tau}{2} \widehat{F}(t_n,\varrho_n)\right)e^{\frac{\tau}{2} A_n^{\dag}},\quad \widehat{F}(t,\rho)=\sum_{k=1}^K\gamma_k(t) \, L_k \, \rho \,L_k^{\dag}.\]
  It follows that
  \begin{eqnarray}\label{equ3.6}
    \|\widetilde{\varrho}_{n+1}-\varphi_{n+1}\|_1&\leq& \underbrace{\left\|e^{\tau A_{n+1/2}}\varrho_ne^{\tau A_{n+1/2}^{\dag}}-
    \mathfrak{e}^{\tau A_{n+1/2}}\varrho_n\mathfrak{e}^{\tau A_{n+1/2}^{\dag}}\right\|_1}_{:=I_1} \nonumber\\
     &&+ \underbrace{
    \left\|e^{\frac{\tau}{2} A_{n+1/2}}\left(\tau\widehat{F}(t_{n+1/2},\widetilde{\varrho}_{n+1/2})-\widetilde{G}_{n+1/2}\widetilde{G}_{n+1/2}^{\dag}\right)e^{\frac{\tau}{2} A_{n+1/2}^{\dag}}\right\|_1}_{:=I_2} \nonumber \\
    &&+\underbrace{\left\|e^{\frac{\tau}{2} A_{n+1/2}}\widetilde{G}_{n+1/2}\widetilde{G}_{n+1/2}^{\dag}e^{\frac{\tau}{2} A_{n+1/2}^{\dag}}-\mathfrak{e}^{\frac{\tau}{2} A_{n+1/2}}\widetilde{G}_{n+1/2}\widetilde{G}_{n+1/2}^{\dag}\mathfrak{e}^{\frac{\tau}{2} A_{n+1/2}^{\dag}}\right\|_1}_{:=I_3}\nonumber\\
    &&+\underbrace{\left\|\mathfrak{e}^{\frac{\tau}{2} A_{n+1/2}}\left(\widetilde{G}_{n+1/2}\widetilde{G}_{n+1/2}^{\dag}-G_{n+1/2}G_{n+1/2}^{\dag}\right)\mathfrak{e}^{\frac{\tau}{2} A_{n+1/2}^{\dag}}\right\|_1}_{:=I_4}.
  \end{eqnarray}

  By the assumption on the matrix exponential algorithm, we obtain
\begin{equation}\label{equ3.7}
  I_1\leq C_e \, \varepsilon_2, \quad I_3\leq C_e \, \varepsilon_2.
\end{equation}
Using the following inequality
\begin{eqnarray*}
  \left\|\mathfrak{e}^{\frac{\tau}{2} A_{n+1/2}}\sigma\mathfrak{e}^{\frac{\tau}{2} A_{n+1/2}^{\dag}}\right\|_1
  &\leq& \left\|\mathfrak{e}^{\frac{\tau}{2} A_{n+1/2}}\sigma\mathfrak{e}^{\frac{\tau}{2} A_{n+1/2}^{\dag}}-
  e^{\frac{\tau}{2} A_{n+1/2}}\sigma e^{\frac{\tau}{2} A_{n+1/2}^{\dag}}\right\|_1 
  +\left\|e^{\frac{\tau}{2} A_{n+1/2}}\sigma e^{\frac{\tau}{2} A_{n+1/2}^{\dag}}\right\|_1\\
  &\leq& C_e\varepsilon_2+\|\sigma\|_1,
\end{eqnarray*}
and noting that $G_{n+1/2}=\mathcal{T}_{\varepsilon_1}(\widetilde{G}_{n+1/2})$, we obtain
\begin{equation}\label{equ3.8}
  I_4\leq C_e\varepsilon_2+\left\|\widetilde{G}_{n+1/2}\widetilde{G}_{n+1/2}^{\dag}-G_{n+1/2}G_{n+1/2}^{\dag}\right\|_1
  \leq C_e\varepsilon_2+\varepsilon_1.
\end{equation}

Note from \eqref{equ2.15} that $\widetilde{G}_{n+1/2}\widetilde{G}_{n+1/2}^{\dag}=\tau \widehat{F}(t_{n+1/2},X_{n+1/2}X_{n+1/2}^{\dag})$ 
and $X_{n+1/2}=\mathcal{T}_{\varepsilon_1}(\widetilde{X}_{n+1/2})$, we have
\begin{eqnarray}\label{equ3.9}
  I_2&\leq& \tau \left\|\widehat{F}(t_{n+1/2},\widetilde{\varrho}_{n+1/2})-\widehat{F}(t_{n+1/2},X_{n+1/2}X_{n+1/2}^{\dag})\right\|_1 \nonumber \\
  &\leq& \tau C_L\left\|\widetilde{\varrho}_{n+1/2}-X_{n+1/2}X_{n+1/2}^{\dag}\right\|_1\nonumber \\
  &\leq& \tau C_L\left\|\widetilde{\varrho}_{n+1/2}-\widetilde{X}_{n+1/2}\widetilde{X}_{n+1/2}^{\dag}\right\|_1
  +\tau C_L\left\|\widetilde{X}_{n+1/2}\widetilde{X}_{n+1/2}^{\dag}-X_{n+1/2}X_{n+1/2}^{\dag}\right\|_1\nonumber \\
  &\leq& \tau C_L\left\|e^{\frac{\tau}{2} A_n}\left(\varrho_n+\frac{\tau}{2} \widehat{F}(t_n,\varrho_n)\right)e^{\frac{\tau}{2} A_n^{\dag}}-
  \mathfrak{e}^{\frac{\tau}{2} A_n}\left(\varrho_n+\frac{1}{2} G_nG_n^{\dag}\right)\mathfrak{e}^{\frac{\tau}{2} A_n^{\dag}}\right\|_1+\tau C_L\varepsilon_1\nonumber \\
  &\leq&\frac{\tau C_L}{2}\left\|\tau e^{\frac{\tau}{2} A_n}\widehat{F}(t_n,\varrho_n)e^{\frac{\tau}{2} A_n^{\dag}}-\mathfrak{e}^{\frac{\tau}{2} A_n} G_nG_n^{\dag}\mathfrak{e}^{\frac{\tau}{2} A_n^{\dag}}\right\|_1 +\tau C_L(C_e\varepsilon_2+\varepsilon_1)\nonumber \\
  &\leq& \frac{\tau C_L}{2}\left\|\tau e^{\frac{\tau}{2} A_n}\widehat{F}(t_n,\varrho_n)e^{\frac{\tau}{2} A_n^{\dag}}-\tau\mathfrak{e}^{\frac{\tau}{2} A_n} \hat{F}(t_n,\varrho_n)\mathfrak{e}^{\frac{\tau}{2} A_n^{\dag}}\right\|_1 \nonumber\\
  &&+\frac{\tau C_L}{2}\left\|\mathfrak{e}^{\frac{\tau}{2} A_n} \left(\widetilde{G}_n\widetilde{G}_n^{\dag}-G_nG_n^{\dag}\right)\mathfrak{e}^{\frac{\tau}{2} A_n^{\dag}}\right\|_1+\tau C_L(C_e\varepsilon_2+\varepsilon_1)\nonumber \\
  &\leq& \frac{1}{2}\tau^2C_LC_e\varepsilon_2
  +\frac{1}{2}\tau C_L(C_e\varepsilon_2+\varepsilon_1)+\tau C_L(C_e\varepsilon_2+\varepsilon_1).
\end{eqnarray}
The desired result follows from \eqref{equ3.4}-\eqref{equ3.9}.
\end{proof}

Now we derive the error estimate of the LREM scheme \eqref{equ2.15}.
We assume that the initial low-rank approximation satisfies
\[\|\rho_0-\varrho_0\|_1\leq\delta,\]
for some $\delta>0$.

\begin{theorem}\label{thm3.3}
 Assume that the error tolerances of the column compression and matrix exponential algorithm satisfy $\varepsilon_1=\tau\epsilon_1$ 
 and $\varepsilon_2=\tau\epsilon_2$ for some $\epsilon_1,~\epsilon_2>0$, respectively. Let $\rho(t)$ be the solution of the 
 Lindblad equation \eqref{equ2.2} and $\{\varrho_n\}_{n=0}^N$ be the numerical solution generated by the LREM scheme \eqref{equ2.15}. 
 Then it holds that
  \[
  \|\rho(t_n)-\varrho_n\|_1\leq c_1 \, \tau^2+c_2 \, \delta+c_3 \, \epsilon_1+c_4 \, \epsilon_2,\qquad 0\leq n\leq N,
  \]
  where $c_1=C_{\varrho}(e^{C_Lt_n}-1)/C_L+C_1$, $c_2=e^{C_L t_{n}}$, $c_3=2\widetilde{c}_1(e^{C_Lt_n}-1)/C_L$ and $c_4=2\widetilde{c}_2(e^{C_Lt_n}-1)/C_L$.
\end{theorem}
\begin{proof}
   We first split the global error $\rho(t_{n+1})-\varrho_{n+1}$ as follows:
\begin{equation}\label{equ3.11}
  \rho(t_{n+1})-\varrho_{n+1}=(\rho(t_{n+1})-\rho_{n+1})+(\rho_{n+1}-\widecheck{\rho}_{n+1})+(\widecheck{\rho}_{n+1}-\varrho_{n+1}),
\end{equation}
where the auxiliary quantities $\rho_{n+1}$ and $\widecheck{\rho}_{n+1}$ are derived from the unnormalized FREM scheme \eqref{equ2.7} 
with initial value $\rho_0$ and low-rank initial value $\varrho_0$, respectively.  In other words,
  \begin{align}
    & \rho_{n+1}=\Phi(t_n,\rho_n),\qquad 0\leq n\leq N-1,\label{equ3.12}\\
    & \widecheck{\rho}_{n+1}=\Phi(t_n,\widecheck{\rho}_{n}),\qquad 0\leq n\leq N-1,\label{equ3.13}
  \end{align}
where $\widecheck{\rho}_0=\varrho_0$. Note that the first component $\rho(t_{n+1})-\rho_{n+1}$ in \eqref{equ2.11} 
denotes the global error of the unnormalized FREM scheme \eqref{equ2.7}. We apply Theorem \ref{thm3.1} to find
\begin{equation}\label{equ3.14}
  \|\rho(t_{n+1})-\rho_{n+1}\|_1\leq C_1 \, \tau^2.
\end{equation}

The second component $\rho_{n+1}-\widecheck{\rho}_{n+1}$ is the difference between the full-rank solutions 
with initial values $\rho_0$ and low-rank $\varrho_0$.
Subtracting \eqref{equ3.13} from \eqref{equ3.12} and applying Lemma \ref{lem3.2}, we obtain
\begin{eqnarray}\label{equ3.15}
  \|\rho_{n+1}-\widecheck{\rho}_{n+1}\|_1&=&\|\Phi(t_n,\rho_n)-\Phi(t_n,\widecheck{\rho}_n)\|_1
  \leq\left(1+\tau C_L +\tau^2C_L^2/2\right)\|\rho_n-\widecheck{\rho}_n\|_1 \nonumber\\
  &\leq&\left(1+\tau C_L +\tau^2C_L^2/2\right)^{n+1}\|\rho_0-\widecheck{\rho}_0\|_1\leq e^{C_L t_{n+1}} \, \delta.
\end{eqnarray}

The third component $\widecheck{\rho}_{n+1}-\varrho_{n+1}$ in \eqref{equ2.11} is the difference of the solutions obtained 
with the FREM scheme \eqref{equ2.7} and the LREM scheme \eqref{equ2.15} with the same low-rank initial value $\varrho_0$.
By using \eqref{equ2.17}, Lemmas \ref{lem3.2}, \ref{lem3.3} and \ref{lem3.4}, we get
\begin{eqnarray*}
  \|\widehat{\varrho}_{n+1}-\varrho_{n+1}\|_1&=&\|(\Tr(\widehat{\varrho}_{n+1})-1)\varrho_{n+1}\|_1
  =|\Tr(\widehat{\varrho}_{n+1})-1|\|\varrho_{n+1}\|_1 \\
  &=& |\Tr(\widehat{\varrho}_{n+1})-1|=|\Tr(\widetilde{\varrho}_{n+1})-\Tr(\vartheta_{n+1})-1|\\
  &\leq&|\Tr(\widetilde{\varrho}_{n+1})-1|+|\Tr(\vartheta_{n+1})|\leq \|\vartheta_{n+1}\|_1+ C_{\varrho}\tau^3,
\end{eqnarray*}
and it follows that
\begin{eqnarray}\label{equ3.16}
  \|\widecheck{\rho}_{n+1}-\varrho_{n+1}\|_1&\leq& \|\widecheck{\rho}_{n+1}-\widetilde{\varrho}_{n+1}\|_1
  +\|\widetilde{\varrho}_{n+1}-\hat{\varrho}_{n+1}\|_1+\|\widehat{\varrho}_{n+1}-\varrho_{n+1}\|_1\nonumber\\
  &=&\|\Phi(t_n,\widecheck{\rho}_n)-\Phi(t_n,\varrho_n)\|_1+2\|\vartheta_{n+1}\|_1+C_{\varrho} \tau^3 \nonumber\\
  &\leq&\left(1+\tau C_L+\tau^2C_L^2/2\right)\|\widecheck{\rho}_n-\varrho_n\|_1+2\tau(\widetilde{c}_1\epsilon_1+\widetilde{c}_2\epsilon_2)+C_{\varrho}\tau^3.
\end{eqnarray}
By recursion and noting that $\widecheck{\rho}_0=\varrho_0$, we obtain
\begin{eqnarray}\label{equ3.17}
    \|\widecheck{\rho}_{n+1}-\varrho_{n+1}\|_1&\leq&\left(1+\tau C_L+\tau^2C_L^2/2\right)^{n+1}\|\widecheck{\rho}_0-\varrho_0\|_1\nonumber\\
    &&+\left(2\widetilde{c}_1\tau\epsilon_1+2\widetilde{c}_2\tau\epsilon_2+C_{\varrho}\tau^3\right)\sum_{j=0}^{n}\left(1+\tau C_L+\tau^2C_L^2/2\right)^{j} \nonumber\\
    &\leq& \frac{C_{\varrho}}{C_L}(e^{C_Lt_{n+1}}-1)\tau^2 + \frac{ 2\widetilde{c}_1}{C_L}(e^{C_Lt_{n+1}}-1)\epsilon_1+ \frac{ 2\widetilde{c}_2}{C_L}(e^{C_Lt_{n+1}}-1)\epsilon_2.
  \end{eqnarray}
   Combining \eqref{equ3.11}, \eqref{equ3.14},
  \eqref{equ3.15} and \eqref{equ3.17} completes the proof.
\end{proof}

\section{Error analysis of the adjoint problem}\label{sec-err2}
In this section, we consider the error estimates of the FREM scheme \eqref{equ2.11} and the LREM scheme \eqref{equ2.18} 
for the adjoint Lindblad equation \eqref{equ2.3}.

\subsection{Error estimate of the FREM scheme}
We first consider the consistency of the FREM scheme \eqref{equ2.11}. Using error estimates of the basic right-rectangle quadrature 
formula and midpoint quadrature formula, it then follows from \eqref{equ2.10} that
\begin{subequations}\label{equ4.4}
\begin{align}
  &q(t_{n+1/2})=e^{\frac{\tau}{2} A_{n+1}^{\dag}} \left( q(t_{n+1}) + \frac{\tau}{2}\sum_{k=1}^K\gamma_k(t_{n+1}) \, L_k^{\dag} \, q(t_{n+1}) \,L_k \right) e^{\frac{\tau}{2} A_{n+1}}+\mathcal{O}(\tau^2) ,\label{equ4.4a}\\
    &q(t_{n})=e^{\tau A_{n+1/2}^{\dag}}  q(t_{n+1})  e^{\tau A_{n+1/2}} +
  \tau \sum_{k=1}^K\gamma_k(t_{n+1/2}) e^{\frac{\tau}{2} A_{n+1/2}^{\dag}}  L_k^{\dag}  q(t_{n+1/2}) L_k e^{\frac{\tau}{2} A_{n+1/2}}+\mathcal{O}(\tau^3).\label{equ4.4b}
\end{align}
\end{subequations}
Inserting \eqref{equ4.4a} into \eqref{equ4.4b} yields
\begin{equation}\label{equ4.5}
  q(t_{n})=\Psi(t_{n+1},q(t_{n+1}))+\widetilde{R}_n,
\end{equation}
and the truncation error $\widetilde{R}_n$ satisfies
\begin{equation}\label{equ4.6}
  \max_{0\leq n\leq N-1}\|\widetilde{R}_n\|_1\leq C_p \, \tau^3,
\end{equation}
where the positive constant $C_p$ depends on $L_k$, $A(t)$, $\gamma_k(t)$, $q(t)$ and their first and second order derivatives.

\begin{lemma}\label{lem4.3}
 For any Hermitian matrices $q_{n+1},p_{n+1}\in \mathbb{C}^{m\times m}$, it holds that
 \[ \|\Psi(t_{n+1},q_{n+1})-\Psi(t_{n+1},p_{n+1})\|_1\leq \left(1+\tau C_L+\tau^2C_L^2/2\right) \, \|q_{n+1}-p_{n+1}\|_1,\]
 $n=N-1,\ldots,0$.
\end{lemma}
\begin{proof}
  Considering the difference of the equations \eqref{equ2.11a} with respect to different initial values $q_{n+1}$ and $p_{n+1}$ 
  and using Lemma \ref{lem3.1}, we get
  \begin{eqnarray*}
    \|q_{n+1/2}-p_{n+1/2}\|_1&\leq & \|q_{n+1}-p_{n+1}\|_1+\frac{\tau }{2} \sum_{k=1}^{K}\gamma_k(t_{n+1})\|L_k^{\dag}(q_{n+1}-p_{n+1})L_k\|_1\\
    &\leq& (1+\frac{\tau}{2}C_L)\|q_{n+1}-p_{n+1}\|_1.
  \end{eqnarray*}
  Similarly, using \eqref{equ2.11b}, Lemma \ref{lem3.1}, $q_{n}=\Psi(t_{n+1},q_{n+1})$ and $p_{n}=\Psi(t_{n+1},p_{n+1})$, we have
  \begin{eqnarray*}
  \|q_{n}-p_{n}\|_1 &=&\|\Psi(t_{n+1},q_{n+1})-\Psi(t_{n+1},p_{n+1})\|_1\\
   &\leq& \|q_{n+1}-p_{n+1}\|_1+\tau \sum_{k=1}^{K}\gamma_k(t_{n+1/2})\|L_k^{\dag}(q_{n+1/2}-p_{n+1/2})L_k\|_1\\
    &\leq&\|q_{n+1}-p_{n+1}\|_1+\tau C_L \|q_{n+1/2}-p_{n+1/2}\|_1\\
    &=&\left(1+\tau C_L+\tau^2 C_L^2/2\right)\|q_{n+1}-p_{n+1}\|_1,
  \end{eqnarray*}
  which completes the proof.
\end{proof}

The following result shows the second-order convergence of the FREM scheme \eqref{equ2.11}.
\begin{theorem}\label{thm4.1}
  The numerical solution $q_n$ generated by the FREM scheme \eqref{equ2.11} with $q_N=q_T$ satisfies the error estimate
  \[\|q(t_n)-q_n\|_1\leq \widetilde{C}_1 \, \tau^2,\qquad 0\leq n\leq N,\]
  where $\widetilde{C}_1=\left(e^{C_L(T-t_n)}-1\right)C_p/C_L$.
\end{theorem}
\begin{proof}
  Subtracting $q_n=\Psi(t_{n+1},q_{n+1})$ from \eqref{equ4.5}, applying \eqref{equ4.6} and Lemma \ref{lem4.3} yields
  \begin{eqnarray*}
    \|q(t_n)-q_n\|_1&\leq&\|\Psi(t_{n+1},q(t_{n+1}))-\Psi(t_{n+1},q_{n+1})\|_1+\|\widetilde{R}_n\|_1\\
    &\leq& \left(1+\tau C_L+\tau^2C_L^2/2\right)\|q(t_{n+1})-q_{n+1}\|_1 + C_p\tau^3.
  \end{eqnarray*}
  Then we have
  \begin{equation*}
    \|q(t_n)-q_n\|_1\leq \left(1+\tau C_L+\tau^2C_L^2/2\right)^{N-n}\|q(T)-q_N\|_1
    +C_p\tau^3\sum_{j=0}^{N-n-1}\left(1+\tau C_L+\tau^2C_L^2/2\right)^j.
  \end{equation*}
  Noting that $q_N=q(T)$, so that the statement holds.
\end{proof}

\subsection{Error estimate of the LREM scheme}
Our next aim is to estimate the error of the LREM scheme \eqref{equ2.18} for the adjoint Lindblad equation. 
First, we present the following result concerning the bound of $\theta_n$ as defined in \eqref{equ2.20}.

\begin{lemma}\label{lem4.5}
 Let $\varepsilon_1>0$ be the error tolerance of the column compression algorithm used in the LREM scheme \eqref{equ2.18}. 
 Assume that the matrix-exponential approximation used in the LREM
scheme \eqref{equ2.18} satisfies
\[
\left\|
e^{\mu\tau A_{\nu}^{\dagger}}
\sigma
e^{\mu\tau A_{\nu}}
-
\mathfrak{e}^{\mu\tau A_{\nu}^{\dagger}}
\sigma
\mathfrak{e}^{\mu\tau A_{\nu}}
\right\|_1
\leq
\widetilde{C}_e\varepsilon_2,
\]
for
\[
\mu\in\left\{\frac12,1\right\},
\qquad
\nu\in
\left\{
n+\frac12,n+1
\right\},
\qquad
n=0,\ldots,N-1,
\]
and for every Hermitian
\(\sigma\in\mathbb{C}^{m\times m}\), where $\varepsilon_2>0$ is the
   corresponding error tolerance and $\widetilde{C}_e>0$ is the error constant. Then it holds that
  \begin{equation*}
  \|\theta_{n}\|_1\leq \widehat{c}_1 \, \varepsilon_1 + \widehat{c}_2  \, \varepsilon_2,\qquad n=0,1,\ldots,N-1,
\end{equation*}
where
$\widehat{c}_1=2+3C_L\tau/2$ and $\widehat{c}_2=3\widetilde{C}_e+\widetilde{C}_eC_L\tau(3+\tau)/2$.
\end{lemma}
\begin{proof}
  With the notation
  \begin{equation*}
    \phi_{n}=\mathfrak{e}^{\tau A_{n+1/2}^{\dag}}p_{n+1}\mathfrak{e}^{\tau A_{n+1/2}}+
    \mathfrak{e}^{\frac{\tau}{2} A_{n+1/2}^{\dag}}W_{n+1/2}Z_{n+1/2}W_{n+1/2}^{\dag}\mathfrak{e}^{\frac{\tau}{2} A_{n+1/2}},
  \end{equation*}
  and the triangle inequality, we have
  \begin{eqnarray}\label{equ4.7}
    \|\theta_{n}\|_1=\|\widetilde{p}_{n}-\phi_{n}+
    \phi_{n}-p_{n}\|_1
    \leq\|\widetilde{p}_{n}-\phi_{n}\|_1
    +\|\phi_{n}-p_{n}\|_1.
  \end{eqnarray}
  Note from \eqref{equ2.18} and \eqref{equ2.20} that $\phi_n=\widetilde{Y}_n\widetilde{S}_n\widetilde{Y}_n^{\dag}$,
  $p_n=Y_nS_nY_n^{\dag}$ and $\left[Y_n,S_n\right]=\widetilde{\mathcal{T}}_{\varepsilon_1}\left(\widetilde{Y}_n,\widetilde{S}_n\right)$. We then have
  \begin{equation}\label{equ4.8}
    \|\phi_{n}-p_{n}\|_1=\left\|\widetilde{Y}_n\widetilde{S}_n\widetilde{Y}_n^{\dag}-Y_nS_nY_n^{\dag}\right\|_1
    \leq \varepsilon_1.
  \end{equation}

  Since $\widetilde{p}_n=\Psi(t_{n+1},p_{n+1})$, straightforward calculation shows that
  \begin{equation*}
    \widetilde{p}_n=e^{\tau A_{n+1/2}^{\dag}}p_{n+1}e^{\tau A_{n+1/2}}
    +\tau e^{\frac{\tau}{2} A_{n+1/2}^{\dag}}\widecheck{F}(t_{n+1/2},\widetilde{p}_{n+1/2})e^{\frac{\tau}{2} A_{n+1/2}},
  \end{equation*}
  where
  \[\widetilde{p}_{n+1/2}=e^{\frac{\tau}{2} A_{n+1}^{\dag}}\left(p_{n+1}+\frac{\tau}{2} \widecheck{F}(t_{n+1},p_{n+1})\right)e^{\frac{\tau}{2} A_{n+1}},\quad \widecheck{F}(t,p)=\sum_{k=1}^K\gamma_k(t) \, L_k^{\dag} \, p \,L_k.\]
 Then by the triangle inequality we get
  \begin{eqnarray}\label{equ4.9}
   \|\widetilde{p}_{n}-\phi_{n}\|_1 &\leq& \underbrace{\left\|e^{\tau A_{n+1/2}^{\dag}}p_{n+1}e^{\tau A_{n+1/2}}-
    \mathfrak{e}^{\tau A_{n+1/2}^{\dag}}p_{n+1}\mathfrak{e}^{\tau A_{n+1/2}}\right\|_1}_{:=I_1} \nonumber\\
     &&+ \underbrace{
    \left\|e^{\frac{\tau}{2} A_{n+1/2}^{\dag}}\left(\tau\widecheck{F}(t_{n+1/2},\widetilde{p}_{n+1/2})-\widetilde{W}_{n+1/2}\widetilde{Z}_{n+1/2}\widetilde{W}_{n+1/2}^{\dag}\right)e^{\frac{\tau}{2} A_{n+1/2}}\right\|_1}_{:=I_2} \nonumber \\
    &&+\underbrace{\left\|e^{\frac{\tau}{2} A_{n+1/2}^{\dag}}\widetilde{W}_{n+1/2}\widetilde{Z}_{n+1/2}\widetilde{W}_{n+1/2}^{\dag}e^{\frac{\tau}{2} A_{n+1/2}}-\mathfrak{e}^{\frac{\tau}{2} A_{n+1/2}^{\dag}}\widetilde{W}_{n+1/2}\widetilde{Z}_{n+1/2}\widetilde{W}_{n+1/2}^{\dag}\mathfrak{e}^{\frac{\tau}{2} A_{n+1/2}}\right\|_1}_{:=I_3}\nonumber\\
    &&+\underbrace{\left\|\mathfrak{e}^{\frac{\tau}{2} A_{n+1/2}^{\dag}}\left(\widetilde{W}_{n+1/2}\widetilde{Z}_{n+1/2}\widetilde{W}_{n+1/2}^{\dag}-W_{n+1/2}Z_{n+1/2}W_{n+1/2}^{\dag}\right)\mathfrak{e}^{\frac{\tau}{2} A_{n+1/2}}\right\|_1}_{:=I_4}.
  \end{eqnarray}

 By the assumption on the matrix exponential algorithm, we have
\begin{equation}\label{equ4.10}
  I_1\leq \widetilde{C}_e \, \varepsilon_2,\quad I_3\leq \widetilde{C}_e \, \varepsilon_2.
\end{equation}
Note that
\begin{eqnarray*}
  \left\|\mathfrak{e}^{\frac{\tau}{2} A_{n+1/2}^{\dag}}\sigma\mathfrak{e}^{\frac{\tau}{2} A_{n+1/2}}\right\|_1
  &\leq& \left\|\mathfrak{e}^{\frac{\tau}{2} A_{n+1/2}^{\dag}}\sigma\mathfrak{e}^{\frac{\tau}{2} A_{n+1/2}}-
  e^{\frac{\tau}{2} A_{n+1/2}^{\dag}}\sigma e^{\frac{\tau}{2} A_{n+1/2}}\right\|_1+\left\|e^{\frac{\tau}{2} A_{n+1/2}^{\dag}}\sigma e^{\frac{\tau}{2} A_{n+1/2}}\right\|_1\\
  &\leq& \widetilde{C}_e\varepsilon_2 + \|\sigma\|_1,
\end{eqnarray*}
this combines with $\left[W_{n+1/2},Z_{n+1/2}\right]=\widetilde{\mathcal{T}}_{\varepsilon_1}\left(\widetilde{W}_{n+1/2},\widetilde{Z}_{n+1/2}\right)$ gives
\begin{equation}\label{equ4.11}
  I_4\leq \widetilde{C}_e\varepsilon_2+\left\|\widetilde{W}_{n+1/2}\widetilde{Z}_{n+1/2}\widetilde{W}_{n+1/2}^{\dag}-W_{n+1/2}Z_{n+1/2}W_{n+1/2}^{\dag}\right\|_1
  \leq \widetilde{C}_e\varepsilon_2+\varepsilon_1.
\end{equation}

We see from \eqref{equ2.18} that 
\[\widetilde{W}_{n+1/2}\widetilde{Z}_{n+1/2}\widetilde{W}_{n+1/2}^{\dag}=\tau \widecheck{F}\left(t_{n+1/2},Y_{n+1/2}S_{n+1/2}Y_{n+1/2}^{\dag}\right), \quad
 \left[Y_{n+1/2},S_{n+1/2}\right]=\mathcal{T}_{\varepsilon_1}\left(\widetilde{Y}_{n+1/2},\widetilde{S}_{n+1/2}\right),\]
  then we have
\begin{eqnarray}\label{equ4.12}
  I_2&\leq& \tau \left\|\widecheck{F}(t_{n+1/2},\widetilde{p}_{n+1/2})-\widecheck{F}(t_{n+1/2},Y_{n+1/2}S_{n+1/2}Y_{n+1/2}^{\dag})\right\|_1 \nonumber \\
  &\leq& \tau C_L\left\|\widetilde{p}_{n+1/2}-Y_{n+1/2}S_{n+1/2}Y_{n+1/2}^{\dag}\right\|_1\nonumber \\
  &\leq& \tau C_L\left\|\widetilde{p}_{n+1/2}-\widetilde{Y}_{n+1/2}\widetilde{S}_{n+1/2}\widetilde{Y}_{n+1/2}^{\dag}\right\|_1
  +\tau C_L\left\|\widetilde{Y}_{n+1/2}\widetilde{S}_{n+1/2}\widetilde{Y}_{n+1/2}^{\dag}-Y_{n+1/2}S_{n+1/2}Y_{n+1/2}^{\dag}\right\|_1\nonumber \\
  &\leq& \tau C_L\left\|e^{\frac{\tau}{2} A_{n+1}^{\dag}}(p_{n+1}+\frac{\tau}{2} \widecheck{F}(t_{n+1},p_{n+1}))e^{\frac{\tau}{2} A_{n+1}}-
  \mathfrak{e}^{\frac{\tau}{2} A_{n+1}^{\dag}}(p_{n+1}+\frac{1}{2} W_{n+1}Z_{n+1}W_{n+1}^{\dag})\mathfrak{e}^{\frac{\tau}{2} A_{n+1}}\right\|_1+\tau C_L\varepsilon_1\nonumber \\
  &\leq&\frac{\tau C_L}{2}\left\|\tau e^{\frac{\tau}{2} A_{n+1}^{\dag}}\widecheck{F}(t_{n+1},p_{n+1})e^{\frac{\tau}{2} A_{n+1}}-\mathfrak{e}^{\frac{\tau}{2} A_{n+1}^{\dag}} W_{n+1}Z_{n+1}W_{n+1}^{\dag}\mathfrak{e}^{\frac{\tau}{2} A_{n+1}}\right\|_1 +\tau C_L(\widetilde{C}_e\varepsilon_2+\varepsilon_1)\nonumber \\
  &\leq& \frac{\tau C_L}{2}\left\|\tau e^{\frac{\tau}{2} A_{n+1}^{\dag}}\widecheck{F}(t_{n+1},p_{n+1})e^{\frac{\tau}{2} A_{n+1}}-\tau\mathfrak{e}^{\frac{\tau}{2} A_{n+1}^{\dag}} \widecheck{F}(t_{n+1},p_{n+1})\mathfrak{e}^{\frac{\tau}{2} A_{n+1}}\right\|_1 \nonumber\\
  &&+\frac{\tau C_L}{2}\left\|\mathfrak{e}^{\frac{\tau}{2} A_{n+1}^{\dag}} \left(\widetilde{W}_{n+1}\widetilde{Z}_{n+1}\widetilde{W}_{n+1}^{\dag}-W_{n+1}Z_{n+1}W_{n+1}^{\dag}\right)\mathfrak{e}^{\frac{\tau}{2} A_{n+1}}\right\|_1+\tau C_L(\widetilde{C}_e\varepsilon_2+\varepsilon_1)\nonumber \\
  &\leq& \frac{1}{2}\tau^2C_L\widetilde{C}_e\varepsilon_2
  +\frac{1}{2}\tau C_L(\varepsilon_1+\widetilde{C}_e\varepsilon_2)+\tau C_L(\widetilde{C}_e\varepsilon_2+\varepsilon_1).
\end{eqnarray}
The desired result follows from \eqref{equ4.7}-\eqref{equ4.12}.
\end{proof}

Finally, we prove a convergence result for the LREM scheme \eqref{equ2.18}.
We assume that the terminal low-rank approximation satisfies
\[\|p_N-p_T\|_1\leq\delta,\]
for some $\delta>0$.

\begin{theorem}\label{thm4.3}
 Assume that the error tolerances of the column compression and matrix exponential algorithm satisfy $\varepsilon_1=\tau\epsilon_1$ 
 and $\varepsilon_2=\tau\epsilon_2$ for some $\epsilon_1,~\epsilon_2>0$, respectively. Let $q(t)$ be the solution of the adjoint Lindblad 
 equation \eqref{equ2.3} and $\{p_n\}_{n=0}^N$ be the numerical solution generated by the LREM scheme \eqref{equ2.18}. Then it holds that
  \[\|q(t_n)-p_n\|_1\leq \widecheck{c}_1 \, \tau^2+\widecheck{c}_2 \, \delta+\widecheck{c}_3 \, \epsilon_1+\widecheck{c}_4 \, \epsilon_2,\qquad 0\leq n\leq N,\]
  where $\widecheck{c}_1=\widetilde{C}_1$, $\widecheck{c}_2=e^{C_L(T-t_n)}$,  $\widecheck{c}_3=\widehat{c}_1\left(e^{C_L(T-t_n)}-1\right)/C_L$ 
 and $\widecheck{c}_4=\widehat{c}_2\left(e^{C_L(T-t_n)}-1\right)/C_L$.
\end{theorem}
\begin{proof}
  We split the global error $q(t_{n})-p_{n}$ as follows:
\begin{equation}\label{equ4.14}
  q(t_{n})-p_{n}=(q(t_{n})-q_{n})+(q_{n}-\widecheck{q}_{n})+(\widecheck{q}_{n}-p_{n}),
\end{equation}
where the auxiliary quantities $q_{n}$ and $\widecheck{q}_{n}$ are obtained from the FREM scheme \eqref{equ2.11} 
with terminal value $q_T$ and low-rank terminal value $p_N$, respectively.  In other words, we have
  \begin{align}
    & q_{n}=\Psi(t_{n+1},q_{n+1}),\qquad n=N-1,\ldots,0, \label{equ4.15}\\
    & \widecheck{q}_{n}=\Psi(t_{n+1},\widecheck{q}_{n+1}),\qquad n=N-1,\ldots,0,\label{equ4.16}
  \end{align}
where $q_N=q_T$ and $\widecheck{q}_N=p_N$. Note that the first component $q(t_{n})-q_{n}$ in \eqref{equ4.14} denotes the global error of the FREM scheme \eqref{equ2.11}. Therefore, applying Theorem \ref{thm4.1} yeilds
\begin{equation}\label{equ4.17}
  \|q(t_{n})-q_{n}\|_1\leq \widetilde{C}_1 \, \tau^2.
\end{equation}

The second component $q_{n}-\widecheck{q}_{n}$ in \eqref{equ4.14} is the difference between the full-rank solutions 
with terminal values $q_T$ and low-rank $p_N$.
Subtracting \eqref{equ4.16} from \eqref{equ4.15} and applying Lemma \ref{lem4.3} gives
\begin{eqnarray}\label{equ4.18}
  \|q_{n}-\widecheck{q}_{n}\|_1&=&\|\Psi(t_{n+1},q_{n+1})-\Psi(t_{n+1},\widecheck{q}_{n+1})\|_1
  \leq\left(1+\tau C_L +\tau^2C_L^2/2\right)\|q_{n+1}-\widecheck{q}_{n+1}\|_1 \nonumber\\
  &\leq&\left(1+\tau C_L +\tau^2C_L^2/2\right)^{N-n}\|q_N-\widecheck{q}_N\|_1\leq e^{C_L(T-t_n)} \, \delta.
\end{eqnarray}

The third component $\widecheck{q}_{n}-p_{n}$ in \eqref{equ4.14} is the difference of the solutions obtained with the 
FREM scheme \eqref{equ2.11} and the LREM scheme \eqref{equ2.18} with the same low-rank terminal value $p_N$.
By the triangle inequality we get
\begin{eqnarray}\label{equ4.19}
  \|\widecheck{q}_{n}-p_{n}\|_1&\leq& \|\widecheck{q}_{n}-\widetilde{p}_{n}\|_1
  +\|\widetilde{p}_{n}-p_{n}\|_1,
\end{eqnarray}
where $\widetilde{p}_{n}$ is as defined in \eqref{equ2.20}.
 Applying \eqref{equ4.19} and Lemma \ref{lem4.5}, we obtain
 \begin{eqnarray}\label{equ4.21}
   \|\widecheck{q}_{n}-p_{n}\|_1&\leq& \|\Psi(t_{n+1},\widecheck{q}_{n+1})-\Psi(t_{n+1},p_{n+1})\|_1
   +\|\theta_n\|_1 \nonumber\\
   &\leq& \left(1+\tau C_L+\tau^2C_L^2/2\right)\|\widecheck{q}_{n+1}-p_{n+1}\|_1+\tau\left(\widehat{c}_1\epsilon_1
   +\widehat{c}_2\epsilon_2\right) \nonumber\\
   &\leq&\left(1+\tau C_L+\tau^2C_L^2/2\right)^{N-n}\|\widecheck{q}_{N}-p_{N}\|_1+ \widecheck{c}_3\epsilon_1+\check{c}_4\epsilon_2\nonumber\\
   &=& \widecheck{c}_3\epsilon_1+\widecheck{c}_4\epsilon_2,
 \end{eqnarray}
 where $\widecheck{c}_3=\widehat{c}_1\left(e^{C_L(T-t_n)}-1\right)/C_L$ 
 and $\widecheck{c}_4=\widehat{c}_2\left(e^{C_L(T-t_n)}-1\right)/C_L$. By combining \eqref{equ4.14}, \eqref{equ4.17}, \eqref{equ4.18} and \eqref{equ4.21} 
 we complete the proof.
\end{proof}

\begin{corollary}
\label{cor:adjoint-second-order-lrem}
Under the assumptions of Theorem~\ref{thm4.3}, suppose in addition that
the terminal low-rank approximation error and the scaled approximation
parameters satisfy
\[
\delta \leq C_\delta \tau^2,
\qquad
\epsilon_1 \leq C_{\epsilon_1}\tau^2,
\qquad
\epsilon_2 \leq C_{\epsilon_2}\tau^2,
\]
with constants \(C_\delta,C_{\epsilon_1},C_{\epsilon_2}\) independent of
\(\tau\). Then the low-rank adjoint approximation generated by
\eqref{equ2.18} satisfies
\[
\|q(t_n)-p_n\|_1
\leq
C\,\tau^2,
\qquad
0\leq n\leq N,
\]
where \(C>0\) is independent of \(\tau\). Equivalently, since 
$\varepsilon_1=\tau\epsilon_1$ and $\varepsilon_2=\tau\epsilon_2$, 
second-order convergence of the complete LREM approximation is obtained
whenever
\[
\delta=O(\tau^2),
\qquad
\varepsilon_1=O(\tau^3),
\qquad
\varepsilon_2=O(\tau^3).
\]
\end{corollary}
\begin{proof}
The estimate of Theorem~\ref{thm4.3} gives 
$ \|q(t_n)-p_n\|_1
\leq
\widecheck{c}_1\tau^2
+
\widecheck{c}_2\delta
+
\widecheck{c}_3\epsilon_1
+
\widecheck{c}_4\epsilon_2
$. 
Using the assumed bounds on \(\delta\), \(\epsilon_1\), and \(\epsilon_2\),
we obtain
\[
\|q(t_n)-p_n\|_1
\leq
\left(
\widecheck{c}_1
+
\widecheck{c}_2 C_\delta
+
\widecheck{c}_3 C_{\epsilon_1}
+
\widecheck{c}_4 C_{\epsilon_2}
\right)\tau^2,
\]
which proves the result.
\end{proof}

The theoretical results established above provide the foundation for the numerical methods proposed in our work. 
The following experiments are designed to assess their performance from three complementary
perspectives, that is, convergence, preservation of the physical structure of the Lindblad dynamics, 
and applicability within adjoint-based optimal control methods.

\section{Numerical experiments}
\label{sec-experiments}

This section reports numerical experiments supporting the analytical
and computational developments of the paper. First, the convergence,
structure-preserving properties, and computational performance of the
full- and low-rank exponential midpoint propagators are examined.
Subsequently, these propagators are employed within the SQH framework
for smooth and nonsmooth optimal control problems.

The propagation and optimal control experiments were implemented in Python~3.12.4 using
the matrix exponential routines provided by \emph{SciPy}. All computations
were performed on a laptop equipped with an Intel Core i7-8565U
processor at \(1.80\,\mathrm{GHz}\) and \(16\,\mathrm{GB}\) of RAM.

Throughout the experiments, an \(X\)-\(X\) Ising-chain Hamiltonian,
also considered in \citep{Chen}, was used:
\begin{equation}
\label{eq:numerical-Hamiltonian}
H(t)
=
\sum_{k=1}^{K}
\left(
aJ_z^{(k)}
+
b\bigl(J_z^{(k)}\bigr)^2
\right)
+
u(t)
\sum_{1\leq k<\ell\leq K}
J_x^{(k)}J_x^{(\ell)}.
\end{equation}
Here,
\begin{equation}
\label{eq:embedded-angular-momentum}
J_w^{(k)}
=
I_d^{\otimes(k-1)}
\otimes J_w
\otimes I_d^{\otimes(K-k)},
\qquad
w\in\{x,z\},
\end{equation}
where \(I_d\) denotes the \(d\times d\) identity matrix and
\(J_x,J_z\in\mathbb{R}^{d\times d}\) are angular-momentum matrices.
The Lindblad operators were chosen as
\[
L_k=J_z^{(k)},
\qquad
k=1,\ldots,K,
\]
with constant decay rates $\gamma_k(t)\equiv\gamma$. 

\subsection{Verification of the full- and low-rank propagators}
\label{subsec:propagator-experiments}

The first group of experiments verifies the convergence estimates and
the structural properties established for the exponential midpoint
propagators. For the full-rank schemes, the terminal errors are measured
by
\begin{equation}
\label{eq:full-rank-numerical-errors}
e_\rho
=
\|\rho_N-\rho(T)\|_1,
\qquad
e_q
=
\|q_0-q(0)\|_1.
\end{equation}
For the low-rank schemes, the corresponding errors are
\begin{equation}
\label{eq:low-rank-numerical-errors}
\check e_\rho
=
\|X_NX_N^\dagger-\rho(T)\|_1,
\qquad
\check e_q
=
\|Y_0S_0Y_0^\dagger-q(0)\|_1.
\end{equation}

Reference solutions for \(\rho(T)\) and \(q(0)\) were computed using
the solver \emph{mesolve} from QuTiP \citep{Johansson}. Experimental
orders of convergence were obtained from successive refinements of the
uniform time grid.

The full-rank midpoint propagators were first assessed for the forward
and adjoint Lindblad equations. The initial and terminal conditions were
chosen as
\begin{subequations}\label{equ6.1}
\begin{align}
\rho(0)
&=
\frac12\Bigl(
|0\rangle^{\otimes K}\langle0|^{\otimes K}
+
|0\rangle^{\otimes K}\langle d-1|^{\otimes K}
+
|d-1\rangle^{\otimes K}\langle0|^{\otimes K}
+
|d-1\rangle^{\otimes K}\langle d-1|^{\otimes K}
\Bigr),\\
q(T)
&=
\frac12\Bigl(
|1\rangle^{\otimes K}\langle1|^{\otimes K}
+
|1\rangle^{\otimes K}\langle d-2|^{\otimes K}
+
|d-2\rangle^{\otimes K}\langle1|^{\otimes K}
-
|d-2\rangle^{\otimes K}\langle d-2|^{\otimes K}
\Bigr).
\end{align}
\end{subequations}
The matrix exponential routines employed in the FREM propagators were
used with the default machine-precision tolerance.

Figure~\ref{fig6.1} reports the forward and adjoint terminal errors
together with the corresponding experimental orders of convergence.
Second-order convergence was observed for both propagators, in
agreement with Theorems~\ref{thm3.2} and \ref{thm4.1}. The preservation
of the physical structure is illustrated in
Figure~\ref{fig6.2}, where the evolution of representative populations
and the trace error confirm that the forward propagator preserves
positivity and unit trace.

\begin{figure}[h!]
{ \centering
\includegraphics[width=0.48\textwidth]{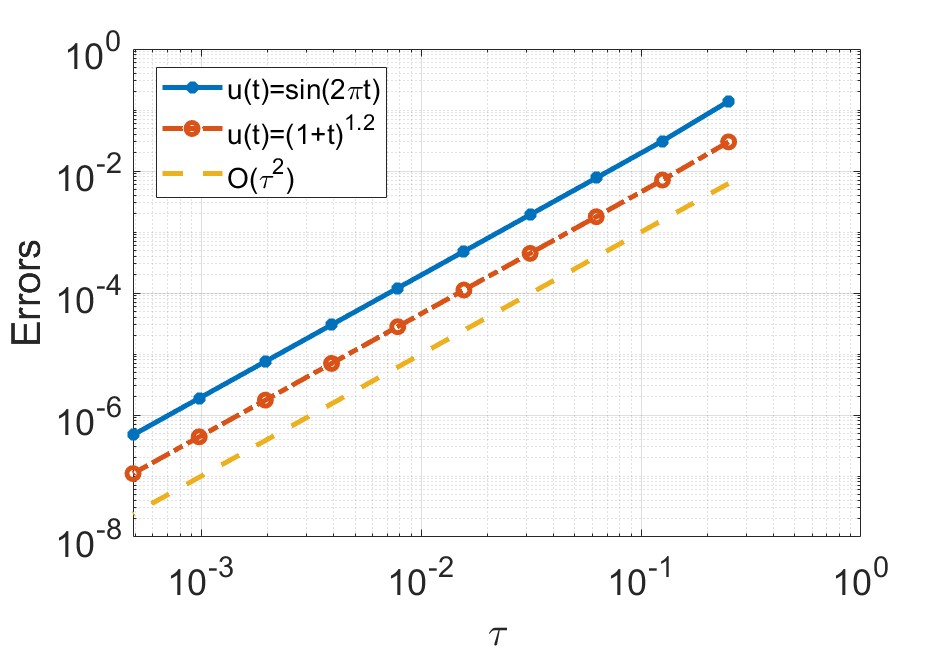}
}  { \centering
\includegraphics[width=0.48\textwidth]{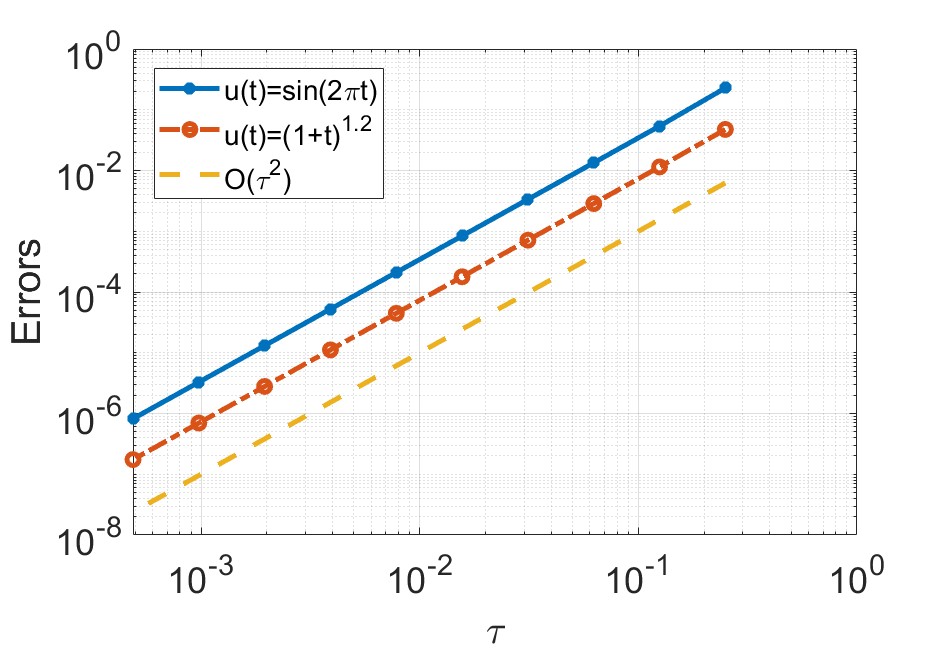}
} \caption{Numerical results of FREM schemes for the Lindblad equations with $d=6$, $K=2$, $a=1.5$, $b=1$, 
$\gamma=0.05$, $T=1$. Left: errors vs step sizes for the forward Lindblad equation. 
Right: errors vs step sizes for the adjoint Lindblad equation.}
\label{fig6.1}
\end{figure}

\begin{figure}[h!]
{ \centering
\includegraphics[width=0.48\textwidth]{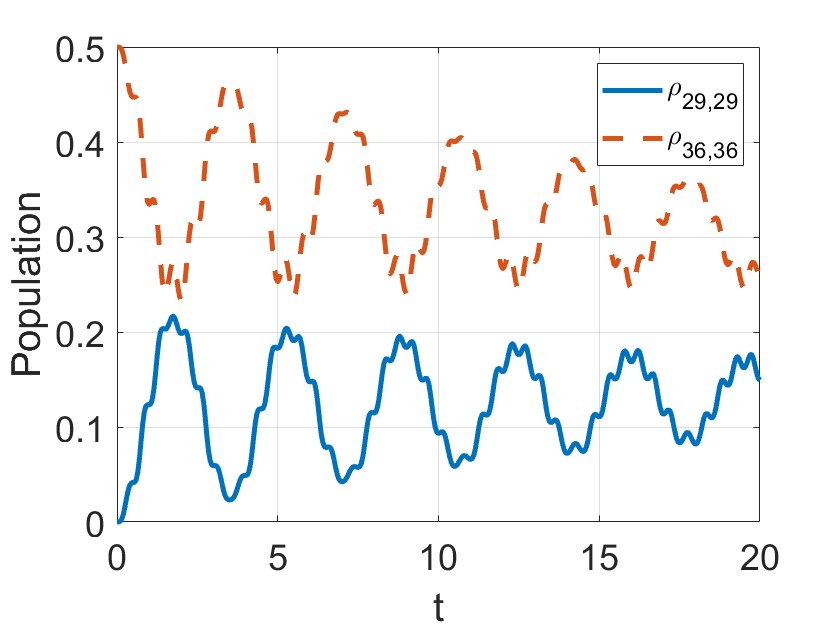}
}  { \centering
\includegraphics[width=0.48\textwidth]{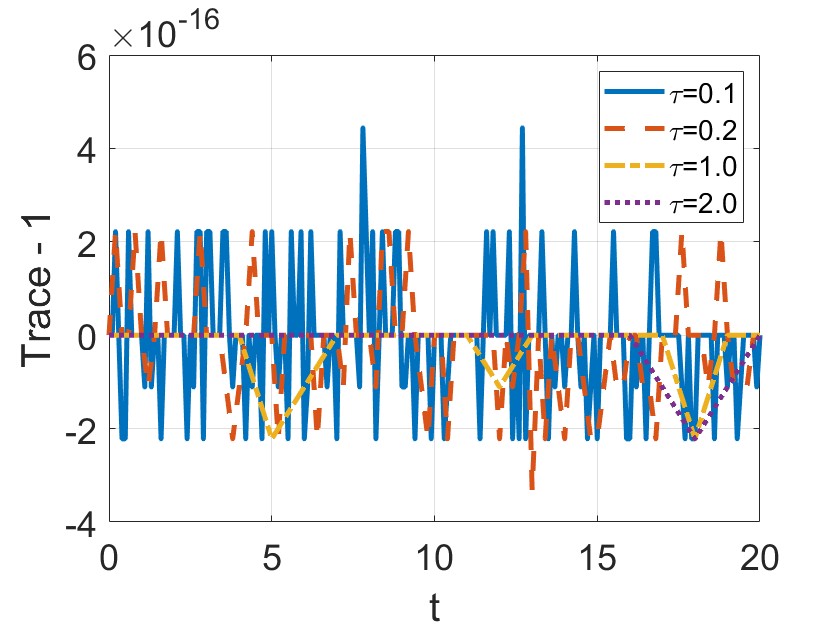}
} \caption{Numerical results of FREM scheme for the forward Lindblad equation with $d=6$, $K=2$, $a=1.5$, $b=1$, $\gamma=0.05$, $T=20$, $u(t)=\sin(2\pi t)$. Left: evolutions of the populations $\rho_{29,29}$ and $\rho_{36,36}$ with $\tau=0.05$. Right: evolution of $\Tr(\rho_n)-1$.}
\label{fig6.2}
\end{figure}

The low-rank midpoint propagators were then examined. The initial and
terminal data were chosen as
\[
\rho(0)
=
\left(1-\frac{\delta}{2}\right)z_1z_1^\top
+
\frac{\delta}{2}z_2z_2^\top,
\qquad
q(T)
=
\left(1-\frac{\delta}{2}\right)z_3z_3^\top
+
\frac{\delta}{2}z_4z_4^\top,
\]
where \(z_1,z_2,z_3,z_4\) are orthonormal vectors obtained from the
singular value decomposition of a random matrix. The corresponding
initial low-rank factors were chosen as
\(X_0=z_1\) and \(Y_N=z_3, S_N=1\), so that the initial approximation error is
prescribed by the parameter \(\delta\).

Figures~\ref{fig6.3}-\ref{fig6.5} investigate the influence of the
initial low-rank approximation error, the column-compression
tolerance, and the matrix-exponential tolerance, respectively.
Second-order convergence was recovered whenever these additional error
sources remained below the temporal discretization error, whereas the
overall accuracy became limited once the low-rank approximation or the
compression tolerances became dominant. These observations are fully
consistent with the error estimates established in
Theorems~\ref{thm3.3} and \ref{thm4.3}. The preservation of positivity
and trace by the low-rank forward propagator is confirmed
in Figure~\ref{fig6.6}.

\begin{figure}[h!]
{ \centering
\includegraphics[width=0.48\textwidth]{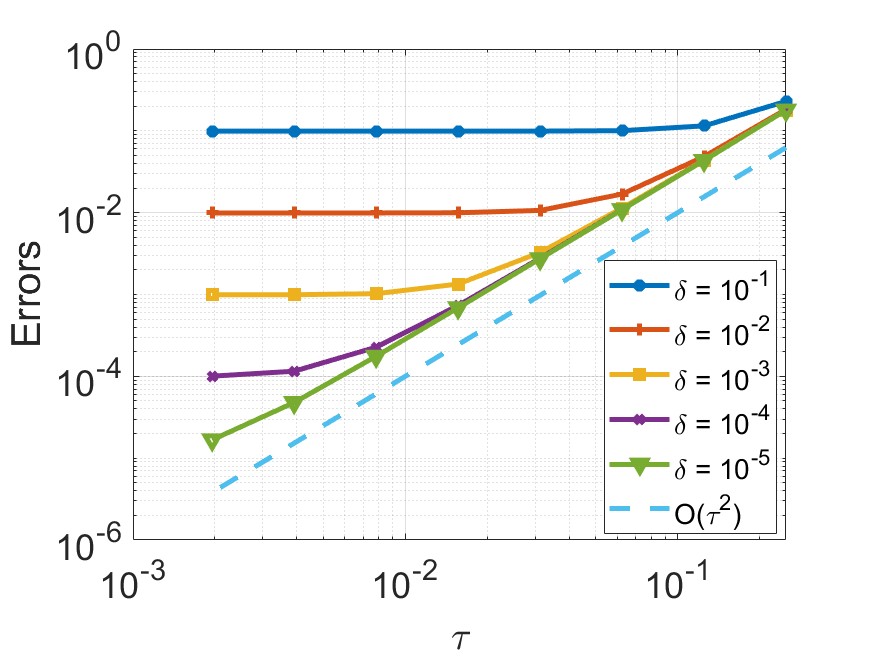}
}  { \centering
\includegraphics[width=0.48\textwidth]{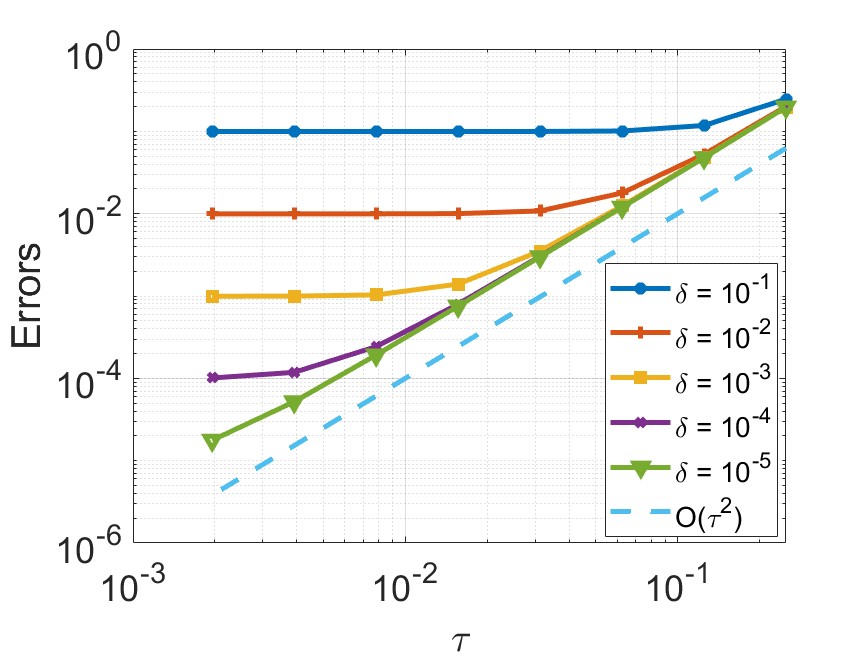}
} \caption{Numerical results of LREM schemes with fixed $\varepsilon_1=\varepsilon_2=10^{-10}$ and 
different $\delta$ for the Lindblad equations with $d=4$, $K=4$, $a=1.5$, $b=1$, $\gamma=0.05$, $T=1$, $u(t)=\sin(2\pi t)$. 
Left: errors vs step sizes for the forward Lindblad equation. Right: errors vs step sizes for the adjoint Lindblad equation.}
\label{fig6.3}
\end{figure}

\begin{figure}[h!]
{ \centering
\includegraphics[width=0.48\textwidth]{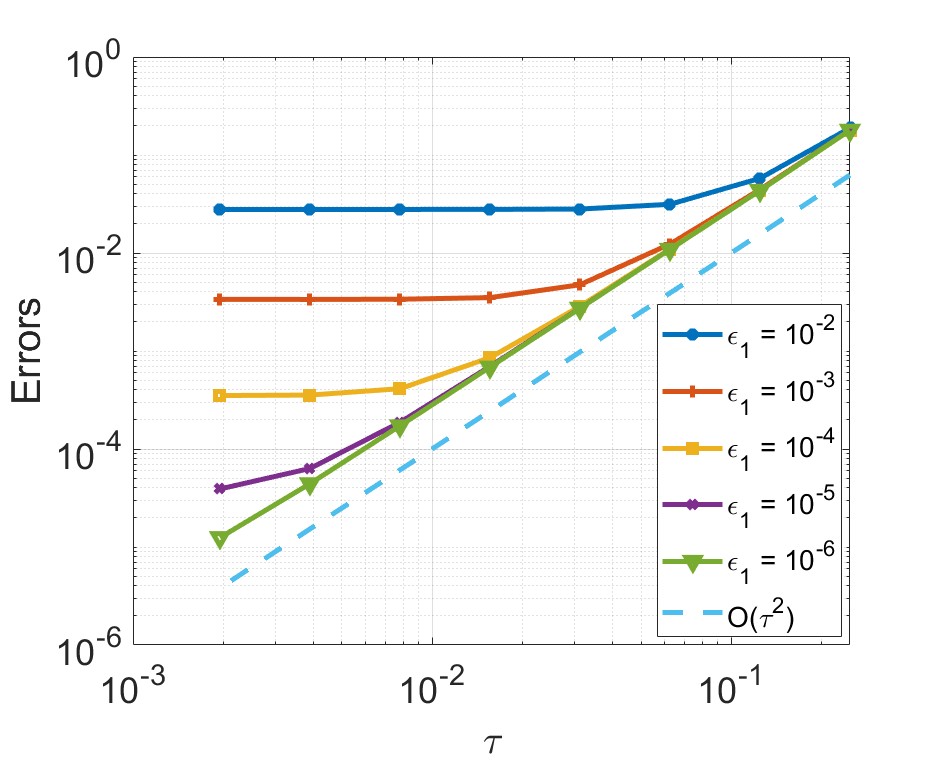}
}  { \centering
\includegraphics[width=0.48\textwidth]{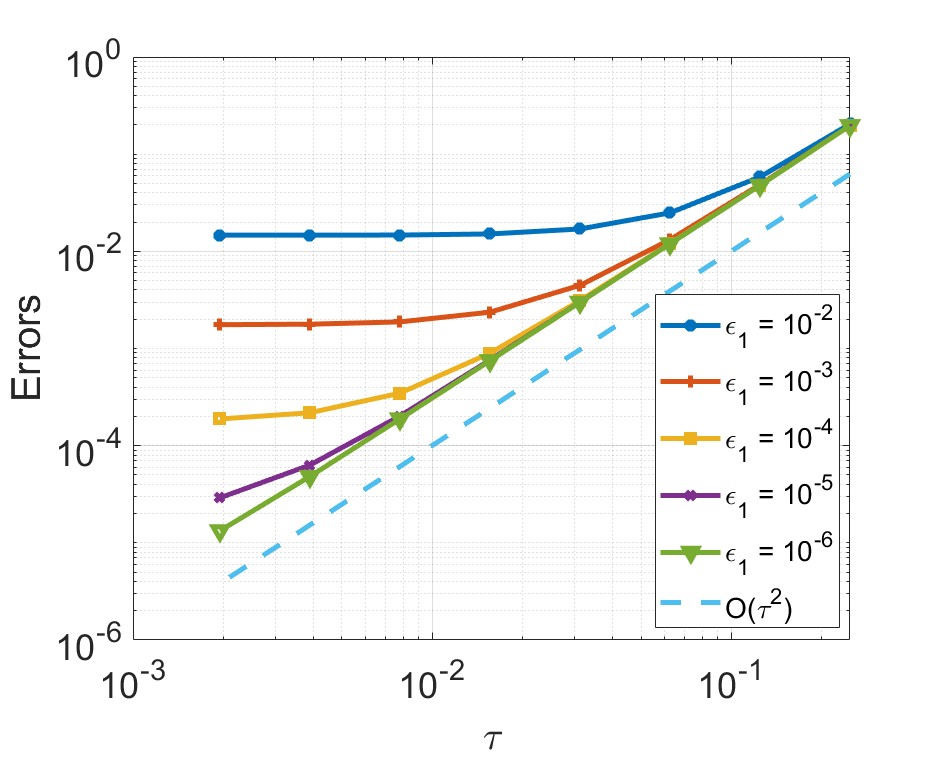}
} \caption{Numerical results of LREM schemes with fixed $\delta=\varepsilon_2=10^{-10}$ and 
different $\varepsilon_1=\tau\epsilon_1$ for the Lindblad equations with $d=4$, $K=4$, $a=1.5$, $b=1$, 
$\gamma=0.05$, $T=1$, $u(t)=\sin(2\pi t)$. Left: errors vs step sizes for the forward Lindblad equation. 
Right: errors vs step sizes for the adjoint Lindblad equation.}
\label{fig6.4}
\end{figure}

\begin{figure}[h!]
{ \centering
\includegraphics[width=0.48\textwidth]{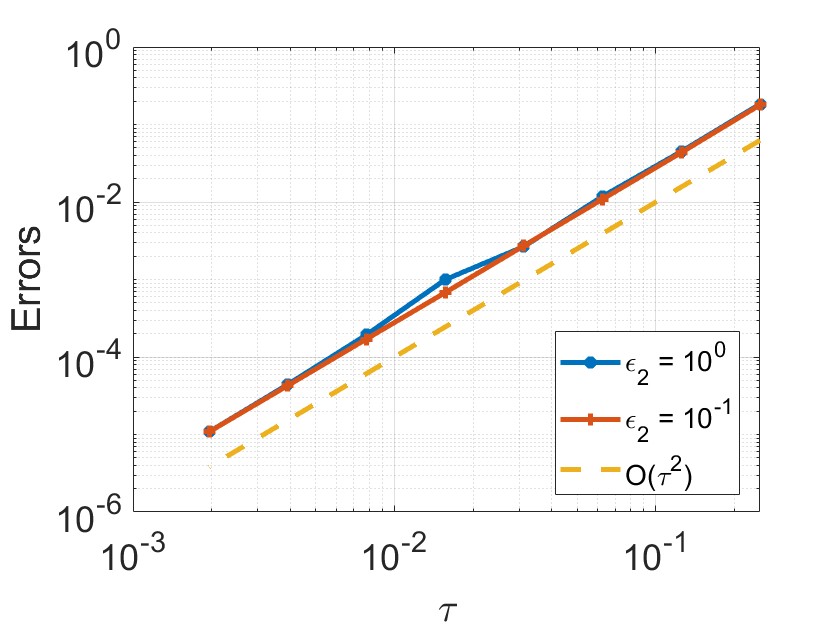}
}  { \centering
\includegraphics[width=0.48\textwidth]{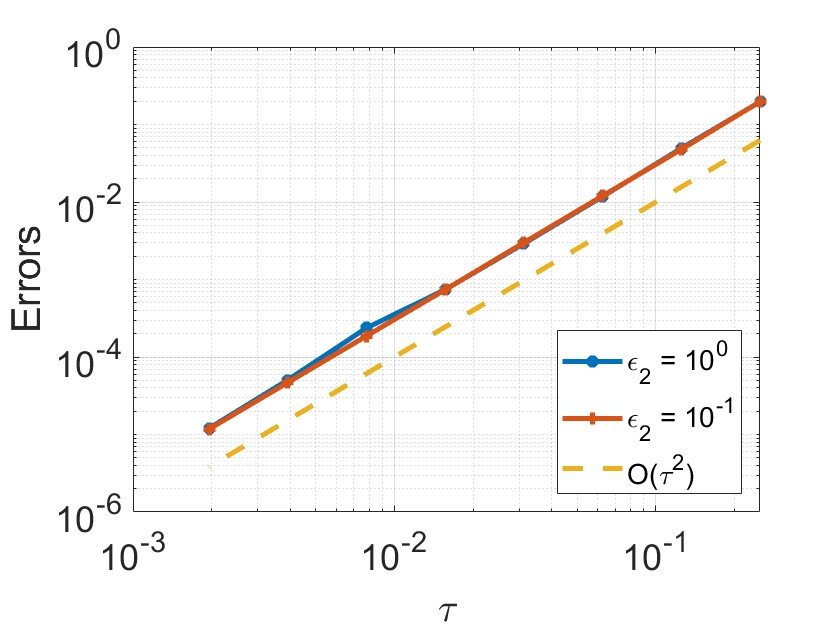}
} \caption{Numerical results of LREM schemes with fixed $\delta=\varepsilon_1=10^{-10}$ and 
different $\varepsilon_2=\tau\epsilon_2$ for the Lindblad equations with $d=4$, $K=4$, $a=1.5$, $b=1$, 
$\gamma=0.05$, $T=1$, $u(t)=\sin(2\pi t)$. Left: errors vs step sizes for the forward Lindblad equation. 
Right: errors vs step sizes for the adjoint Lindblad equation.}
\label{fig6.5}
\end{figure}

\begin{figure}[h!]
{ \centering
\includegraphics[width=0.48\textwidth]{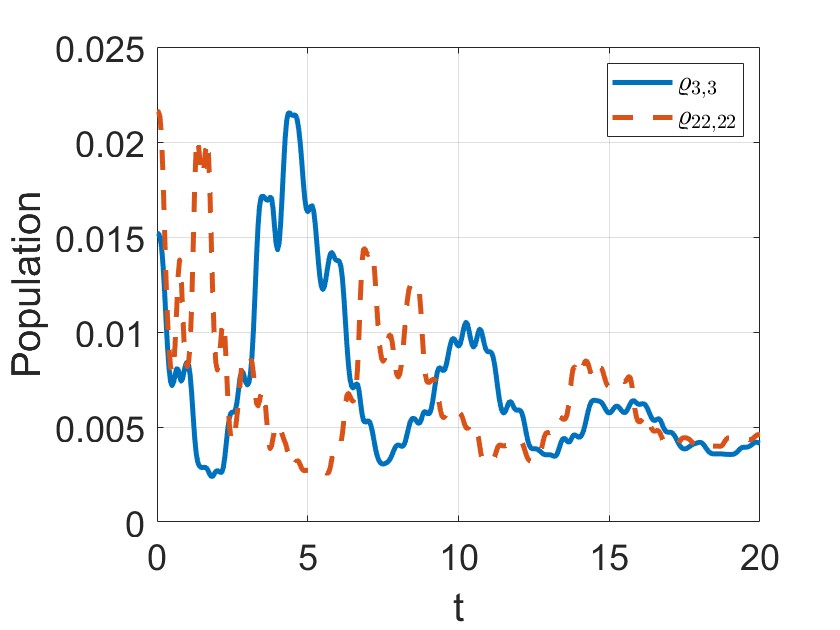}
}  { \centering
\includegraphics[width=0.48\textwidth]{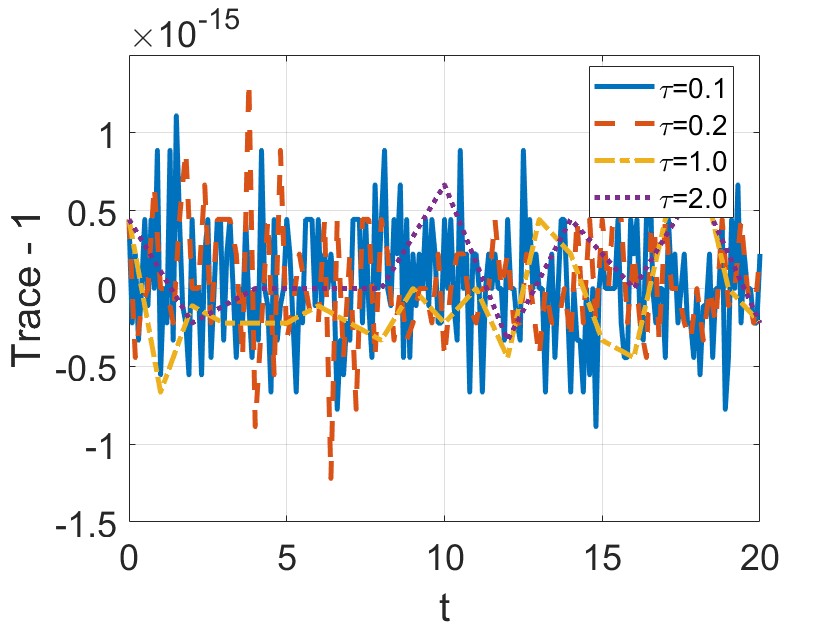}
} \caption{Numerical results of LREM scheme for forward Lindblad equation with $d=4$, $K=4$, $a=1.5$, $b=1$, $\gamma=0.05$, 
$T=20$, $u(t)=\sin(2\pi t)$. Left: evolutions of the populations $\varrho_{3,3}$ and $\varrho_{22,22}$ with $\tau=0.05$. 
Right: evolution of $\Tr(\varrho_n)-1$.}
\label{fig6.6}
\end{figure}

Finally, the proposed propagators were compared with the QuTiP solver
\emph{mesolve} \cite{Johansson}, using the Dormand-Prince
\emph{dop853} integrator applied to the vectorized Lindblad equation.
Figures~\ref{fig6.7} and \ref{fig6.8} summarize the resulting
computational times and approximation errors for the forward and
adjoint equations. For comparable levels of accuracy, the proposed
low-rank propagators substantially reduced the computational cost,
particularly as the dimension of the underlying Hilbert space
increased.

\begin{figure}[h!]
{ \centering
\includegraphics[width=0.48\textwidth]{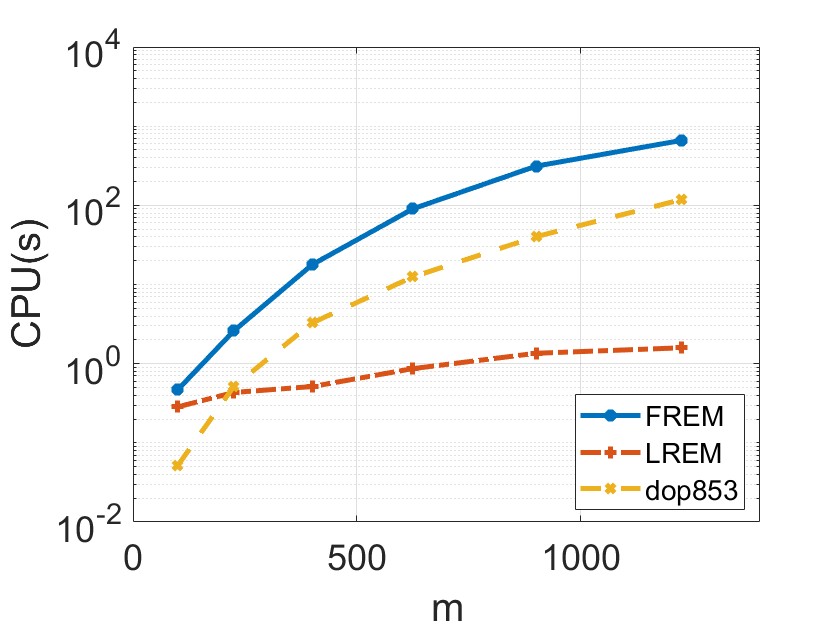}
} { \centering
\includegraphics[width=0.48\textwidth]{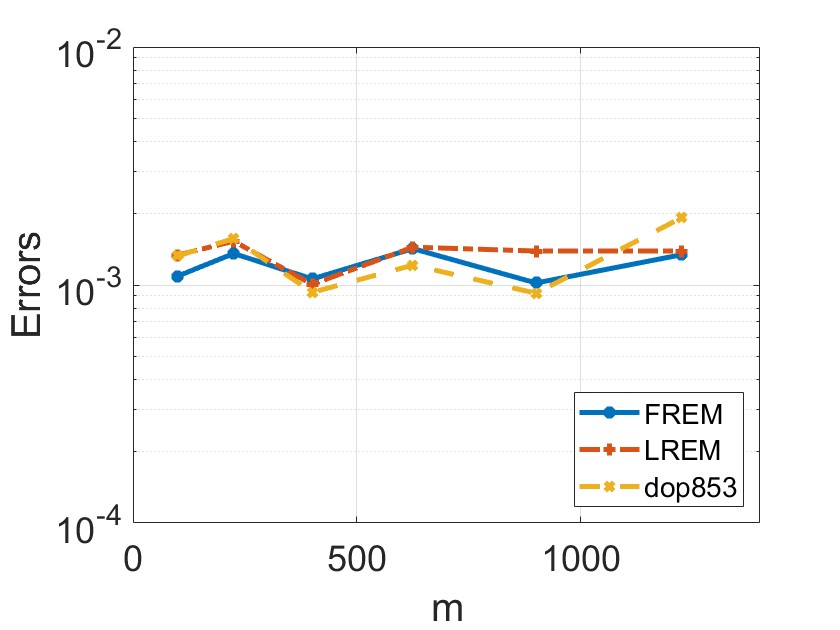}
} \caption{Numerical comparison between the proposed exponential schemes and the QuTip solver 
for the forward Lindblad equation with $K=2$, $a=1.5$, $b=1$, $\gamma=0.05$, $T=1$, $u(t)=\sin(2\pi t)$. Left: CPU times vs $m$. Right: errors vs $m$.
}
\label{fig6.7}
\end{figure}

\begin{figure}[h!]
{ \centering
\includegraphics[width=0.48\textwidth]{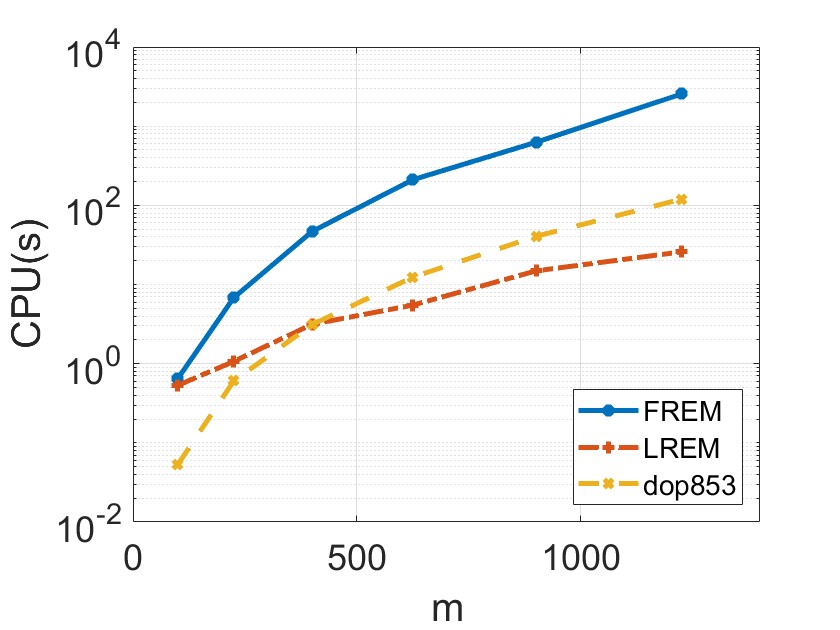}
} { \centering
\includegraphics[width=0.48\textwidth]{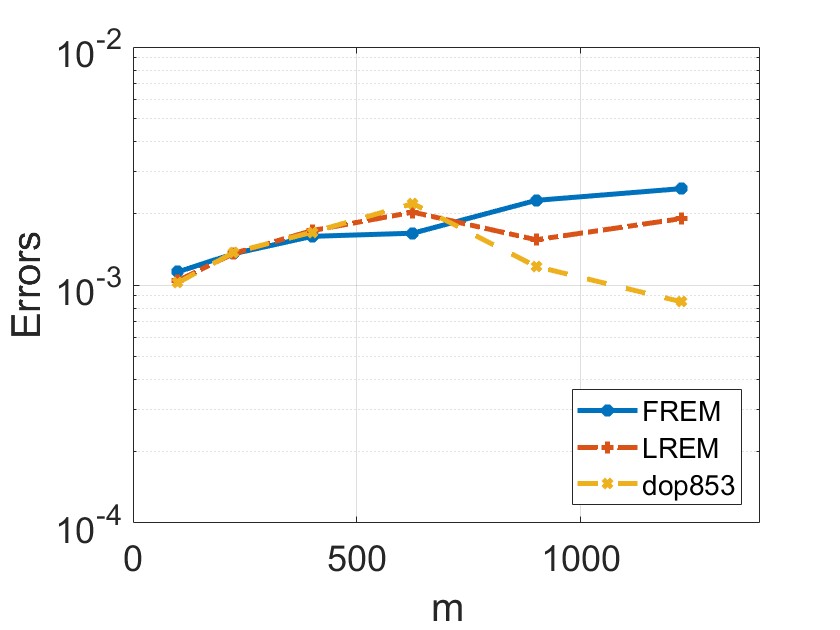}
} \caption{Numerical comparison between the proposed exponential schemes and the QuTip solver for the adjoint
Lindblad equation with $K=2$, $a=1.5$, $b=1$, $\gamma=0.05$, $T=1$, $u(t)=\sin(2\pi t)$. Left: CPU times vs $m$. Right: errors vs $m$.
}
\label{fig6.8}
\end{figure}

The preceding experiments validate the convergence theory and the
structure-preserving properties of the proposed propagators, thereby
providing the numerical foundation for their use within the optimal
control framework developed in the following subsection.

\subsection{Numerical results for optimal control}
\label{subsec:optimal-control-experiments}

The preceding experiments validated the convergence and
structure-preserving properties of the proposed propagators. Their
effectiveness is now assessed within the SQH framework for open quantum
optimal control problems. Throughout this subsection, the forward and
adjoint Lindblad equations are solved by the proposed full-rank
exponential midpoint (FREM) and low-rank exponential midpoint (LREM)
propagators, leading to the SQH-FREM and SQH-LREM algorithms,
respectively. For comparison, numerical results obtained with the GRAPE
algorithm \citep{Schulte} are also reported whenever applicable. 

The uncontrolled Hamiltonian and the control Hamiltonian are chosen as
\[
H_0
=
\sum_{k=1}^{K}
\left(
aJ_z^{(k)}
+
b(J_z^{(k)})^2
\right),
\qquad
H_1
=
\sum_{k<\ell}
J_x^{(k)}J_x^{(\ell)},
\]
with Lindblad operators
\[
L_k=J_z^{(k)},
\qquad
k=1,\ldots,K.
\]
Unless otherwise stated, the SQH parameters are
\[
\zeta=0.9,\qquad
\sigma=1.1,\qquad
\eta=10^{-5},\qquad
\kappa=10^{-4},
\]
with initial value \(\epsilon=1\) and a randomly generated initial
control.

\paragraph{Example 1 (smooth control).}

The first experiment considers the smooth running cost
\begin{equation*}
J(u) = \frac12 \|\rho(T)-\rho_T\|_F^2 + \frac{\alpha}{2} \int_0^T u(t)^2\,dt
\end{equation*}
with \(U=\mathbb{R}\). Two coupled four-level qudits are
considered, corresponding to \(K=2\), \(d=4\), and
\(m=d^K=16\). The remaining parameters are
\[
T=2,\qquad
a=1.5,\qquad
b=1,\qquad
\alpha=10^{-3},\qquad
\gamma_k=0.05,
\]
while the initial and target states are
\[
\rho_0=|00\rangle\langle00|,
\qquad
\rho_T=|11\rangle\langle11|.
\]

Figures~\ref{fig1} and \ref{fig2} compare SQH-FREM, SQH-LREM, and
GRAPE for this problem. Both SQH implementations produced comparable
controls and population transfers while exhibiting monotone decrease of
the objective functional. For the tested discretizations and system
dimensions, SQH-FREM and SQH-LREM required less computational time
than GRAPE, whereas the low-rank implementation provided an additional
reduction in CPU time compared with the full-rank realization.

\begin{figure}[h!]
{ \centering
\includegraphics[width=0.315\textwidth]{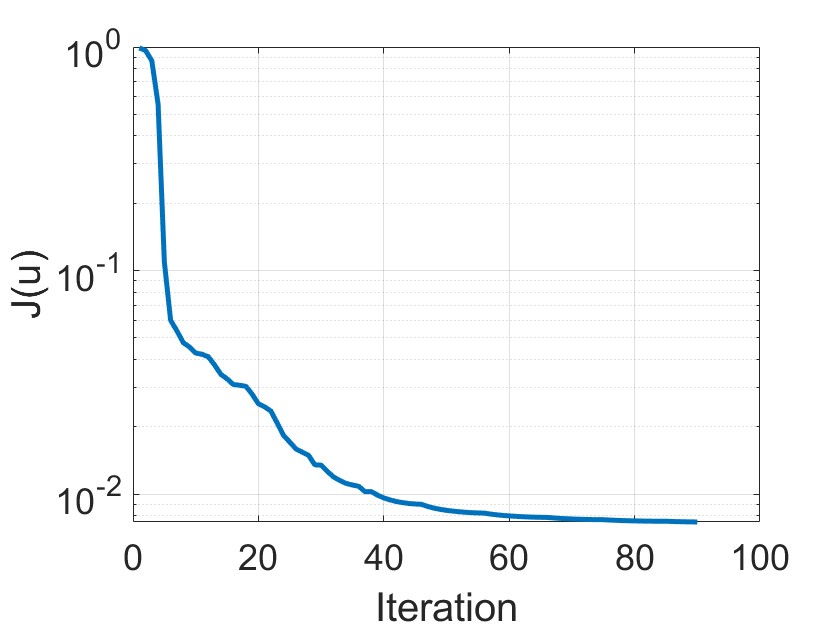}
} { \centering
\includegraphics[width=0.315\textwidth]{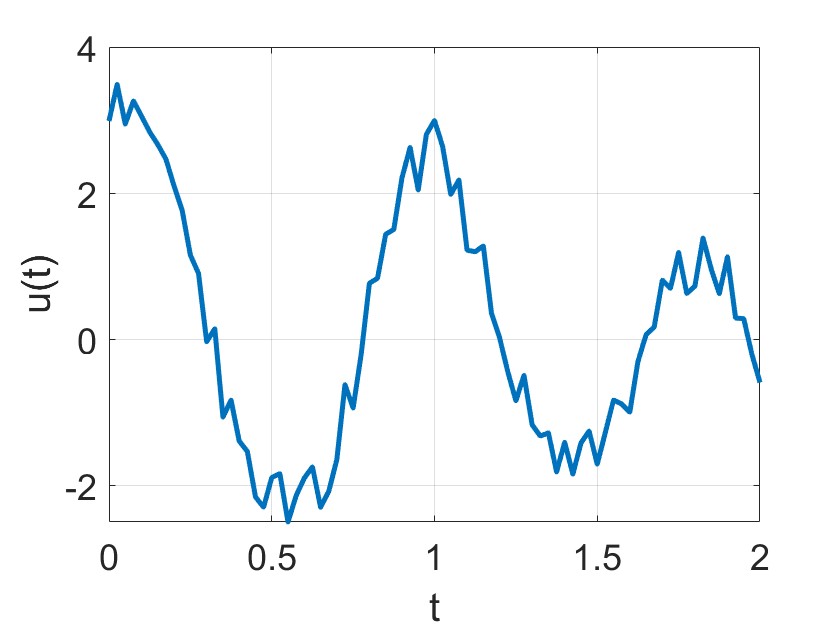}
} { \centering
\includegraphics[width=0.315\textwidth]{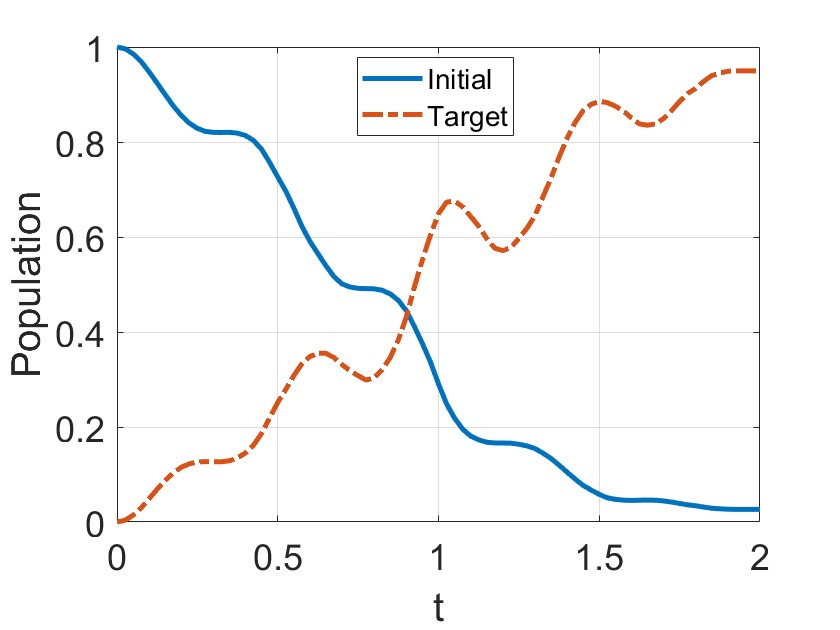}
} \\{ \centering
\includegraphics[width=0.315\textwidth]{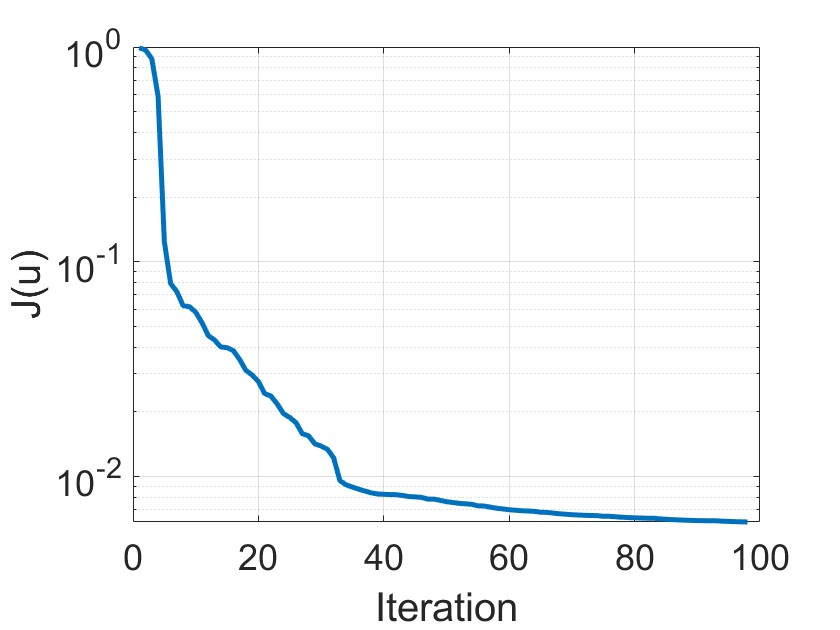}
} { \centering
\includegraphics[width=0.315\textwidth]{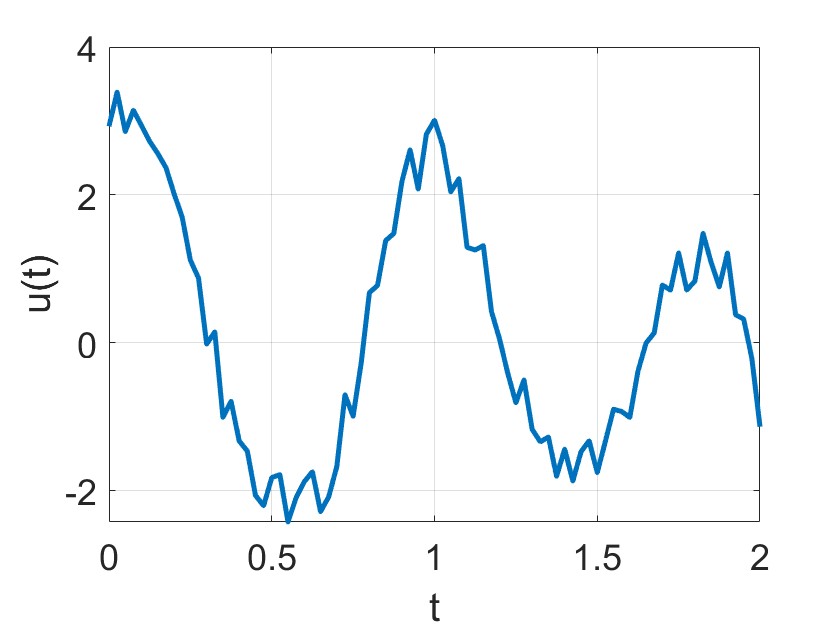}
} { \centering
\includegraphics[width=0.315\textwidth]{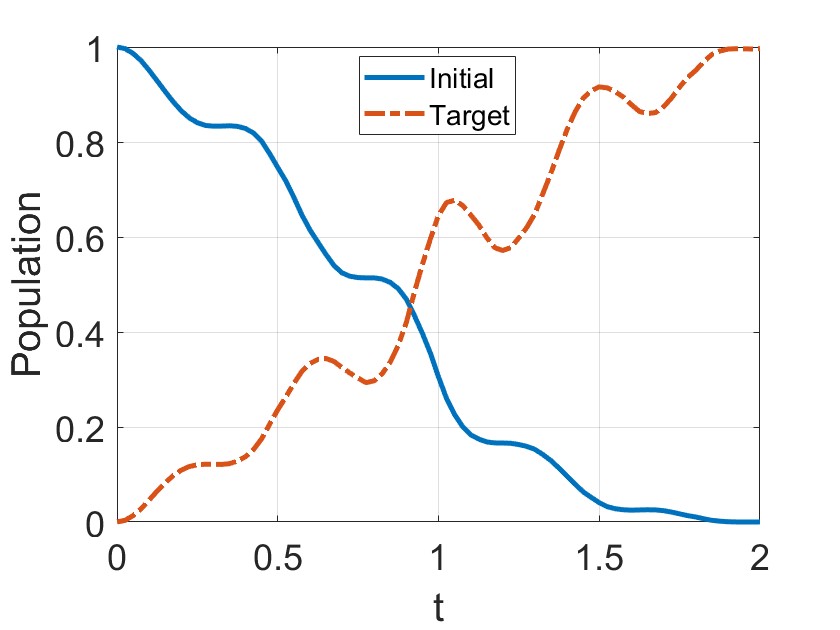}
} \\{ \centering
\includegraphics[width=0.315\textwidth]{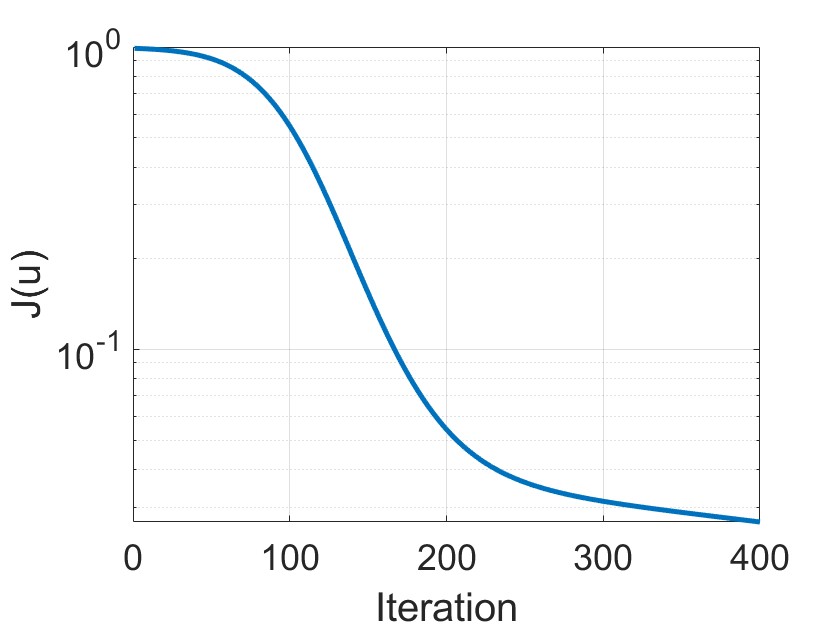}
} { \centering
\includegraphics[width=0.315\textwidth]{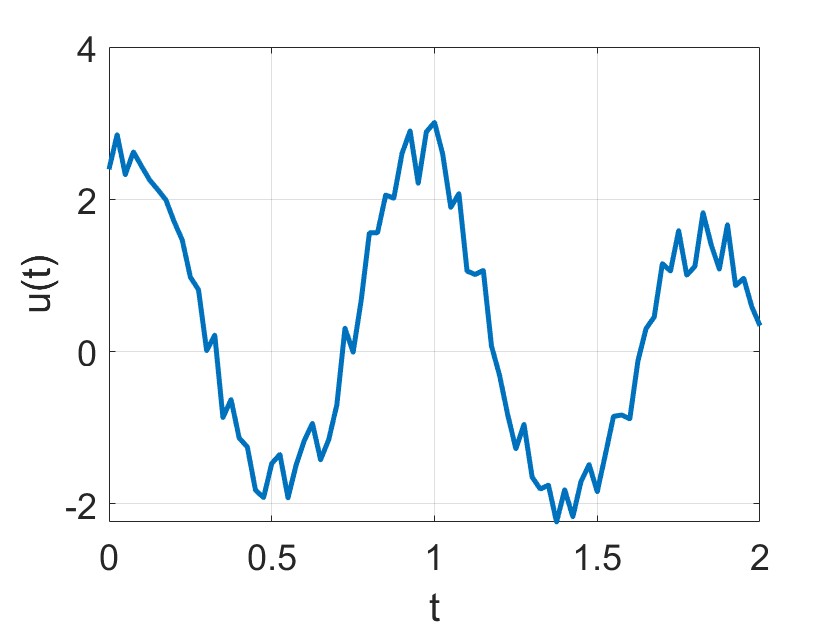}
} { \centering
\includegraphics[width=0.315\textwidth]{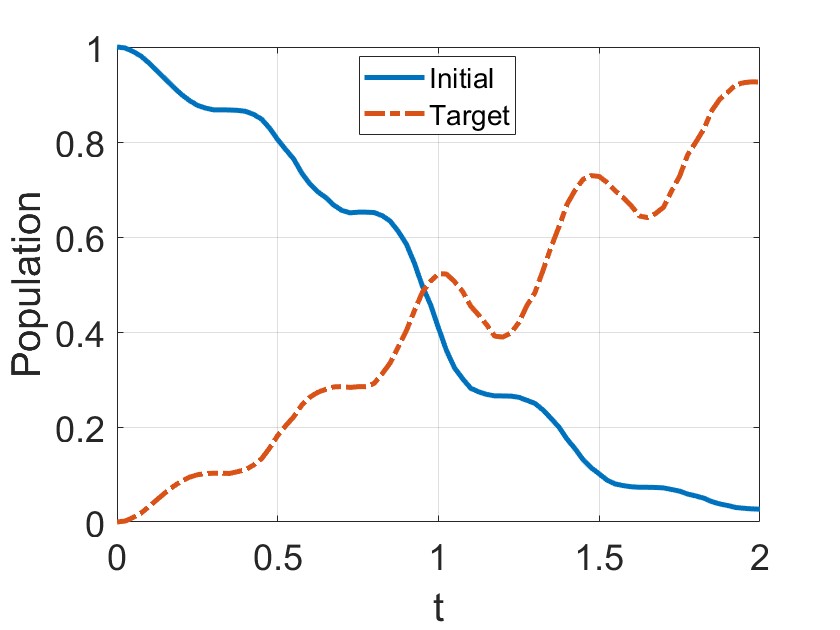}
} 
\caption{Smooth control problem. From left to right: objective history, optimized control,
and population transfer. From top to bottom: SQH-FREM, SQH-LREM, and GRAPE
(\(\tau=0.025\)).}
\label{fig1}
\end{figure}

\begin{figure}[h!]
{ \centering
\includegraphics[width=0.48\textwidth]{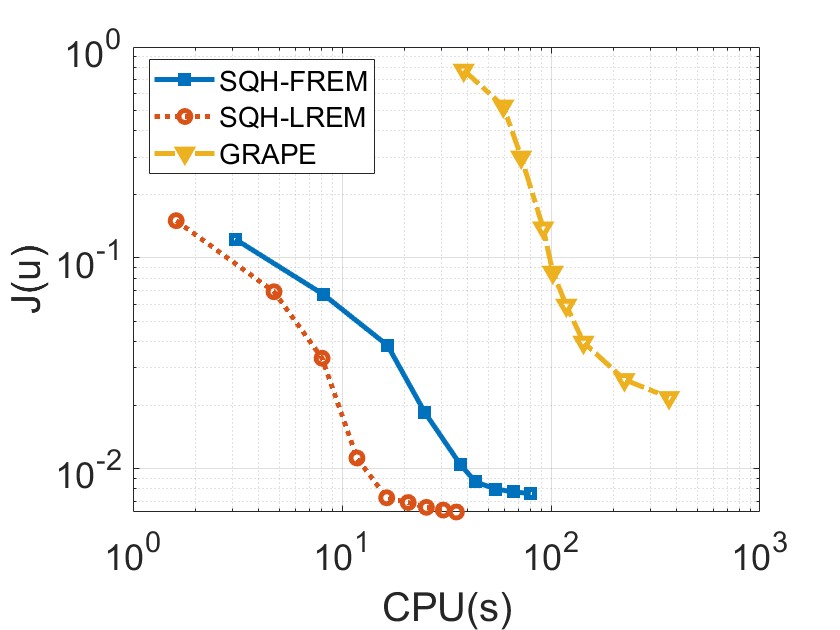}
} { \centering
\includegraphics[width=0.48\textwidth]{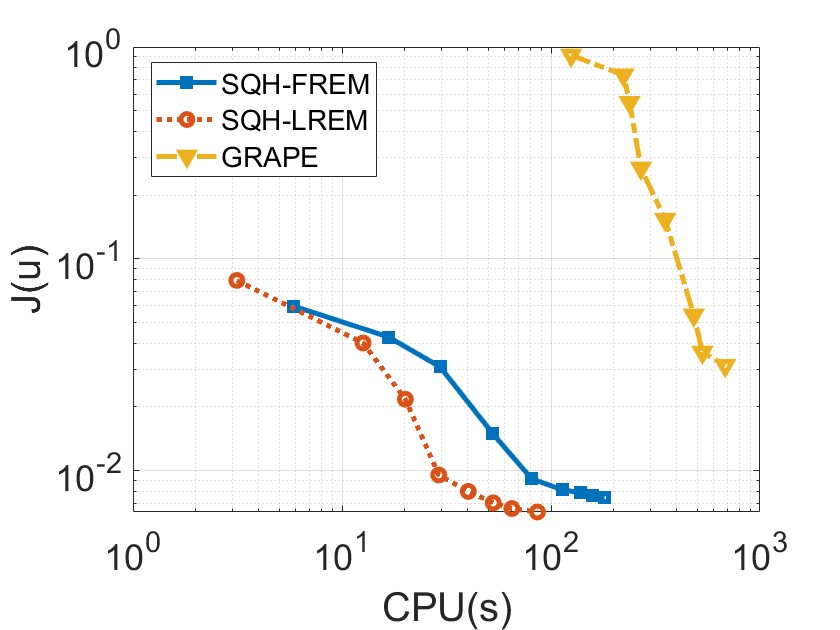}
} \caption{Numerical comparison of SQH and GRAPE methods for Example 1. From left to right: $\tau=0.05$ and $\tau=0.025$.
}
\label{fig2}
\end{figure}

\paragraph{Example 2 (nonsmooth and nonconvex control).}

The second experiment considers the discontinuous running cost
\begin{equation*}
J(u) =-\operatorname{Tr}(Q\rho(T)) + \int_0^T \ell(u(t))\,dt ,
\end{equation*}
where
\[
\ell(u)=
\begin{cases}
\dfrac{\alpha}{2}u^2+\beta|u|,
&
|u|>s,
\\[2mm]
\dfrac{\alpha}{2}u^2,
&
|u|\le s,
\end{cases}
\]
 is both nonsmooth and nonconvex. The physical model is the same
as in Example~1, with
\[
T=2,\qquad
a=1.5,\qquad
b=1,\qquad
\gamma_k=0.05,
\]
and
\[
\rho_0=|00\rangle\langle00|,
\qquad
Q=|11\rangle\langle11|.
\]

Since GRAPE relies on differentiability of the running cost, only the
SQH algorithms are considered.

Figures~\ref{fig3} and \ref{fig4} report two representative parameter
regimes. In both cases, the SQH iteration generated monotone decrease
of the objective functional together with controls that successfully
steered the population toward the prescribed target state. These
experiments demonstrate that the proposed computational framework
remains effective for nonsmooth and nonconvex optimal control problems,
for which conventional gradient-based approaches are generally not
applicable.

\begin{figure}[h!]
{ \centering
\includegraphics[width=0.315\textwidth]{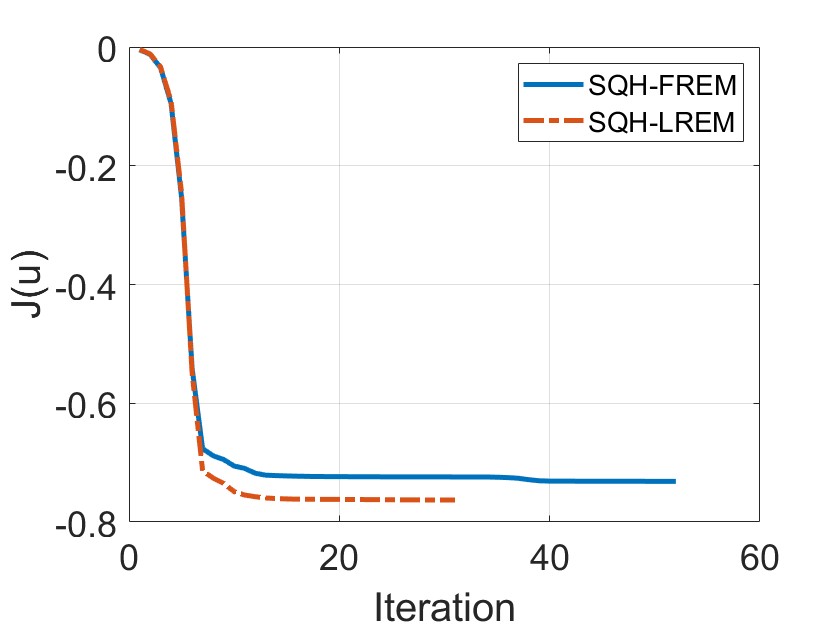}
} { \centering
\includegraphics[width=0.315\textwidth]{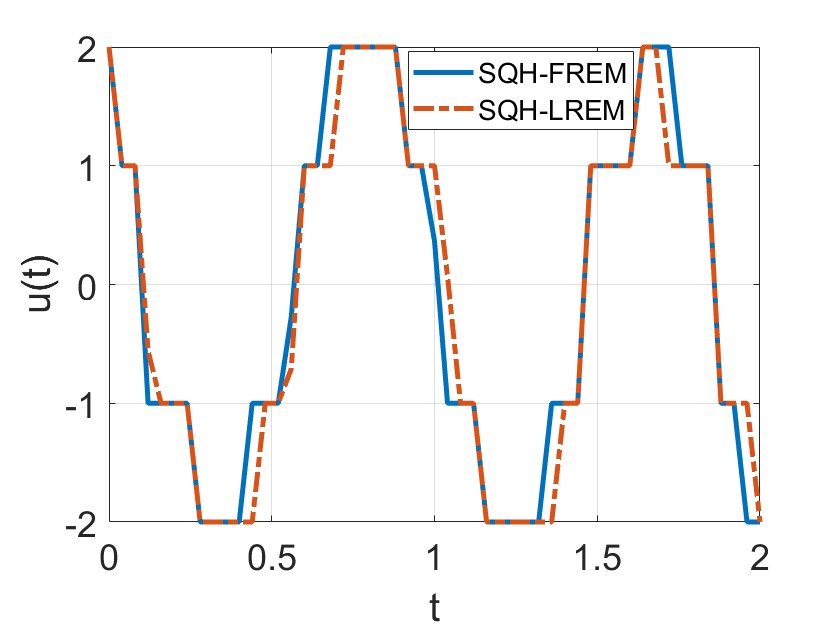}
} { \centering
\includegraphics[width=0.315\textwidth]{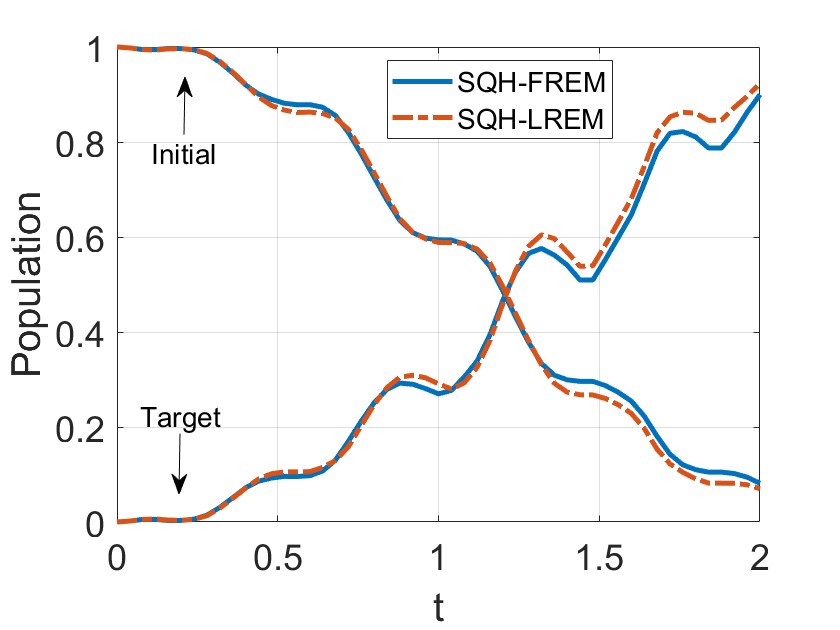}
} \caption{Numerical results of SQH methods for Example 2 with $\alpha=10^{-4}$, $\beta=10^{-1}$, $s=1$, $\tau=0.04$ and $U=[-2,2]$. From left to right: cost functional, optimized control, population transfer.
}
\label{fig3}
\end{figure}

\begin{figure}[h!]
{ \centering
\includegraphics[width=0.315\textwidth]{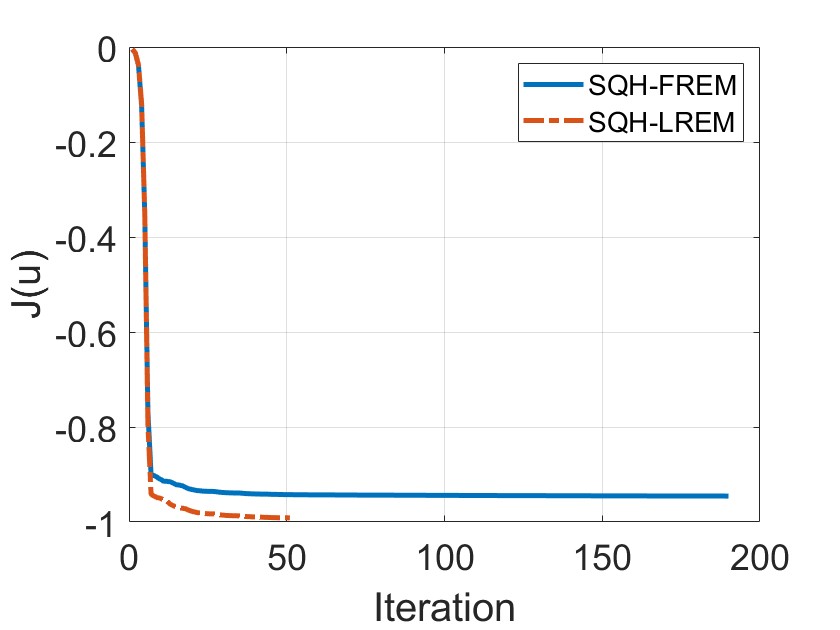}
} { \centering
\includegraphics[width=0.315\textwidth]{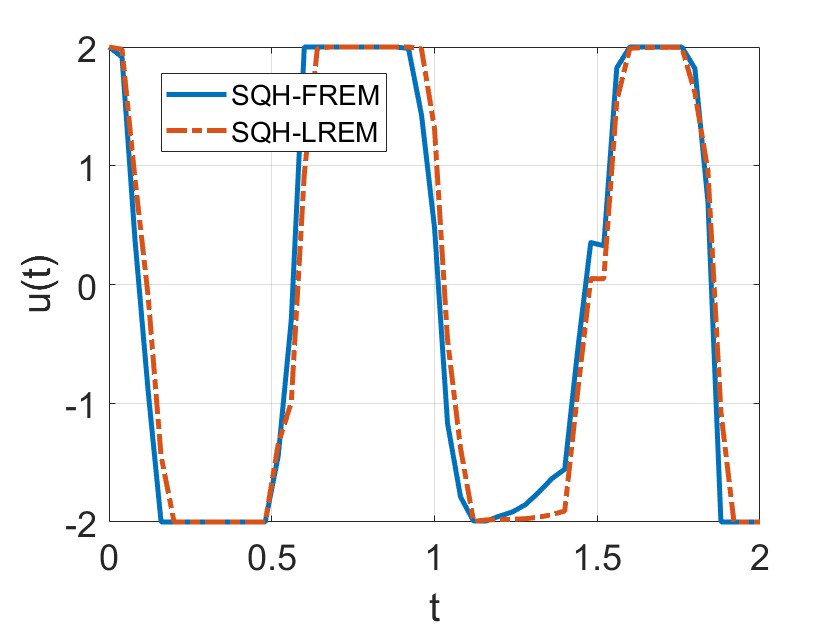}
} { \centering
\includegraphics[width=0.315\textwidth]{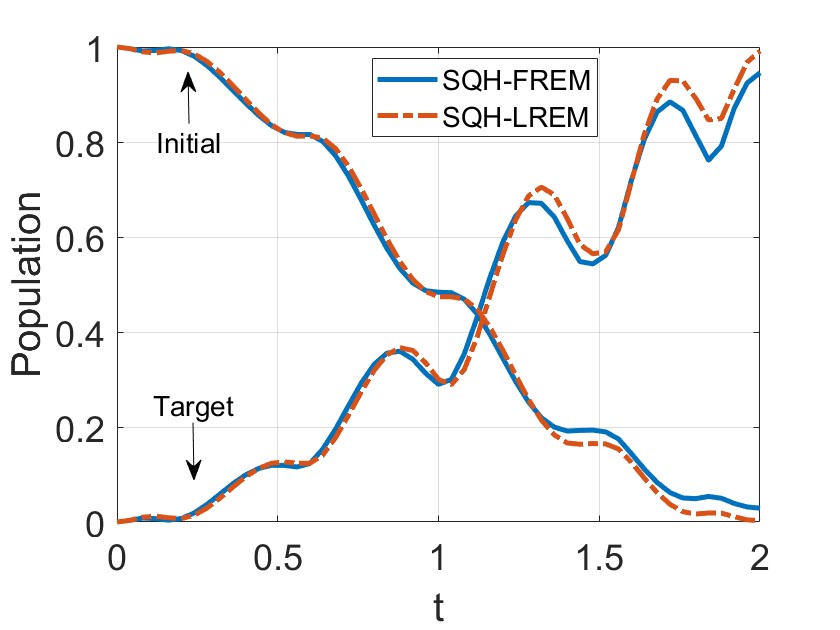}
} \caption{
Numerical results of SQH methods for Example 2 with $\alpha=10^{-4}$, $\beta=1$, $s=2$, $\tau=0.04$ and $U=[-20,20]$. From left to right: cost functional, optimized control, population transfer.
}
\label{fig4}
\end{figure}

\FloatBarrier

\section{Conclusions}\label{sec-conclusions}

A unified computational framework for the optimal control of open
quantum systems governed by Lindblad equations was presented. The
continuous optimality system was formulated by means of the Pontryagin
maximum principle and was solved by a sequential quadratic Hamiltonian
method. Structure-preserving full- and low-rank exponential midpoint
methods were developed for the forward and adjoint Lindblad equations,
thereby providing efficient computational kernels for the optimization
process. Rigorous error estimates were established for the full- and
low-rank propagators, and the resulting algorithms were shown to
preserve the fundamental physical structure of the density matrix.
Their effectiveness was demonstrated by numerical experiments, which
confirmed the predicted convergence behavior and illustrated the
computational advantages of the proposed approach for both smooth and
nonsmooth optimal control problems.

The proposed framework was intended to provide a foundation for the
large-scale numerical solution of open quantum optimal control
problems. Several directions for further investigation were left open,
including adaptive time-stepping strategies, improved low-rank
compression techniques, extension to more general classes of open
quantum systems, and the development of parallel implementations for
high-dimensional applications. The integration of structure-preserving
numerical methods with advanced optimization techniques was expected
to provide an effective basis for the computational treatment of
increasingly complex quantum control problems.

\section*{Acknowledgements}
The work of H. C. was supported by the Natural Science Foundation Project of Chongqing (No. CSTB2025NSCQ-GPX1015)

\bibliographystyle{cas-model2-names}

\bibliography{lindblad-refs}



\end{document}